%% file: main.tex
\documentclass[a4paper, 11pt]{article}
\usepackage{standalone}
\input{preamble.tex}

\usepackage[margin=1in]{geometry}

\usepackage[backend=biber,style=numeric,  eprint=false, sorting=nyt,giveninits=true]{biblatex}
\AtEveryBibitem{%
  \ifentrytype{article}{%
    \clearfield{url}%
    \clearfield{urldate}%
  }{}
}
\AtEveryBibitem{%
  \ifentrytype{book}{%
    \clearfield{url}%
    \clearfield{urldate}%
  }{}
}

\title{Non-Kähler metrics on complex manifolds of LVMB type}
\author{Federico Thiella\footnote{Dipartimento di Matematica e Informatica, Università degli Studi di Firenze}}
\date{}

\begin{document}
\maketitle

\begin{abstract}
    LVMB manifolds are a class of non-Kähler compact complex manifolds with a remarkably rich geometry: in many cases they admit a holomorphic bundle structure over a compact toric manifold. In fact, such a bundle is determined by an algebro-combinatoric datum encapsulating a simplicial fan. This is reflected in a close relationship between the geometry of an LVMB manifold and that of its toric base space.
    
    Throughout this paper, we restrict to this subclass of LVMB manifolds. We provide a formula for the characteristic class of the bundle in terms of the original LVMB datum. Subsequently, this expression is employed to address the existence of Hermitian metrics satisfying ``special'' conditions. We consider balanced and SKT metrics, showing that the former are obstructed in most cases. Moreover, within the LVMB class, we construct a new example of a manifold admitting a balanced metric. Finally, restricting ourselves to the subclass of LVM manifolds, we characterize those admitting an SKT metric.
\end{abstract}

\section*{Introduction}

LVMB manifolds are a class of compact non-Kähler manifolds whose structure is controlled by an algebro-combinatoric datum. Some of them are very well-known examples in the literature: complex tori, Hopf manifolds and Calabi--Eckmann manifolds all belong to the LVMB class. They were first introduced by López de Medrano and Verjovsky \cite{LopezVerjovsky1997}, and then generalized by Meersseman \cite{Meersseman2000}. A further generalization was then provided by Bosio \cite{Bosio2001}. More references for LVMB manifolds and their geometry include \cite{MeerssemanVerjovsky2004,Cupit-FoutouZaffran2007,PanovUstinovsky2012,Tambour2012,Battisti2013,BattagliaZaffran2015,BattagliaZaffran2017,Ishida2019,KatzarkovLupercioMeerssemanVerjovsky2021,IshidaKrutowskiPanov2022,Madera2025,Panov2025}.

The algebro-combinatoric nature of LVMB manifolds makes them one of the few systematic constructions of non-Kähler manifolds in the landscape of complex manifolds. Although complex geometry aims to describe complex manifolds in their full generality, constructing and studying new complex manifolds plays a leading role in contemporary research. In fact, very few results are known to hold in full generality, while many others are proved just for restricted classes of complex manifolds. This gap motivates the study of new examples of complex manifolds: understanding their properties may either disprove a resilient conjecture in general, or extend the range of cases in which it holds.

This paper investigates the geometry of LVMB manifolds and, in particular, establishes conditions that ensure the existence of \emph{canonical} Hermitian metrics. The existence of these metrics is a central problem in complex geometry, aiming to extend the Uniformization Theorem to higher dimensional complex manifolds. Such generalization has a twofold nature. Indeed, a canonical Hermitian metric can be characterized either by a curvature property of some Hermitian connection (like being Einstein or with constant scalar curvature), or by a cohomological property of its fundamental form.

In this sense, the Kähler setting is very rigid. Indeed, as Kähler metrics are precisely the ones for which the Levi--Civita connection is Hermitian, this connection is the most natural choice in the Kähler setting. The problem of classifying space forms---i.e. complete manifolds with constant curvature---is solved even in the Riemannian setting; similarly, complete Kähler manifolds with constant holomorphic sectional curvature have just three models: the projective space, the complex space, and the open unit ball \cite[\S IX.7]{KobayashiNomizu1969}. Thus, finding canonical metrics on a Kähler manifold essentially reduces to two classes: Kähler--Einstein metrics and cscK metrics. The existence of the former on General Type and Calabi--Yau compact manifolds was proved in the seminal works of Yau and Aubin \cite{Szekelyhidi2014}, while for Fano manifolds this problem is more delicate. Recently, the existence of cscK metrics has also been fully characterized \cite{ChenCheng2021II}.

Addressing the same problem in a generic Hermitian setting is much more involved. Indeed, a non-Kähler manifold admits several Hermitian connections that can be considered when a curvature condition is imposed; since they all have torsion, some technicalities need to be introduced. On the other hand, the Kähler condition can be relaxed as well: this is done by identifying weaker cohomological properties for the fundamental form. Yet several choices emerge naturally, depending on the context \cite{GrayHervella1980,GarciaFernandez2016}. This leads to introducing many different \emph{``special'' Hermitian metrics}.

Although the equations defining such metrics appear deceptively simple (we will discuss two of them in \cref{Special_metrics_LVMB}), examples appear either in very special classes, or sporadically: this leaves many questions and conjectures unanswered. Indeed, it remains unclear to what extent admitting a special metric restricts a complex manifold, and how two different such metrics interact. For example, it is conjectured by A. Fino and L. Vezzoni in \cite{FinoVezzoni2015} that the coexistence of a balanced metric with an SKT one (see below) forces the underlying manifold to be Kähler. This holds for a large class of quotients of Lie groups carrying a complex structure, but since examples lack in the first place, the general picture has not yet been fully understood.

Many of these problems have been investigated on quotients of Lie groups where explicit computations are often possible \cite{AbbenaGrassi1986,FinoPartonSalamon2004}, while outside this class our knowledge is significantly more limited \cite{Fino2025}. This is our main motivation to address this kind of problem for LVMB manifolds. On one hand, their structure is stiff enough to make computations feasible; on the other hand, they are very abundant as any compact toric manifold admits multiple ``companion'' LVMB manifolds. Moreover, their construction is not related to Lie groups, hence we expect to encounter genuinely different behaviours.

In the geometric structure of LVMB manifolds we can find a holomorphic foliation \cite{LoebNicolau1996,Meersseman2000}. Under a certain rationality condition on the initial datum, the leaves are complex tori and the leaf space is a compact toric variety---either smooth, or endowed with finite quotient singularities \cite{MeerssemanVerjovsky2004}. We restrict to smooth base-spaces; this identifies a subclass---which we call \emph{regular}---of LVMB manifolds that are holomorphic principal torus bundles. This suggests that their geometry could be determined (at least partially) by the characteristic class of the associated torus bundle and hence by the generating algebro-combinatoric datum itself. The relations between this datum, the LVMB manifold it generates, and the bundle structure, are explored in \cref{section:construction} of this paper. This closer description of the structure of regular LVMB manifolds will be crucial to apply some general results by D. Grantcharov, G. Grantcharov and Y.~S. Poon \cite{GrantcharovPoon2008} to address the existence of “special'' Hermitian metrics on this class of manifolds.

In order to investigate the bundle structure of regular LVMB manifolds, we adopt Bosio's approach \cite{Bosio2001} and the construction described in \cite{BattagliaZaffran2015}. From their perspective, any complete simplicial fan can be embedded---often by suitably adding some vectors called \emph{ghosts}---into a so-called \emph{triangulated vector configuration} \cite{LoeraRambauSantos2010}. For rational smooth fans, such an object is given by a pair \((V_\Sigma,\calT_\Sigma)\), where \(V_\Sigma\) is a list containing all primitive vectors of a fan \(\Sigma\), plus the ghosts, while \(\calT_\Sigma\) is the combinatorial datum of \(\Sigma\). We will call these triangulated configurations  \emph{regular}.

From this, by choosing an ordered basis \((f_0,\dots,f_{2m})\) for the space of linear relations among the vectors in \(V_\Sigma\), and then applying the so-called \emph{Gale transform} as in \cite{BattagliaZaffran2015}, we obtain a point configuration in a complex affine space. Together with the fan combinatorial datum, this gives rise to an LVMB manifold \(N\), which then is the total space of a principal holomorphic torus bundle over the compact toric manifold \(X_\Sigma\).

This structure clearly emerges when this construction is compared with the well-known Audin--Cox construction of toric manifolds. Exploiting this, in this paper we describe how the choice of a Gale dual configuration determines a splitting of \(\pi : N \to X\) into \(S^1\)-principal sub-bundles. This point of view is crucial for our discussion: indeed each of them can be represented as a complex line bundle over \(X\), and since the base space is a toric manifold, their first Chern class can be algebraically described in terms of the simplicial fan corresponding to \(X\). This traces a clear link between the characteristic class of the holomorphic torus bundle \(\pi : N \to X_\Sigma\) associated to a regular LVMB manifold, and the initial triangulated vector configuration for which a Gale dual configuration has been chosen. We provide an explicit formula in the theorem:

\begin{namedtheorem}[{\labelcref{them:father}}]
    Let \((V_\Sigma, \calT_\Sigma)\) be a regular triangulated configuration with underlying fan \(\Sigma\). Let \((f_0, f_1, \dots, f_{2m})\) be any Gale dual configuration for this datum and \(K\) be the lattice spanned by these vectors. Then, with \(N\) denoting the corresponding LVMB manifold, the characteristic class of the principal torus bundle \(\pi : N \to X_\Sigma\) is
    \begin{equation*}
        \gamma(\pi) = \sum_{j=1}^{2m} \iota^\vee(f_j^*) \otimes f_j \in K_\Sigma^\vee \otimes K,
    \end{equation*}
    where  \(\iota : K_\Sigma \hookrightarrow K\) is the inclusion of the sublattice \(K_\Sigma\) of linear relations among non-ghost vectors, and \(\iota^\vee\) its adjoint map.
\end{namedtheorem}

We use this to provide a bigraded model for the algebra of differential forms of a regular LVMB manifold. As an application, we then prove that a regular LVMB manifold satisfying the \(\del\delbar\)-lemma must be a complex torus. These facts were already known for a subclass of LVMB manifolds called \emph{moment-angle manifolds} \cite{PanovUstinovsky2012}. Leveraging some results in \cite{Hofer1993}, we can further exploit the bigraded model to explore the Borel spectral sequence induced by the holomorphic torus bundle \(\pi : N \to X_\Sigma\). This yields a neat expression for the first Chern and Bott--Chern classes of a regular LVMB manifold.

It has been clear since the early works on LVMB manifolds, that those admitting a Kähler metric are only complex tori \cite{LopezVerjovsky1997,Meersseman2000,Bosio2001}. Yet, the existence of other special metrics on manifolds in the LVMB class remained unexplored before Faucard \cite{Faucard2024} addressed the existence of \emph{Locally Conformally Kähler} metrics (LCK for brevity). His answer is mostly negative: it is proven that, apart from Hopf manifolds, LCK metrics on LVMB manifolds are obstructed at the level of the group of automorphisms. See also \cite{Faucard2025} for LCK metrics on coverings of LVMB manifolds.

As an application of \cref{lemma:charClassGhosts} and of its cohomological implications, we explore the existence of two other types of special metrics on LVMB manifolds: \emph{balanced} metrics and \emph{SKT} metrics, of which we recall the definitions in \cref{Special_metrics_LVMB}. These metrics appear naturally in complex geometry: for example, the former exist on bimeromorphic modifications of Kähler manifolds, while the latter define a canonical class of Hermitian metrics on complex surfaces. We obtain the following main facts regarding them:

\begin{namedtheorem}[{\labelcref{balancedness_obstruction}}]
    Let \((V,\calT)\) be a regular vector configuration containing at least one ghost vector. Then, any corresponding LVMB manifold admits a balanced metric if and only if it is a complex torus.
\end{namedtheorem}
This does not hold for configurations without ghost vectors. Indeed, in \cref{section:submersionMetricsBalanced} we explicitly construct an LVMB manifold \(N_3 \to (\CP^1)^3\) of complex dimension \(4\) supporting a balanced metric. This is, to the best of our knowledge, a new example. Indeed, employing the results developed so far, we show that it does not belong to any known class of compact complex manifolds admitting a balanced metric. This construction can be generalized to provide a family of examples of the form \(N_\ell \to (\CP^1)^\ell\) with complex dimension \(\left\lfloor\frac{3}{2}\ell\right\rfloor\).

Finally, we discuss the existence of SKT metrics on regular LVMB manifolds. Employing our formula for the characteristic class \(\gamma(\pi)\), we provide a cohomological obstruction to their existence. For LVMB manifolds \(N \to X_\Sigma\) such that \(X_\Sigma\) is projective, this is indeed a characterization.

\begin{namedtheorem}[{\labelcref{LVMB_SKT}}]
    Let \(\pi : N \to X_\Sigma\) be the principal torus bundle associated to a regular LVMB manifold \(N\), and let \(\gamma(\pi) = \sum_{j=1}^{2m} \iota^\vee(f_j^*) \otimes f_j\) be its characteristic class. Then, \(N\) admits an SKT metric only if
    \begin{equation*}
        \sum_{j=1}^{2m} \iota^\vee(f_j^*) \wedge \iota^\vee(f_j^*) = 0 \in H^4(X;\Z).
    \end{equation*}
    If \(X_\Sigma\) is projective, this condition is also sufficient.
\end{namedtheorem}
The projectivity of the base space is not a very restrictive assumption: this is indeed the case considered by Meersseman and Verjovsky in \cite{MeerssemanVerjovsky2004}. As an application we can prove:
\begin{namedtheorem}[\labelcref{SKT_either_balanced_family}]
    Let \(N_\ell \to (\CP^1)^\ell\) be the family of LVMB manifolds constructed as above. Then, \(N_\ell\) admits a balanced metric if and only if \(\ell\) is odd, while it admits an SKT metric if and only if \(\ell\) is even. In this case it is a product of Calabi--Eckmann threefolds.
\end{namedtheorem}
This behaviour is consistent with the aforementioned conjecture by A.~Fino and L.~Vezzoni \cite{FinoVezzoni2015}.

\paragraph{Acknowledgements}
I am deeply grateful to my advisors prof. Daniele Angella and prof. Fiammetta Battaglia for their constant support and encouragement during the writing of this paper. I wish to thank also prof. Cristiano Spotti for his support and the fruitful conversation with him during my visit at Aarhus; and prof.s Gueo Grantcharov and Anna Fino for their valuable comments during the revision.

The author is supported by: Università degli Studi di Firenze, INdAM--GNSAGA and PRIN 2022 project “Real and Complex Manifolds: Geometry and holomorphic dynamics'' (code 2022AP8HZ9).

\section{The structure of LVMB manifolds}
\label{section:construction}
In their full generality, LVMB manifolds can be obtained as leaf spaces of a holomorphic group action; by construction, their geometry is controlled by a datum \((\Lambda,\mathcal{E})\) which we will call \emph{LVMB datum} \cite{Meersseman2000,Bosio2001}. This is \emph{Gale dual} to a so-called \emph{triangulated vector configuration} \((V,\calT)\) \cite{LoeraRambauSantos2010}; however, the way the latter determines the former depends on some choices. This means that the same triangulated configuration produces several LVMB data, hence many LVMB manifolds. In turn, any triangulated configuration contains a simplicial fan embedded in it, thus considering such configurations instead of LVMB data highlights the relation between LVMB and toric manifolds.

This is the setting we consider throughout this paper: when we restrict to the regular case---i.e. to considering triangulated vector configurations \((V,\calT)\) encapsulating a \emph{smooth} complete simplicial fan \(\Sigma\)---the LVMB manifolds we obtain are the total spaces of holomorphic torus bundles over the toric manifold associated to \(\Sigma\) \cite{MeerssemanVerjovsky2004}. This motivates the smoothness assumption we make throughout this paper. However, in the more general case where \(\Sigma\) is just a rational simplicial fan, these bundles get only mild singularities \cite{MeerssemanVerjovsky2004}, so we conjecture that the techniques we describe here could be generalized.

\subsection{Construction of LVMB manifolds}
\label{section:initialData}
In the following, we adopt the approach of \cite{Bosio2001} and \cite{BattagliaZaffran2015}.

\begin{definition}
    Let \(E\) be a \(d\)-dimensional vector space over \(\R\).
    A \emph{triangulated vector configuration} is a pair \((V, \calT)\) where
    \begin{itemize}
        \item \(V = (v_1, \dots, v_n)\) is a list of vectors in \(E\), admitting repetitions. We assume that \(E = \left\langle v_1, \dots, v_n \right\rangle_\R\).
        \item We call \(\tau \subseteq \{1,\dots,n\}\) a \emph{simplex} if \(\{v_j:j\in\tau\}\) is a set of linearly independent vectors. A simplicial \emph{cone} over \(\tau \in \calT\) is the convex set \(\cone(\tau) = \{\sum_{j\in\tau}\R_{\geq 0}v_j\}\). Then, \(\calT\) is an abstract simplicial complex of these simplexes such that:
        \begin{itemize}
            \item \(\cone(\tau) \cap \cone(\tau') = \cone(\tau\cap\tau')\) for all \(\tau,\tau' \in \calT\),
            \item \(\bigcup_{\tau \in \calT} \cone(\tau) \supseteq \cone(V)\).
        \end{itemize}
        In particular \(\cone(\emptyset) = \{0_E\}\), while, by convention, \(\cone(V) = \sum_{j=1}^n \R_{\geq} v_j\).
    \end{itemize}
    A triangulated vector configuration \((V,\calT)\) is \emph{rational} if there exists a full-rank lattice \(L\) in \(E\) such that \(v_i \in L\) for all \(i = 1,\dots, n\). Notice that, in this case, 
    \begin{equation*}
        \Rel(V) = \left\{c \in \R^n : \sum_{j=1}^n c_j v_j = 0\right\}
    \end{equation*}
    is a rational subspace in \(\R^n\)---that is, it admits a basis of vectors in \(\Q^n\). We say that \((V,\calT)\) is \emph{regular}, if, for any \(\tau \in \calT\) such that \(|\tau| = d\), the set \(\{v_j : j\in \tau\}\) is a \(\Z\)-basis of \(L \subset E\).

    A configuration \((V,\calT)\) is called \emph{odd} if \(V\) has odd cardinality, and \emph{balanced} if \(\sum_i v_i = 0\).
\end{definition}
\begin{remark}
    This definition allows some vectors in \(V\) not to belong to any simplex of \(\calT\). These will be called \emph{ghost vectors}, and up to a change of indices, they can be put at the head of \(V\) as \(v_1, \dots, v_k\).
\end{remark}

For the construction of LVMB manifolds we will need to consider odd and balanced configurations satisfying \(n - d = 2m+1\) for some positive integer \(m\). This is not very restrictive as, just by adding ghost vectors, we can turn any configuration \((V,\calT)\) into a new one satisfying the conditions above. In fact, balancedness can be achieved by prepending \(-\sum_i v_i\) to \(V\) and increasing \(n\) by 1; oddness, instead, can be obtained just by prepending one or two copies of \(0\) to \(V\) and correspondingly increasing \(n\).

\subsubsection{Simplicial fans as triangulated vector configurations}
\label{section:simplicial_fans_as_triangulated_vector_configurations}
Polyhedral fans are the cornerstones of toric geometry. For an exhaustive presentation of this subject, we refer to Cox, Little and Schenck's book \cite{CoxLittleSchenck2011}. Here we recall some definitions.
\begin{definition}
    Let \(\Sigma\) be a polyhedral fan in \(E\), we denote its \(j\)-dimensional skeleton \(\Sigma^{(j)}\):
    \begin{itemize}
        \item \(\Sigma\) is called \emph{simplicial} when the minimal generators of each of its maximal cones are linearly independent over \(\R\),
        \item \(\Sigma\) is called \emph{complete} if its support \(|\Sigma|\) coincides with \(E\),
        \item \(\Sigma\) is called \emph{rational} if there exists  a full-rank lattice \(L\) in \(E\) that has nonempty intersection with each one-dimensional cone in the fan,
        \item \(\Sigma\) is called \emph{smooth} if the minimal generators of maximal cones in \(\Sigma\) generate \(L\) as \(\Z\)-module.
    \end{itemize}
    Thus, to any rational fan there naturally corresponds the set of minimal generators of the one-dimensional cones; these last are called \emph{rays}. If the fan is smooth, this set is a set of \(\Z\)-generators of the lattice.
\end{definition}

Any \emph{smooth} fan \(\Sigma\) can be encoded in a \emph{regualar} triangulated vector configuration \((V_\Sigma,\calT_\Sigma)\) as follows. We consider \(\tilde v_1, \dots, \tilde v_s\) as the minimal generators of \(\Sigma^{(1)}\), then we set \(V_\Sigma = (\tilde v_1, \dots, \tilde v_s)\); on the other hand, for the combinatorial datum \(\calT_\Sigma\), we just copy the combinatorics of \(\Sigma\). More precisely, \(\tau \in \calT_\Sigma\) if and only if \(\{\tilde v_j : j \in \tau\}\) spans a cone in \(\Sigma^{(|\tau|)}\). As we started with a smooth fan, \((V_\Sigma,\calT_\Sigma)\) is well-defined, rational and, moreover, regular.

The configuration \((V_\Sigma,\calT_\Sigma)\) obtained this way is not necessarily odd and balanced; however it can be embedded into such a one by adding an appropriate number of ghost vectors as described above. In order to preserve rationality, these must lie in the lattice with respect to which \(\Sigma\) is rational. At the end of this process, we get a new rational vector configuration
\begin{equation*}
    V_\Sigma' = (v_1, \dots, v_k, v_{k+1} = \tilde v_1, \dots, v_{k+s} = \tilde v_s).
\end{equation*}
When no confusion is feared, we will not distinguish \(V_\Sigma\) and \(V_\Sigma'\). Hence, we will often use the notation \((V_\Sigma,\calT_\Sigma)\) to denote an odd and balanced triangulated configuration encapsulating a complete simplicial fan \(\Sigma\); when needed, ghost vectors will be specified in the context.
\begin{remark}
    \label{remark:The_note}
    A complete simplicial fan, whether rational or not, yields a whole family of toric data \cite{Prato2001}. In turn, each such datum can be embedded into infinitely many triangulated vector configurations, all odd and balanced \cite{BattagliaZaffran2015}. In this paper, we restrict our attention to the smooth case, see \cref{remark:nonrational} for comments about the general setting.
\end{remark}

\begin{note}
    Any simplicial fan \(\Sigma\) embedded in a balanced triangulated configuration \((V_\Sigma,\calT_\Sigma)\) must be complete. Indeed, since \(\sum_{v \in V_\Sigma} v = 0\), it holds \( E = \cone(V_\Sigma) \subseteq \bigcup_{\tau \in \calT} \cone(\tau)\). As these are precisely the cones of \(\Sigma\), it is complete.
\end{note}
\subsubsection{Gale Transform}
Now choose a Gale Transform of a triangulated vector configuration \((V,\calT)\); this is the first of two choices required to produce an LVMB manifold from a triangulated vector configuration.

This is done by choosing an ordered basis \((f_0,\dots, f_{2m})\) for \(\Rel(V)\), such that the matrix \(M\) has the following form:
\begin{equation*}
    M =
    \begin{bmatrix}
        \vertbar &  & \vertbar\\
        f_0 & \cdots & f_{2m}\\
        \vertbar &  & \vertbar
    \end{bmatrix}
    =
    \begin{bNiceMatrix}
        1 & f_1^1 & \Cdots & f^1_{2m}\\
        \Vdots & \Vdots & \Ddots & \Vdots \\
        1 & f^n_1 & \Cdots & f^n_{2m}
    \end{bNiceMatrix}.
\end{equation*}
A \emph{Gale dual configuration} of \(V\) is then the list of vectors \(\hat\Lambda^\R = (\hat\Lambda^\R_1, \dots, \hat\Lambda^\R_{n})\) in \(\R^{2m+1}\) such that
\begin{equation*}
    M =
    \begin{bmatrix}
        \horzbar \hat\Lambda_1^\R \horzbar\\
        \vdots\\
        \horzbar \hat\Lambda_n^\R \horzbar
    \end{bmatrix}.
\end{equation*}
Since each \(\hat\Lambda^\R_j\) has \(1\) as its first entry, we can trim them to a configuration of points \(\Lambda^\R = (\Lambda^\R_1, \dots, \Lambda^\R_n)\) in the affine space \(\mathbb{A}^{2m}_{\R}\), where \(\hat\Lambda^\R_j \eqqcolon [1 \horzbar \Lambda^\R_j \horzbar]\).
\begin{remark}
    As in the rest of the paper, the \((2m+1)\)-tuple \((f_0, \dots, f_{2m})\) will play a major role; we will often refer to it as a Gale dual configuration of some \((V,\calT)\) in place of \(\Lambda^{\R}\). This abuse of notation is negligible, as each of these two pieces of data uniquely determines the other.
\end{remark}

\subsubsection{The construction}
\label{section:TheConstruction}
Now we show how a genuine LVMB datum \((\Lambda,\mathcal{E})\) can be obtained from a triangulated vector configuration \((V,\calT)\), for which a Gale transform \(\Lambda^{\R}\) has been chosen. We do so by selecting a \emph{period matrix}---that is a matrix \(\Pi = [P|I_m] \in \operatorname{Mat}_{m\times 2m}(\C)\) with \(\det\Im P \neq 0\) (cf. \cite[\S 1.2]{BirkenhakeLange2004}). This is a technicality we have introduced in the construction of LVMB manifolds to allow more flexibility. The choice of such matrix gives an identification between \(\mathbb{A}_\R^{2m}\) and \(\mathbb{A}_\C^m\). In turn, \(\Lambda^{\R}\) is identified with the \(n\)-tuple of points \(\Lambda = (\Lambda_1, \dots, \Lambda_n)\) in \(\mathbb{A}_\C^m\) given by
\begin{equation*}
    \Lambda =
    \begin{bmatrix}
        \horzbar \Lambda_1 \horzbar\\
        \vdots\\
        \horzbar \Lambda_n \horzbar\\
    \end{bmatrix}
    =
    \begin{bmatrix}
        \horzbar \Lambda_1^{\R} \horzbar\\
        \vdots\\
        \horzbar \Lambda_n^{\R} \horzbar\\
    \end{bmatrix}
    \trans{\Pi}
    = \Lambda^{\R} \trans{\Pi},
\end{equation*}
which is the first datum in the pair \((\Lambda,\mathcal{E})\).
On the other hand, \(\mathcal{E}\) is the \emph{virtual chamber} associated to \(\calT\); this is given by
\begin{equation*}
    \mathcal{E} = \left\{\{1, \dots, n\}\setminus \tau : \tau \in \calT, |\tau| = d \right\}.
\end{equation*}
We denote its elements as \(\varepsilon\). The pair \((\Lambda,\mathcal{E})\) obtained in this way is an LVMB datum \cite[Prop. 2.1]{BattagliaZaffran2015} that allows us to construct a unique LVMB manifold following \cite{Bosio2001}.

The LVMB manifold associated to \((\Lambda,\mathcal{E})\) then arises as follows: the virtual chamber \(\mathcal{E}\) determines an open subset in \(\C^n\) defined as
\begin{equation*}
    U(\calT) = \bigcup_{\varepsilon \in \mathcal{E}} \bigcap_{j \in \varepsilon} D(z_j)
\end{equation*}
where \(D(z_j) = \{z_j \neq 0\}\) denotes the distinguished open subset of the Zariski topology of \(\C^n\). On its projectivization \(\proj(U(\calT))\), the datum provided with \(\Lambda\) defines a \(\C^m\)-action via the map
\begin{equation}
    \label{equation:LVMBaction}
    \begin{aligned}
        \alpha : \C^m \times \proj(U(\calT)) &\longrightarrow \proj(U(\calT))\\
        (u, (z_1 : \dots : z_n)) &\longmapsto \left(e^{2\pi i \left\langle \Lambda_1, u \right\rangle} z_1 : \dots : e^{2\pi i \left\langle \Lambda_n, u \right\rangle} z_n\right)
    \end{aligned}    
\end{equation}
where \(\left\langle \cdot, \cdot\right\rangle\) denotes the standard bilinear pairing in \(\C^m\). The quotient \(\faktor{\proj(U(\calT))}{\alpha}\) is a compact complex manifold which we will denote by \(N\).

We say that \(N\) is \emph{regular} if the generating triangulated configuration \((V,\calT)\) is regular.

\begin{remark}
    \label{indispensableRemark}
    The Gale Transform is what links a triangulated vector configuration to an LVMB datum. In particular, as the indices of ghost vectors belong to any \(\varepsilon \in \mathcal{E}\), the corresponding coordinates in \(U(\calT)\) never vanish. In the literature, these are known as \emph{indispensable indexes} \cite{Meersseman2000}. From this perspective, the Gale Transform makes ghost vectors correspond to indispensable points of the LVMB datum. 
\end{remark}

\subsubsection{Dependency on choices}
\label{section:dependecyOnChoices}
The construction described above depends on two choices: the Gale transform, which turns a vector configuration \(V\) into an \(n\)-tuple \(\Lambda^{\R}\) of points in \(\mathbb{A}^{2m}_{\R}\), and the period matrix \(\Pi\), which identifies \(\R^{2m}\) with \(\C^m\). Consequently, \(\mathbb{A}^{2m}_{\R}\) and \(\mathbb{A}^{m}_{\C}\) are also identified.

However, these choices are not independent; indeed, we have the following:
\begin{proposition}
    \label{ThreeOutOfFour}
    Let \((V,\calT)\) be any triangulated vector configuration. Fix a Gale dual configuration \(\Lambda{^\R}\) and a period matrix \(\Pi\). Then, for any other Gale dual configuration \(\tilde\Lambda^{\R}\) there exists a unique period matrix \(\Pi'\) such that
    \begin{equation*}
        \Lambda^{\R} \trans{\Pi} = \tilde\Lambda^{\R} \trans{\Pi}'.
    \end{equation*}
    Similarly, for any choice of \(\Pi'\) there exists a unique \(\tilde\Lambda^{\R}\) satisfying the same equation.
\end{proposition}
\begin{proof}
    As the columns of both \(\Lambda^{\R}\) and \(\tilde\Lambda^{\R}\) span \(\Rel(V)\), there exists \(T \in \operatorname{GL}(2m;\R)\) such that \(\Lambda^{\R} T = \tilde\Lambda^{\R}\). Then, \(\Pi' = \Pi \trans{T}^\inv\) has the desired property. This is still a period matrix because
    \begin{equation*}
        \det
        \begin{bmatrix}
            \Pi'\\
            \overline{\Pi}'
        \end{bmatrix}
        = \det
            \begin{bmatrix}
                \Pi\\
                \overline{\Pi}
            \end{bmatrix} 
            \cdot \det T^\inv \neq 0
    \end{equation*}
    as both factors are non-singular.

    Conversely, suppose that a period matrix \(\Pi'\) has been fixed. Then we can observe that
    \begin{equation*}
        [\Re\trans{\Pi}|\Im\trans{\Pi}] = \frac{1}{2} \left[\trans{\Pi}\,\Big|\!\trans{\overline{\Pi}}\right]
        \begin{bmatrix}
            I_m & -i I_m\\
            I_m & i I_m
        \end{bmatrix},
    \end{equation*}
    so the \((2m\times 2m)\)-matrices \([\Re\trans{\Pi} | \Im\trans{\Pi} ]\) and \([\Re\trans{\Pi}' | \Im\trans{\Pi}' ]\) are clearly non-singular. Therefore, we can define \(\tilde{\Lambda}^{\R}\) as
    \begin{equation*}
        \tilde\Lambda^{\R} = \Lambda^{\R} [\Re\trans{\Pi} | \Im\trans{\Pi} ] \, [\Re\trans{\Pi}' | \Im\trans{\Pi}' ]^{\inv}.
    \end{equation*}
    This is a Gale dual configuration as its columns clearly span \(\Rel(V)\); moreover
    \begin{equation*}
        \tilde\Lambda^{\R} \trans{\Pi}' = \tilde\Lambda^{\R} [\Re\trans{\Pi}' | \Im\trans{\Pi}' ] [I_m| i I_m] = \Lambda^{\R} [\Re\trans{\Pi} | \Im\trans{\Pi} ]  [I_m| i I_m] = \Lambda^{\R} \trans{\Pi},
    \end{equation*}
    so the equation in the statement is fulfilled.
\end{proof}

As the construction of LVMB manifolds is controlled by the datum \((\Lambda,\mathcal{E}) = (\Lambda^{\R}\trans{\Pi}, \mathcal{E})\), this result shows that we are free to modify either the choice of \(\Lambda^{\R}\) or the choice of \(\Pi\), provided that the other datum is changed accordingly.

Nevertheless, many choices for \(\Lambda^\R\) and \(\Pi\) produce the same LVMB manifold. Indeed, as the action \(\alpha\) is invariant under automorphisms of \(\mathbb{A}^m_{\C}\) \cite{Meersseman2000}, all the LVMB data in the set \(\{(\Lambda Z,\mathcal{E}) : Z \in \operatorname{GL}(m,\C)\}\) produce the same LVMB manifold. Thus, for a fixed choice of Gale dual configuration \(\Lambda^\R\), any period matrix in \(\{Z\,\Pi : Z \in \operatorname{GL}(m,\C)\}\) induces the same construction. Yet, we have a natural choice for \(\Pi\), namely \(\Pi_\text{std} = i[-i I_m | I_m] = [I_m | i I_m]\), which we will call \emph{standard period matrix}. Unless stated otherwise, we will use \(\Pi_\text{std}\) to construct LVMB manifolds.

\Cref{ThreeOutOfFour} essentially indicates that choosing a different Gale dual configuration or changing the identification between \(\mathbb{A}^{2m}_\C\) and \(\mathbb{A}^m_{\C}\) affects the construction of LVMB manifolds in the same way. Since we are able to keep track of such a change of Gale duality, it is worth identifying some choices that make the forthcoming discussion easier to follow.

\begin{definition}
    Let \((V_\Sigma,\calT_\Sigma)\) be a regular, odd and balanced triangulated vector configuration in \(\R^d\) containing \(n\) vectors; assume it contains \(k \geq 1\) ghost vectors. Let \(\Rel(\Sigma) \subset \Rel(V)\) be the subspace spanned by linear relations among non-ghost vectors in \(V_\Sigma\). Consider the sublattices of \(\Z^n\) given by \(K = \Rel(V_\Sigma) \cap \Z^n\) and \(K_\Sigma = \Rel(\Sigma) \cap \Z^n\), together with the natural inclusion \(\iota : K_\Sigma \hookrightarrow K\). We call \emph{good} any Gale dual configuration \((f_0,\dots,f_{2m})\) that is a basis of \(K\) of the following form:
    \begin{equation*}
        \begin{bmatrix}
            \vertbar & & \vertbar\\
            f_0 & \cdots & f_{2m}\\
            \vertbar && \vertbar
        \end{bmatrix}
        =
        \begin{bmatrix}
            \vertbar & & \vertbar & \vertbar && \vertbar\\
            f_0 & \cdots & f_{k-1} & \iota(h_{1}) &\cdots & \iota(h_{2m-k+1})\\
            \vertbar & & \vertbar & \vertbar && \vertbar
        \end{bmatrix}
        =
        \begin{bNiceArray}{c|ccc|ccc}[margin]
            1 & \Block[borders={bottom}]{1-6}{0} &&&&&\\
            \Vdots&\Block[borders=bottom]{2-3}{I_{k-1}} & & & \Block[borders=bottom]{2-3}{0} \\
            \\
            &\Block{3-3}{[m_{ij}]} &&& \vertbar && \vertbar \\
            &&&& h_{1} & \Cdots & h_{2m-k+1}\\
            1&&&& \vertbar && \vertbar
        \end{bNiceArray},
    \end{equation*}
    where \(m_{ij} \in \Z\), and \(\{h_1,\dots,h_{2m-k+1}\}\) is a basis for \(K_\Sigma\).
\end{definition}
The above definition does not immediately extend to configurations without ghost vectors. Nonetheless, as in this case \(K = K_\Sigma\), we call \emph{good} a Gale dual configuration of the form
\begin{equation*}
    \begin{bmatrix}
        \vertbar & & \vertbar\\
        h_1 & \cdots & h_{2m+1}\\
        \vertbar && \vertbar
    \end{bmatrix}
    =
    \begin{bNiceArray}{c|ccc}[margin]
        1 & \Block[borders={bottom}]{1-3}{0} &&\\
        \Vdots& \vertbar && \vertbar \\
        & h_{1} & \Cdots & h_{2m+1}\\
        1& \vertbar && \vertbar
    \end{bNiceArray}.
\end{equation*}

\begin{proposition}
    \label{goodGaleDualExistence}
    Any regular odd and balanced triangulated vector configuration \((V_\Sigma,\calT_\Sigma)\) admits a good Gale dual configuration.
\end{proposition}
\begin{proof}
    If \(k = 0\) there is nothing to prove. Indeed, as in this case \(\Sigma\) is balanced by itself, we can assume \(h_1 = [1,\dots,1]\). Completing this to a basis of \(K_\Sigma\), we obtain a good Gale dual configuration.

    Conversely, suppose \(k \geq 1\). Then, as we are assuming that ghost vectors are placed at the head of \(V_\Sigma\), \(\Rel(\Sigma) \subset \Span_{\R}(e_{k+1},\dots, e_{n}) =\R^{n-k} \subset \R^n\). It can be easily checked that \(K_\Sigma = \Rel(\Sigma) \cap K\), thus we can apply \cite[Thm.31]{SiegelChandrasekharan1989} to complete a basis \(\{h_1,\dots, h_{2m-k+1}\}\) of \(K_\Sigma\) to a basis \(\{f_0, \dots, f_{k-1},\iota(h_{1}),\dots, \iota(h_{2m-k+1})\}\) of \(K\). We can assume that \(f_0 = [1,\dots, 1]\) since, as \((V_\Sigma,\calT_\Sigma)\) is balanced, \(f_0 \in K\). Moreover, the lattice \(\hat K = K \cap \Span_{\Z}(e_2,\dots, e_n)\) is such that \(K = \Z f_0 \oplus \hat K\). In fact, the natural inclusion \(\hat K \hookrightarrow K\) is a section for the exact sequence \(0 \to \Z f_0 \to K \to \hat K \to 0\) whose projection takes
    \begin{equation*}
        K \ni \kappa = \sum_{j=1}^n \kappa_j e_j \mapsto \sum_{j=2}^n (\kappa_j - \kappa_1) e_j \in \hat K.
    \end{equation*}
    
    Finally, we need to prove that \(f_1,\dots, f_{k-1}\) can be put in the form described in the statement. For this, we consider any maximal simplex \(\tau \in \calT\). Since \((V_\Sigma,\calT_\Sigma)\) is regular, \(\Sigma\) is a smooth fan; hence \(\{v_j \in V_\Sigma : j \in \tau\}\) is a basis of \(\Z^d \subset \R^d\). Thus, the projection sending \(\Z^n \ni e_j \to v_j \in \Z^d\) is split epic. Hence, \(K\), and a fortiori \(\hat K\), are direct summands of \(\Z^n\). This tells us that \(f_1,\dots, f_{k-1}\) is a basis of \(\hat K\), which can be completed to a basis of \(\Z^n\). Thus, by taking its \emph{Hermite normal form} (see e.g. \cite[\S 1.5]{Mader2000}), we obtain a Gale dual configuration of the form
    \begin{equation*}
        \begin{bNiceArray}{c|ccc|ccc}[margin]
            1 & \Block[borders={bottom}]{1-6}{0} &&&&&\\
            \Vdots&\Block[borders=bottom]{2-3}{A} & & & \Block[borders=bottom]{2-3}{0} \\
            \\
            &\Block{3-3}{[m_{ij}]} &&& \vertbar && \vertbar \\
            &&&& h_{1} & \Cdots & h_{2m-k+1}\\
            1&&&& \vertbar && \vertbar
        \end{bNiceArray},
    \end{equation*}
    where \(A \in \operatorname{SL}(k-1,\Z)\). By performing another unimodular change of basis, we can tweak \(A\) so that \(A = I_{k-1}\).
\end{proof}

\subsubsection{Classical examples}
\label{ClassicalExamples}
LVMB manifolds generalize well-known constructions of complex manifolds \cite{LopezVerjovsky1997,Meersseman2000}. Here we briefly review how they can be obtained in the framework described above.

\begin{example}[Complex tori]
    \label{example:complex_tori}
    Complex tori can be interpreted as degenerate LVMB manifolds. Indeed, the \(m\)-dimensional complex torus is obtained via the LVMB construction considering the trivial configuration \((V_0, \calT_0)\) in \(\R^0 = \{0\}\). \(V_0\) is a list containing \(2m+1\) copies of \(0\), and \(\calT_0\) is generated by the \((2m+1)\)-tuple \((0,\dots,0)\). Hence, \(\proj(U(\calT_0)) \iso (\C^*)^{2m}\) and, choosing the trivial linear relations among the null vectors, we get the Gale dual configuration given by the points \(\Lambda_j = e_j\) for \(j=1,\dots, m\), and \(\Lambda_{j} = i e_{j-m}\) for \(j=m+1,\dots, 2m\). Taking the quotient \(\faktor{\proj(U(\calT_0))}{\alpha}\), we clearly get the \(m\)-dimensional complex torus of period matrix \([I_m | iI_m]\). The complex structure of the torus can be prescribed using any period matrix in the construction.
\end{example}

\begin{example}[Diagonal Hopf manifolds]
    \label{example:Hopf_manifolds}
    These are precisely the LVMB manifolds obtained from the polyhedral fan whose corresponding variety is \(\CP^d\). Let \(\Sigma\) be such fan in \(\R^d\), then  the corresponding \((V_\Sigma,\calT_\Sigma)\) is such that \(n-d = d+1-d = 1\). Hence, in order to get a valid configuration, we need to prepend to \(V_\Sigma\) two copies of \(0 \in \R^d\). As an ordered basis of \(\Rel(V)\) we take the columns of the matrix
    \begin{equation*}
        \begin{bNiceMatrix}
            1 & 0 & 0\\
            & 1 & 0\\
            \Vdots & 0 & 1\\
            & \Vdots & \Vdots\\
            1 & 0 &1
        \end{bNiceMatrix};
    \end{equation*}
    thus the corresponding Gale dual configuration is given by the complex points \(\Lambda_1 = 0, \Lambda_2 = 1\) and \(\Lambda_j = i\) for \(j = 3,\dots, d+3\). Since \(\proj (U(\calT)) \iso \C^* \times (\C^{d+1}\setminus \{0\})\), we immediately see that the quotient of the \(\C\)-action
    \begin{equation*}
        t.(z_1,\dots, z_{d+1}) = (e^{2\pi i t} z_1, e^{-2\pi t} z_2, \dots, e^{-2\pi t} z_{d+1})
    \end{equation*}
    defines the diagonal Hopf manifold \(M_{0,n;i}\). Again, for a generic \(\Pi = [1,\tau]\) we obtain \(M_{0,n;\tau}\).    
\end{example}

\begin{note}
    Any linear Hopf manifold \cite[Appendix]{Haefliger1985} can be obtained using this construction. This is done by making different choices for the ghost vectors added to \(V_\Sigma\). Tweaking the choice of the Gale dual configuration, we can obtain the desired linear homothety.
\end{note}

\begin{example}[Calabi--Eckmann manifolds]
    \label{example:C-E_manifolds}
    Similarly to the case of Hopf manifolds, we can obtain all Calabi-Eckmann manifolds as LVMB manifolds. Let \(\Sigma\) be the polyhedral fan corresponding to the toric variety \(\CP^a \times \CP^b\). As the corresponding triangulated configuration \((V_\Sigma,\calT_\Sigma)\) is balanced and satisfies \(n-d = (a +1 + b+1) - (a+b) = 2\), we need to prepend a zero as a ghost vector in order to obtain a valid configuration for the LVMB construction. As a basis for \(\Rel(V)\) we consider the columns of the \((a+b+1)\times 3\) matrix
    \begin{equation*}
        \begin{bNiceMatrix}
            1 & 0 & 0\\
            \Vdots& 1 & \Vdots\\
            & \Vdots \\
            & 1 &0 \\
            & 0 & 1\\
            & \Vdots & \Vdots\\
            1 & 0 &1
        \end{bNiceMatrix}.
    \end{equation*}
    The corresponding Gale dual configuration of points is then \(\Lambda_1 = 0\), \(\Lambda_j = 1\) for \(j = 2,\dots, a+1\) and \(\Lambda_j = i\) for \(j= a+2,\dots, a+b+1\). Hence, since \(\proj(U(\calT)) \iso (\C^{a+1}\setminus\{0\}) \times (\C^{b+1}\setminus\{0\})\), the \(\C\)-action of \(\Lambda\) can be written as
    \begin{equation*}
        t.(z_1,\dots, z_{a+1}, z_{a+2}, \dots, z_{a+b+2}) = (e^{2\pi i t} z_1, \dots, e^{2\pi i t} z_{a+1}, e^{-2\pi t} z_{a+2}, \dots, e^{-2\pi t} z_{a+b+2});
    \end{equation*}
    whence it is clear that the quotient LVMB manifold is the Calabi--Eckmann manifold \(M_{a,b;i}\) diffeomorphic to \( S^{2a+1} \times S^{2b+1}\). Of course, also in this case we obtain \(M_{a,b;\tau}\) setting \(\Pi = [1,\tau]\).
\end{example}

\begin{remark}
    From this construction it is clear how diagonal Hopf manifolds are a special case of Calabi--Eckmann manifolds. Actually, elliptic curves are also such manifolds. Indeed, taking \(a = b = 0\) in the Calabi--Eckmann construction leads to a configuration with three zeroes as ghosts. Hence, \(M_{0,0;\tau}\) is the elliptic curve of period \(\tau\) fibering over \(\CP^0 \times \CP^0 = \{\cdot\}\).
\end{remark}

\subsection{The manifold \texorpdfstring{\(N\)}{N} as a torus bundle}
\label{section:relationWithToricGeometry}

It is well known that toric varieties can be obtained as complex quotients \cite{Cox1995}, for the smooth case see also \cite{Audin2004}.
We keep our interest restricted to the smooth rational case, and we retrace the Audin--Cox construction in our terms. This allows us to split the principal torus bundle structure of a regular LVMB manifold into its \(S^1\)-principal components. This splitting is determined by the choice of the Gale transform in the construction.

To see this, let \(\Sigma\) be a smooth complete simplicial fan in \(\Z^d \subset \R^d\), and let \((V_\Sigma, \calT_\Sigma)\) be an associated triangulated configuration, with \(v_1, \dots, v_k\) as ghost vectors; suppose it is odd and balanced. These data induce the morphism of short exact sequences
\begin{equation}
    \label{equation:AudinExactSequence}
    \begin{tikzcd}
        0 \ar[r] & K_\Sigma \ar[r] \ar[d,"\iota|_{K_\Sigma}"] & \Z^{n-k} \ar[r,"p'"] \ar[d,"\iota"] & \Z^d \ar[r] \ar[d,equals]& 0\\
        0 \ar[r] & K \ar[r] & \Z^n \ar[r,"p"]& \Z^d \ar[r] & 0
    \end{tikzcd}
\end{equation}
The bottom row is associated to the map \(p : e_j \mapsto v_j\) for \(j = 1, \dots, n\); its restriction to non-ghost indexes (namely \(e_{k+1},\dots, e_n\)), gives the one on the top. The inclusion \(\iota: e_i' \mapsto e_{k+i}\) for \(i = 1,\dots, n-k\) naturally restricts to a morphism \(\iota|_{K_\Sigma} : K_\Sigma \hookrightarrow K\), mapping a linear relation among non-ghost vectors into itself, but this time seen as a linear relation among vectors in \(V_\Sigma\).

We choose a Gale dual configuration \((f_0,\dots,f_{2m})\), as \(\Sigma\) is smooth, by \cref{goodGaleDualExistence,ChangeOfBundles}, we can assume it to be good up to changing the period matrix. In particular, \(\{f_0,\dots, f_{k-1}\}\) is a basis for a complement \(L\) of \(\iota(K_\Sigma) \subseteq K\), while \(f_j = \iota(h_{j-k+1})\), where \(\{h_1,\dots, h_{2m-k+1}\}\) is a basis of \(K_\Sigma\).

Taking the tensor product with \(\C\) and then exponentiating, the morphism of exact sequences \labelcref{equation:AudinExactSequence} becomes by functoriality an exact sequence of algebraic tori
\begin{equation*}
    \begin{tikzcd}
        1 \ar[r] & K_\Sigma^{\C} \ar[r] \ar[d,"\iota_\C|_{K^{\C}_\Sigma}"] & (\C^*)^{n-k} \ar[r,"p_{\C}'"] \ar[d,"\iota_\C"] & (\C^*)^d \ar[r] \ar[d,equals]& 1\\
        1 \ar[r] & \tilde K^{\C} \ar[r] & (\C^*)^n \ar[r,"p_{\C}"]& (\C^*)^d \ar[r] & 1
    \end{tikzcd}
\end{equation*}
Explicitly, we have
\begin{equation*}
    K^\C_\Sigma = \exp\left(\frac{K_\Sigma \otimes \C}{\bigoplus_{i = 1}^{t} \Z h_i}\right), \qquad \tilde K_\C = \exp\left(\frac{K \otimes \C}{\bigoplus_{i = 0}^{2m} \Z f_i}\right).
\end{equation*}

Notice that the first is the group appearing in the standard Audin--Cox construction and it acts by multiplication on an open subset \(U_\Sigma \subset \C^{n-k}\)  \cite[\nopp VII.1.12]{Audin2004}. The second group, instead, appears in the construction of toric manifolds starting from the datum of a regular triangulated vector configuration \cite[\S 2.3.1]{BattagliaZaffran2015}. As shown there, \(\tilde K_\C\) acts by multiplication on the open subset \(U(\calT_\Sigma)\) described in \cref{section:TheConstruction}. Since ghost vectors do not belong to any simplex in the triangulation \(\calT_\Sigma\), one has \(U(\calT_\Sigma) = (\C^*)^k \times U_\Sigma\).

Since \((V_\Sigma, \calT_\Sigma)\) is regular, both the orbit spaces \(X_\Sigma' = \faktor{U_\Sigma}{K^\C_\Sigma}\) and \(X_\Sigma = \faktor{U(\calT_\Sigma)}{\tilde K^\C}\) are compact complex manifolds \cite[\nopp VII.1.14]{Audin2004} \cite{BattagliaZaffran2017}. In fact, these are biholomorphic \cite[Thm. 2.1]{BattagliaZaffran2017}, and so \(X_\Sigma\) is the compact toric manifold associated to the smooth complete simplicial fan \(\Sigma\). As the action of \(\tilde K^{\C}\) is principal, the projection onto the orbit space \(\tilde\pi : U(\calT_\Sigma) \to X_\Sigma\) is a principal \((\C^*)^{2m+1}\)-bundle.

\subsubsection{The bundle \texorpdfstring{\(\tilde\pi\)}{tilde pi} as a product of \texorpdfstring{\(\C^*\)}{CC*}-bundles}
\label{section:splitting_of_pi}

We now investigate these constructions more closely. Actually, here we observe that the choice of a Gale dual configuration for a triangulated vector configuration, induces a splitting of the \((\C^*)^{2m+1}\)-principal bundle \(\tilde\pi : U(\calT_\Sigma) \to X_\Sigma\) into \(\C^*\)-principal components. Indeed, choosing a Gale transform for \((V_\Sigma,\calT_\Sigma)\) determines an isomorphism \(\tilde K^\C \iso (\C^*)^{2m+1}\), thus its action splits into \(2m+1\) actions of the 1-dimensional algebraic torus \(\C^*\) commuting with each other. The \(j\)\th component of \(\alpha \in (\C^*)^{2m+1}\) acts as
\begin{equation*}
    \alpha_j.(z_0, \dots, z_n) = \left(\alpha_j^{f_j^1} z_1, \dots, \alpha_j^{f_j^{n}} z_n\right).
\end{equation*}
Therefore, under the above isomorphism, we can write \(\tilde K^\C = \prod_{i = 0}^{2m}\C^*_{f_i}\), where each factor is a copy of \(\C^*\) acting with \(f_i\).

\begin{remark}
    Since \(f_0 =[1,\dots,1]\), the corresponding algebraic torus \(\C^*_{f_0}\) acts diagonally on \(\C^n\), hence \(\faktor{U(\calT_\Sigma)}{\C^*_{f_0}} = \proj(U(\calT_\Sigma))\) as open subsets of \(\CP^{n-1}\). Thus,
    \begin{equation*}
        X_\Sigma \iso \frac{U(\calT_\Sigma)}{\tilde K^\C} \iso \frac{\proj(U(\calT_\Sigma))}{\faktor{\tilde K^\C\!\!\!}{\C^*_{f_0}} } \iso \frac{\proj(U(\calT_\Sigma))}{\prod_{i = 1}^{2m}\C^*_{f_i}} \nfd \frac{\proj(U(\calT_\Sigma))}{K^\C}.
    \end{equation*}
    Explicitly, a pair \((u,v) \in \C^m \times \C^m \iso \C^{2m}\) acts as
    \begin{equation}
        \label{equation:toricQuotientAction}
        (u,v).(z_1 : \dots : z_n) = \left(e^{2\pi i\left(\sum_{i=1}^m u_i f_i^1 + \sum_{i=1}^m v_i f_{m+i}^1 \right)} z_1 : \dots : e^{2\pi i\left(\sum_{i=1}^m u_i f_i^n + \sum_{i=1}^m v_i f_{m+i}^n \right)} z_n\right);
    \end{equation}
    we will denote this action by \(\beta\).
\end{remark}

Since \(K^{\C} \subseteq \tilde K^\C\) as a closed subgroup, its action on \(\proj(U(\calT_\Sigma))\) is proper and free as \(\tilde K^\C \action U(\calT_\Sigma)\) is so. Hence, the projection \(\tilde\pi : \proj(U(\calT_\Sigma)) \to X_\Sigma\) defines a principal \(K^\C\)-bundle. Moreover, the same argument can be used to split \(\tilde \pi\) into the product of \(2m\) principal \(\C^*\)-bundles \(\tilde \pi_j : \tilde F_j \to X_\Sigma\).

Indeed, we can consider a new torus \(K^\C_j \dfn \faktor{K^\C\!\!\!}{\C^*_{f_j}}\) acting as the product \(\prod_{i \neq j} \C^*_{f_i}\). As this is a closed subgroup of \(K^{\C}\), its action on \(\proj(U(\calT_\Sigma))\) is still free and proper, thus \(\tilde F_j = \faktor{\proj(U(\calT_\Sigma))}{{K_\C}_j} \) is a smooth manifold. Moreover, we have
\begin{equation*}
    X_\Sigma \iso \frac{\proj(U(\calT_\Sigma))}{ K^\C} \iso \frac{\faktor{\proj(U(\calT_\Sigma))}{K^\C_j}}{ \faktor{K^\C\!\!\!}{K^\C_j}} \iso \frac{\tilde F_j}{\C^*_{f_j}},
\end{equation*}
where \(\C^*_{f_j} \action \tilde F_j\) is the residual action. As this is induced by a free and proper action, it inherits these properties as well. Hence, the projection \(\tilde\pi_j : \tilde F_j \to X_\Sigma\) defines a principal \(\C^*\)-bundle. Moreover, as this action is the projection of (\ref{equation:toricQuotientAction}) on \({\tilde K_\C}_j\)-classes, we have the following isomorphism of principal bundles
\begin{equation*}
    \tilde \pi \iso \Delta^*\left(\prod_{i=1}^{2m} \tilde \pi_i\right),
\end{equation*}
where \(\Delta : X_\Sigma \to \prod_{i=1}^{2m} X_\Sigma\) is the diagonal map.

\begin{remark}
    \label{ChangeOfBundles}
    Different choices of Gale dual configurations \((f_0, f_1, \dots, f_{2m})\) and \((f_0,f_1',\dots, f_{2m}')\) provide different decompositions of \(\tilde\pi\) into \(\C^*\)-principal components
    \begin{equation*}
        \Delta^*\left(\prod_{i=1}^{2m} \tilde \pi_i\right) \iso \tilde \pi \iso \Delta^*\left(\prod_{i=1}^{2m} \tilde \pi_i'\right);
    \end{equation*}
    and the isomorphism \(\Phi_\sharp : \Delta^*\left(\prod_{i=1}^{2m} \tilde \pi_i\right) \to \Delta^*\left(\prod_{i=1}^{2m} \tilde \pi_i'\right)\) is induced by a real automorphism \(\Phi\) of \(\Rel(V_\Sigma)\) such that \(\Phi(f_j) = f_j'\) for \(j \geq 1\), and \(\Phi(f_0) = f_0\). Restricting it to each factor, we can see \(\tilde\pi_j\) as a sub-bundle of \(\tilde\pi \iso \Delta^*\left(\prod_{i=1}^{2m} \tilde \pi_i'\right)\).
    
    Therefore, as long as \((V,\calT)\) is regular, we do not really need to restrict ourselves to only Gale dual configurations lying in the lattice \(K\). In fact, for any other choice, we obtain a different decomposition whose factors inject into the original one. Hence, their structure can be investigated by “transferring back'' the properties of the whole \((\C^*)^m\)-bundle, which can be understood more easily by choosing a Gale dual configuration in \(K\), for example a good one. Clearly, the situation dramatically changes when \(\Rel(V)\) is not rational: in this case \(\Rel(V) \cap \Z^n\) is not full-ranked, hence it admits no basis whose vectors form a Gale dual configuration.
\end{remark}

Decomposing \(\tilde\pi\) into \(\C^*\)-principal components gives an insight into how the constructions of toric manifolds discussed in the previous section diverge. Indeed, such splitting exists for both constructions, but the Audin--Cox one contains fewer \(\tilde F_j\)'s components. This suggests that some of them are overabundant, and they are collapsed when reducing the construction described by Battaglia and Zaffran to the one of Audin and Cox by removing ghost vectors. For good Gale dual configurations we can prove a stronger result.

\begin{proposition}
    \label{cor:bundle_triviality}
    Let \((V_\Sigma,\calT_\Sigma)\) be a regular, odd and balanced triangulated vector configuration. Suppose it contains \(k \geq 2\) ghost vectors. Let \((f_0,\dots,f_{2m})\) be a \emph{good} Gale dual configuration with \(f_0,\dots f_{k-1}\) as a basis of complement \(\Rel(\Sigma)\) in \(\Rel(V)\). Consider \(\tilde\pi_j : \tilde F_j \to X_\Sigma\) as the principal \(\C^*\)-bundles constructed as above, then, \(\tilde\pi_j\) is trivial for \(j=1, \dots, k-1\).
\end{proposition}
\begin{proof}
    As we are considering a good Gale dual configuration, it is of the form \((f_0,\dots,f_{2m}) = (f_0,\dots, f_{k-1}, \iota(h_1),\dots, \iota(h_{2m-k+1}))\), with \(\{h_1,\dots,h_{2m-k+1}\}\) a basis of \(K_\Sigma\). For any \(j\), we have the diffeomorphism
    \begin{align*}
        X_\Sigma \times \C^* &\longrightarrow \tilde F_j = \frac{\proj((\C^*)^{k} \times U_\Sigma)}{\prod_{\substack{i=1\\i\neq j}}^{k-1} \C^*_{f_i} \times \prod_{i=1}^{2m-k+1} \C^*_{\iota(h_i)}}\\
        ([(z_1, \dots, z_\ell)], t) & \longmapsto [(1: w_2: \dots: w_k: z_1: \dots: z_\ell)]
    \end{align*}
    with
    \begin{equation*}
        w_i =
        \begin{cases}
            t & \text{if } i=j+1\\
            1 & \text{else}.
        \end{cases}
    \end{equation*}
    This map is \(\C^*\)-equivariant. Indeed, since each \(f_j\) is chosen such that its last components lie in the span of \(\iota(h_j)\)'s, the action of \(\C^*_{f_j}\) only moves the \((j+1)\)\th component, modulo the simultaneous action of the smaller torus \(\prod_{i=1}^{2m-k+1} \C^*_{\iota(h_i)}\). Therefore, these maps trivialize the bundles.
\end{proof}
This fact shows that good Gale dual configurations can be interpreted as those for which the construction of LVMB manifolds splits. For other choices (namely, for choices not containing a basis of \(K_\Sigma\)), this property breaks as the change of basis discussed in \cref{ChangeOfBundles} may twist the trivial factors by non-trivial ones. 

\subsubsection{The complex structure of the fibers}
LVMB manifolds are endowed with a holomorphic foliation \cite{LopezVerjovsky1997,Meersseman2000}. When the initial configuration is rational such foliation has closed leaves, and the leaf space is a toric variety with at most orbifold singularities \cite{MeerssemanVerjovsky2004}. We adapt here to our purposes the description  in \cite{BattagliaZaffran2015}.

The foliated structure of an LVMB manifold descends from the actions \labelcref{equation:LVMBaction,equation:toricQuotientAction} defined before: the first gives the LVMB manifold \(N\) as quotient of \(\proj(U(\calT_\Sigma))\), while its quotient by the second action returns, instead, the underlying toric manifold. Since these two actions commute, we can consider the residue of the second one on \(N\). Such action is holomorphic and induces a foliation on \(N\) with equidimensional leaves.

Explicitly, consider a triangulated vector configuration \((V_\Sigma,\calT_\Sigma)\) as above with underlying smooth complete fan \(\Sigma\); for this datum, choose any Gale dual configuration \((f_0,\dots, f_{2m})\). Taking \(\Pi_\text{std}\) as period matrix, we can check that the linear map \(\C^m \to \C^{2m}\) given by \(\underline{w} \mapsto \trans{\Pi}_\text{std} \underline{w}\) induces a splitting of \(\C^{2m}\) that can be chosen as follows
\begin{equation}
    \label{std_decomposition}
    \begin{alignedat}{3}
        \C^{2m} &\iso \hspace{1.5ex} \C^m_{N} &&\oplus \hspace{2.5ex} \C^m_{\mathcal{F}}\\
        \begin{bmatrix}
            \underline{u}\\\underline{v}
        \end{bmatrix}
         &=
         \begin{bmatrix}
            -i\underline{v}\\\underline{v}
         \end{bmatrix}
        &&+
        \begin{bmatrix}
            \underline{u} + i \underline{v}\\
            0
        \end{bmatrix}.
    \end{alignedat}    
\end{equation}
More generally, we can do the same with respect to any period matrix \(\Pi = \Pi_\text{std} A \eqqcolon [P|I_m]\) with \(A \in \operatorname{GL}(2m,\R)\). In this case, this decomposition can be written as
\begin{equation*}
    \begin{bmatrix}
        \underline{u}\\\underline{v}
    \end{bmatrix}
    =
    -i[0|\trans{\Pi}] \trans{A}^\inv 
    \begin{bmatrix}
        \underline{u}\\\underline{v}
    \end{bmatrix}
    +
    \begin{bmatrix}
        i I_m & -I_m\\
        i \Re(\trans{P}) & - \Re(\trans{P})
    \end{bmatrix}
    \trans{A}^\inv
    \begin{bmatrix}
        \underline{u}\\\underline{v}
    \end{bmatrix}
\end{equation*}
In any case, the restriction of the action \(\beta\) to \(\C^m_N\) coincides with \(\alpha\), \(\faktor{\proj(U(\calT_\Sigma))}{\C^m_N} \iso N\). On the other hand, this decomposition shows that the residue of \(\beta\) on \(N\) acts on the representatives of the classes in \(N\) as \(\beta\) does when restricted to \(\C^m_\mathcal{F}\).

For \(\Pi = \Pi_\text{std}\), the expression of the action of \(w \in \C^m_\mathcal{F}\) on a generic point in \([z_1: \dots: z_n]_{N} \in N\) can be written as follows:
\begin{equation}
    \label{residualAction}
    w.[z_1: \dots: z_n]_{N} = \left[e^{2\pi i\sum_{i=1}^m w_i f_i^1} z_1: \dots : e^{2\pi i\sum_{i=1}^m w_i f_i^n} z_n\right]_{N};
\end{equation}
where \([z_1:\dots:z_n]_N\) denotes \(\alpha\)-classes. For different choices of \(\Pi\), a substantially similar expression can be obtained by performing a change of basis. When the initial configuration is regular, this action is proper \cite{BattagliaZaffran2015} and each point is fixed by a lattice \(\Psi \subset \C^m\) generated by \(f_1,\dots,f_{2m}\). Hence, the residual action \(\beta/\alpha\) in \labelcref{residualAction} defines a holomorphic principal torus bundle \(\pi : N \to X_\Sigma\) with the complex torus \(\cpxTorus{m} = \C^m / \Psi\) as fiber \cite{MeerssemanVerjovsky2004}. Actually, its complex structure changes accordingly with \(\Pi\), thus, it can be prescribed by tweaking either the period matrix or the Gale dual configuration chosen at the beginning of the construction. Therefore, these choices affect the complex structure of the manifold \(N\) as it was already proved in \cite{Meersseman2000}.

\subsubsection{The bundle \texorpdfstring{\(\pi\)}{pi} as a product of \texorpdfstring{\(S^1\)}{S1} bundles}
\label{section:CircleBundles}

So far, we have discussed how a regular LVMB manifold \(N\) is related to two holomorphic bundles: \(\tilde\pi : \proj(U(\calT)) \to X_\Sigma\) and \(\pi : N \to X_\Sigma\). Moreover, in \cref{section:splitting_of_pi} we showed how the choice of a Gale dual configuration in the construction determines a splitting of \(\tilde\pi\) into its \(\C^*\)-principal components, and how different choices produce isomorphic decompositions.

In this section we investigate how \(\tilde\pi\) and \(\pi\) are related, this will allow transferring the splitting of the former to the latter. This will decompose \(\pi\) in \(2m\) principal \(S^1\)-bundles.
To delve the relation between \(\tilde\pi\) and \(\pi\) we check how the corresponding principal actions \(\alpha\) and \(\beta\) interact. This can be done rewriting the residual action \(\beta/\alpha \action N\) in a real form. As we observed in \cref{ChangeOfBundles}, as long as the triangulated vector configuration is regular, we can consider any Gale dual configuration \((f_0,\dots,f_{2m})\) not necessarily in \(\Rel(V) \cap \Z^n\); thus we can assume to be using \(\Pi_\text{std}\) as period matrix. In this case, according to the decomposition \labelcref{std_decomposition}, we can identify \(\C^m_\mathcal{F}\) with \( \R^{2m}\) via \(w = x + i y\) with \(x,y \in \R^m\), hence the real form of \(\beta/\alpha \action N\) is
\begin{equation*}
    (x,y).[z_1 : \dots : z_n]_{N} = \left[e^{2\pi i\left(\sum_{i=1}^m x_i f_i^1 + \sum_{i=1}^m y_i f_{m+i}^1 \right)} z_1 : \dots : e^{2\pi i\left(\sum_{i=1}^m x_i f_i^n + \sum_{i=1}^m y_i f_{m+i}^n \right)} z_n \right]_{N}.
\end{equation*}

Thus, the same ideas discussed in \cref{section:splitting_of_pi} can be applied to \(\pi\) in order to split it into \(S^1\)-principal components. Indeed, the \(j\)\th factor of the real \(2m\)-dimensional torus \(\realTorus{2m}\) acts on \(N\) via \(f_j\), and these actions commute as well; so with the same abuse of notation of \cref{section:splitting_of_pi}, we can see \(\realTorus{2m}\) as \(\prod_{j=1}^{2m} S^1_{f_j}\).

Therefore, the partial quotient by the subtorus \(\realTorus{}_j = \faktor{\realTorus{2m}\!\!\!}{S^1_{f_j}} \iso \prod_{i \neq j} S^1_{f_i}\) is smooth as we are restricting a free action to a compact subgroup. For the same reason, the projection onto the \(S^1_{f_j}\)-orbits \(\pi_j : F_j \dfn \faktor{N_\Sigma}{\realTorus{}_j} \to X_\Sigma\) is a principal \(S^1\)-bundle and
\begin{equation*}
    \frac{F_j}{S^1_{f_j}} = \frac{\faktor{N}{\realTorus{}_j}}{\faktor{\realTorus{2m}}{\realTorus{}_j}} \iso \frac{N_\Sigma}{\realTorus{2m}} \iso X_\Sigma.
\end{equation*}

Similarly to the decomposition for \(\tilde \pi\), we have \(\pi \iso \Delta^*\left(\prod_{i=1}^{2m} \pi_i\right)\). Actually, \(\pi\) and \(\tilde \pi\) are related as it is proved in the following result.

\begin{proposition}
    \label{prop:bundleReduction}
    Let \(\Sigma\) be a smooth rational polyhedral fan, \((V_\Sigma, \calT_\Sigma)\) an associated triangulated configuration and \(N\) the LVMB manifold corresponding to the choice of the Gale dual configuration \((f_0,\dots,f_{2m})\). Then, \( \pi : N \to X_\Sigma\) is the bundle reduction of \(\tilde \pi : \proj(U(\calT_\Sigma)) \to  X_\Sigma\) along the canonical inclusion of the fibers \(\realTorus{2m} \hookrightarrow (\C^*)^{2m}\).
\end{proposition}
\begin{proof}
    Since we have no condition on the \(f_j\)'s, because of \cref{ThreeOutOfFour}, we can assume that in the construction of \(N\) the period matrix is \(\Pi_\text{std}\). Therefore,  any \((u,v) \in \C^{m} \oplus \C^{m}\) crosses the isomorphisms \(\C^{2m} \overnumber{\iso}{1} \C^m_N \oplus \C^m_\mathcal{F} \overnumber{\iso}{2} \C^m_N \oplus \R^{2m}\) described in \cref{std_decomposition}  as
    \begin{equation*}
        \begin{bmatrix}
            u\\v
        \end{bmatrix}
        \overnumber{=}{1}
        \begin{bmatrix}
            -iv\\v
        \end{bmatrix}
        +
        \begin{bmatrix}
            u+iv\\0
        \end{bmatrix}
        \overnumber{=}{2}
        \begin{bmatrix}
            -iv + i \Im(u+iv)\\
            v - \Im(u+iv)
        \end{bmatrix}
        +
        \begin{bmatrix}
            \Re(u) - \Im(v)\\
            \Re(v) + \Im(u)
        \end{bmatrix}
        \nfd
        \begin{bmatrix}
            \zeta\\
            i\zeta
        \end{bmatrix}
        +
        \begin{bmatrix}
            x\\y
        \end{bmatrix}
        \in \C^m_{N} \oplus \R^{2m}.
    \end{equation*}
    The quotient by the action of the first component gives \(N\), while the second gives the bundle \(\pi : N \to X_\Sigma\). Hence, for any \([\underline{z}] = (z_1 : \dots : z_n)\in \proj(U(\calT_\Sigma))\), the action
    \begin{align*}
        (u,v).[\underline{z}] &= \left(e^{2\pi i\left(\sum_{i=1}^m u_i f_i^1 + \sum_{i=1}^m v_i f_{m+i}^1 \right)} z_1 : \dots : e^{2\pi i\left(\sum_{i=1}^m u_i f_i^n + \sum_{i=1}^m v_i f_{m+i}^n \right)} z_n \right)
        \intertext{can be written by means of \(\zeta, x\) and \(y\) as follows:}
        (\zeta, i\zeta).(x,y).[\underline{z}] &= \left(\cdots : e^{2\pi i\sum_{i=1}^m \zeta_i (f_i^j +i f_{m+i}^j)} e^{2\pi i\left(\sum_{i=1}^m x_i f_i^j + y_i f_{m+i}^j\right)} z_j : \cdots \right).
    \end{align*}
    Since the elements of the form \((\zeta, i\zeta)\) act trivially on \(N\), the above expression projects to
    \begin{equation*}
        (x,y).[\underline{z}]_{N} = \left[\cdots :e^{2\pi i\left(\sum_{i=1}^m x_i f_i^j + y_i f_{m+i}^j\right)} z_j : \cdots \right]_{N},
    \end{equation*}
    which is nothing but the pullback of \(\beta\) along the natural inclusion \(\iota_{\R^{2m}} : \R^{2m} \hookrightarrow \C^{2m}\), or, equivalently, the real expression of the residual action \(\beta/\alpha \action N\). In other words, \(\iota_{\R^{2m}}\) induces the inclusion of principal bundles represented in the following diagram
    \begin{equation*}
        \begin{tikzcd}
            \realTorus{2m} \ar[rr, hook,"\exp \iota_{\R^{2m}}"] \ar[d, hook, "i_x"] && (\C^*)^{2m} \ar[d, hook,"\tilde\imath_x"] \ar[ll, bend right, "\exp r"']\\
            N \ar[dr,"\pi"'] \ar[rr, "(\exp\iota_{\R^{2m}})_{\sharp}"] && \proj(U(\calT)) \ar[dl, "\tilde\pi"]\\
            & X_\Sigma
        \end{tikzcd}
    \end{equation*}
    where \(i_x\) and \( \tilde\imath_x\) are the inclusion of the fiber in \(x \in X_\Sigma\) of the respective bundles, \(r: (u,v) \mapsto (\Re(u)-\Im(v), \Re(v) + \Im(u))\) is a deformation retraction for \(\iota_{\R^{2m}}\) and \((\exp\iota_{\R^{2m}})_{\sharp}\) the corresponding map induced on the bundles. As \((\C^*)^{2m} = (S^1)^m \times (\R^\times_{>0})^{2m}\), the coset bundle \(\tilde\pi/\pi : \proj(U(\calT))/N \to X_\Sigma\) is \( (\R^\times_{>0})^{2m}\)-principal; as this is a contractible Lie group, such bundle admits a section. Therefore, \(\pi \hookrightarrow\tilde\pi\) is a bundle reduction.
\end{proof}

The fact that \(\pi_j \hookrightarrow \tilde \pi_j\) is a bundle reduction holds true for any choice of the period matrix.
Indeed, by \cref{ThreeOutOfFour}, for any Gale dual configuration \((f_0',f_1',\dots,f_{2m}')\) and any period matrix \(\Pi\) we can find another Gale dual configuration \((f_0,f_1,\dots,f_{2m})\) such that \({\Lambda'}^\R \Pi = \Lambda^\R \Pi_\text{std}\). These choices induce two different decompositions of \(\tilde\pi\) in \(\C^*\)-bundles \(\tilde\pi \iso \Delta^*\left(\prod_{i=1}^{2m} \tilde\pi_i'\right) \iso \Delta^*\left(\prod_{i=1}^{2m} \tilde\pi_i\right)\); similarly, \(\pi \iso \Delta^*\left(\prod_{i=1}^{2m} \pi_i'\right) \iso \Delta^*\left(\prod_{i=1}^{2m} \pi_i\right)\). With the same argument of \cref{ChangeOfBundles}, we can show that the inclusions of the factors are compatible with the bundle reduction \(\pi \hookrightarrow \tilde\pi\). Thus, we have a commutative square
\begin{equation*}
    \begin{tikzcd}
        \pi'_j \ar[d,hook] \ar[r,hook] & \pi \iso \Delta^*\left(\prod_{i=1}^{2m} \pi_i\right) \ar[d,hook]\\
        \tilde\pi'_j \ar[r,hook]& \pi \iso \Delta^*\left(\prod_{i=1}^{2m} \tilde\pi_i\right)
    \end{tikzcd}
\end{equation*}
whose bottom map is induced projecting the fiber \(\realTorus{2m}\) to its subgroup \(S^1_{f_j'}\). Hence, \(\pi' \hookrightarrow \tilde\pi'\) is an inclusion of principal bundles, and repeating the argument at the end of \cref{prop:bundleReduction}, we deduce that it is a bundle reduction.

This observation shows that, independently of the choices in the construction, the principal bundles \(\pi_j\) and \(\tilde\pi_j\) have the same representations, thus their associated line bundles  \(\pi_j \times_{\id} \C \) and \(\tilde \pi_j \times_{\id} \C\) are isomorphic.

\subsection{Adding ghosts}
A natural question about the LVMB construction is how adding or removing ghost vectors to a triangulated vector configuration affects the resulting LVMB manifolds. Of course any answer must take into account the consequent choices of Gale dual configurations.

Such problem was first addressed by Meersseman \cite[Thm.12 and Corollary]{Meersseman2000}. There, it is proven that, whenever it is possible, removing two \emph{indispensable points} from an LVMB datum containing at least three of them, affects the topology of \(N\) as a deletion of a factor \(\realTorus{2}\). As we remarked in \labelcref{indispensableRemark}, indispensable points correspond to ghost vectors by Gale duality. Thus, reading his result backwards, we can see that adding two ghost vectors to a configuration containing at least one of them, modifies the underlying smooth manifold by a product with \(\realTorus{2}\). Nonetheless, the new LVMB manifold might not carry the product complex structure.

Here we prove a result intersecting Meersseman's which gives a partial answer to this problem.

\begin{proposition}
    \label{ghostAddition}
    Let \((V,\calT)\) be a regular odd and balanced configuration with a choice of a Gale dual configuration \((\Lambda_1^\R,\dots, \Lambda_{n}^\R)\), and let \(N\) be the corresponding LVMB manifold. Consider the new configuration \((V',\calT')\) obtained from \((V,\calT)\) by prepending two ghost vectors \(v_1' = v_2' = 0\). Denote as \(n', m', k'\) its integer constants, in particular \(k' = k+2\), \(m' = m+1\) and \(n' = n+2\). Then, for a suitable choice of the Gale transform and period matrix, the corresponding LVMB manifold \(N'\) is a holomorphic principal \(\cpxTorus{1}\)-bundle over \(N\).

    Moreover, if \((V,\calT)\) contains at least one ghost vector, then \(N' \iso N \times \cpxTorus{1}\) as complex manifolds.
\end{proposition}
\begin{proof}
    The new vector configuration can be written as \(V' = (v_1', v_2', v_1, \dots, v_n)\); consider as basis of \(\Rel(V')\) the columns of the matrix
    \begin{equation*}
        M' =
        \left[
            \begin{array}{cc}
                1 & 0\\
                1 & 1\\
                1 & 0\\
                \vdots & \vdots\\
                1 & 0
            \end{array}
        \right|\left.\vphantom{\begin{array}{cc}
                1\\
                1\\
                1\\
                \vdots\\
                1
            \end{array}}
            \begin{array}{ccc}
                0 & \cdots & 0\\
                0 &\cdots & 0\\
                \hline
                \vspace{6.5pt}\vertbar & & \vertbar\\
                f_0 & \cdots & f_{2m}\\
                \vertbar & & \vertbar
            \end{array}
        \right]
    \end{equation*}
    where the bottom-right block contains a basis of \(\Rel(V)\). The corresponding point configuration in \(\mathbb{A}^{2m'}_{\R}\) is then
    \begin{equation*}
        {\Lambda'}_1^\R = 0, \quad {\Lambda'}_2^\R = [1,0, \cdots, 0], \quad {\Lambda'}_j^\R = [0, 1, \horzbar \Lambda_{j-2}^\R \horzbar] \text{ for } j=3,\dots,n',
    \end{equation*}
    where \(\Lambda^\R\) is the Gale transform of \((V,\calT)\) from which \(N\) is generated. To obtain the new configuration of complex points \(\Lambda'\), we take \(\tau \in \C\) such that \(\Im\tau > 0\) and we consider the period matrix
    \begin{equation*}
        \Pi' = 
        \begin{bNiceArray}{cc|ccc}[margin]
            1 & \tau & 0 &\Cdots&0 \\
            \hline
            0 & 0 & \Block{3-3}{\Pi} \\
            \Vdots & \Vdots\\
            0 & 0
        \end{bNiceArray},
    \end{equation*}    
    containing the one used in the construction of \(N\) as bottom-right block. With respect to this choice, the points \(\Lambda'_j = {\Lambda'}^{\R}_j \trans{\Pi}'\) are of the form:
    \begin{equation*}
        \Lambda'_1 = 0, \quad \Lambda'_2 = [1,0, \dots, 0], \quad \Lambda'_{j} = [\tau, \horzbar \Lambda_{j-2} \horzbar] \text{ for } j= 3,\dots,n'.
    \end{equation*}
    Notice that this choice for \(\Pi'\) is very natural as this way the pair of real coordinates we added matches together as the first complex coordinate of \(\mathbb{A}^{m'}_{\C}\).

    Consider \(U(\calT) \subset \C^n\) and \(U(\calT)' \subset \C^{n'}\). Since we added two zeroes as ghosts, we have \(U(\calT)' = (\C^*)^2 \times U(\calT)\). So we can define the map
    \begin{align*}
        \Phi : \proj((\C^*)^2 \times U(\calT)) &\longrightarrow \proj(U(\calT))\\
        (z_0: z_1 : \dots : z_{n+1}) &\mapsto (z_2 : \cdots : z_{n+1}),
    \end{align*}
    which is surjective and equivariant with respect to the actions \(\alpha'\) and \(\alpha\) defined on \(U(\calT)\) and \(U(\calT)'\) respectively. Since this is also holomorphic, it descends to a holomorphic map \(\tilde\Phi : N' \to N\).

    To prove this map defines a principal bundle consider the action defined on classes of \(N'\) as
    \begin{align*}
        a : \C\times N' &\longrightarrow N'\\
        (u, [1: z_1 :  \dots z_{n+1}]_{N'}) &\longmapsto [1:e^{2\pi i u} z_1 : z_2 : \dots : z_n]_{N'}.
    \end{align*}
    Since \([1:e^{2\pi i u} z_1 : z_2 : \dots : z_n]_{N'} = [1 : z_1 : e^{-2\pi i \tau u} z_2 : \dots : e^{-2\pi i\tau u} z_{n}]_{N'}\), each point is fixed by \(\Z \oplus {\tau}\Z \subset \C\), hence it defines a free action of the complex torus \(\cpxTorus{1}\) of period \(\tau\). Moreover, it is holomorphic and preserves the fibers of \(\tilde{\Phi}\), hence \(\tilde\Phi\) is a principal \(\cpxTorus{1}\)-bundle over \(N\).

    To prove the second part of the statement, suppose \((V,\calT)\) contains at least one ghost vector. Under this hypothesis, \(U(\calT) = \C^* \times \tilde U\) and \(U(\calT') = (\C^*)^3 \times \tilde U\) for some open subset \(\tilde U \subset \C^{n-1}\). Hence, the natural injection \(\tilde U \hookrightarrow (\C^*)^2 \times \tilde U\) can be extended to a well-defined holomorphic map
    \begin{align*}
        s:\proj(\C^* \times \tilde U) = \proj(U(\calT)) &\longrightarrow \proj((\C^*)^2 \times U(\calT)) = \proj((\C^*)^3 \times \tilde U)\\
        (1:z_2:\cdots:z_n) &\longmapsto (1:1:1:z_2:\dots : z_n),
    \end{align*}
    which is a section for \(\Phi\). Moreover, \(s\) is compatible with \(\alpha\) and \(\alpha'\), hence it descends to a holomorphic section for \(\tilde \Phi : N' \to N\). As this bundle is \(\cpxTorus{1}\)-principal, it must be trivial.
\end{proof}
\begin{remark}
    In the proof of the last part of the statement, the existence of a ghost vector in \((V,\calT)\) is crucial. Indeed, this allows the identification between \(\proj(U(\calT))\) and \(\tilde U\), which in turn is needed to construct the section \(s\). 
    In fact, as a straightforward application of \cref{DeRhamFirstChernClass}, we will be able to see that, if the initial configuration contains no ghost vector, the torus bundle \(\tilde\Phi: N'\to N\) cannot be holomorphically trivial.
\end{remark}

\section{The characteristic class of a regular LVMB manifold}
\label{section:characteristic_class}
We are now able to study \(N\) as a holomorphic principal torus bundle over a toric manifold. Standard ideas in algebraic topology (see e.g. \cite{May1999,Milnor1956}) prove that these bundles can be given a characteristic class in cohomology which classifies them. On the other hand, LVMB manifolds are realizations of the algebro-combinatoric datum \((V,\calT)\), so we expect to be able to extract this kind of topological information directly from it. Indeed, we provide an explicit formula for the characteristic class of the torus bundle underlying a regular LVMB manifold.

We start by recalling how the characteristic class of a principal torus bundle is constructed in general. These general properties will be applied to the LVMB case in \cref{bundles_over_toric_manifolds}.

\subsection{The characteristic class of a principal torus bundle}
\label{section:characteristic_class_torus_bundle}
The characteristic class of a principal \(G\)-bundle \(\pi : P \to M\) can be essentially defined in two ways: the first uses Chern--Weil theory to associate suitable cohomology classes to the curvature of a principal connection, while the second involves algebraic topology and, in particular, the theory of classifying spaces. Here, we will be using the second approach, so we will denote as \(\pi_G : EG \to BG\) the classifying bundle of the generic topological group \(G\). As a consequence of the Eckmann--Hilton duality, we have a weak homotopical equivalence between \(G\) and the loopspace of \(BG\). This fact reads as
\begin{equation*}
    \pi_n(G) \iso \pi_{n+1}(BG)
\end{equation*}
for any \(n \geq 0\).

If \(G = \realTorus{k}\) is a real torus, it can be written as the quotient \(\R^k/\lattice\) for some full rank lattice \(\lattice \subset \R^k\); immediately, we conclude that \(\pi_1(\realTorus{k}) = \lattice\), as loops can be canonically identified with its generators. Since all the other homotopy groups are trivial, we have that \(BT^k\) is an Eilenberg--MacLane space \(K(\lattice,2)\), hence \(H_2(B\realTorus{k}) \iso \pi_2(B\realTorus{k}) \iso \lattice\) because of Hurewicz theorem, and so \(H^*(B\realTorus{k};\Z)\) is isomorphic to the graded group algebra \(\Z(\lattice^\vee)\) generated in degree 2. If \(\lattice = \bigoplus_{i=1}^k \Z\lambda_i\) and \(x_1, \dots, x_k\) is the dual basis, \(\Z(\lattice^\vee)\) is nothing but the polynomial ring \(\Z[x_1, \dots, x_k]\) where \(\deg x_j = 2\). Therefore, a principal \(\realTorus{k}\)-bundle \(\pi : P \to M\) determines a continuous map \(\Gamma : M \to B\realTorus{k}\), unique up to homotopy equivalences; in turn, this provides a graded ring homomorphism \(\Gamma^* : \Z[x_1, \dots, x_k] \to H^*(M;\Z)\). Being graded, such map is completely determined by the restriction \(\Gamma^*|_2 : \bigoplus_{i=1}^k \Z x_i = \Hom(\lattice,\Z) \to H^2(M;\Z)\). Finally, as \(\lattice\) is free and finitely generated, \(\Hom(\Hom(\lattice,\Z), H^2(M,\Z)) \iso \lattice \otimes H^2(M;\Z)\), whence we conclude that torus bundles over \(M\) are classified by this group. More details can be found in \cite{Hofer1993} where this result is discussed via Čech cohomology instead.

\begin{definition}
    We denote the image of \(\Gamma^*|_2\) under this identification as \(\gamma(\pi) \in \lattice \otimes H^2(M;\Z)\). We refer to this object as the \emph{characteristic class} of \(\pi : P \to M\).
\end{definition}
\begin{remark}
    \label{charClassIdentified}
    If \(H_1(M)\) has no torsion (which is always the case of compact toric manifolds that are indeed simply connected \cite[Thm. 12.1.10]{CoxLittleSchenck2011}), the Universal Coefficients Theorem gives the isomorphism \(H^2(M;\Z) \iso \Hom(H_2(M),\Z)\) of finitely generated Abelian groups. Thus, \(\lattice \otimes H^2(M;\Z)\) can be identified in three ways, namely
    \begin{equation}
        \label{canonical_identifications}
        \Hom(\lattice^\vee, H^2(M;\Z)) \iso \lattice \otimes H^2(M;\Z) \iso \Hom(H_2(M), \lattice);
    \end{equation}
    notice that the first is the one we used to pass from \(\Gamma^*|_2\) to \(\gamma(\pi)\). As these identifications are canonical, \(\gamma(\pi)\) is associated to three \emph{apparently} different objects. However, to keep the notation light, we will not distinguish the images of \(\gamma(\pi)\) via the isomorphisms above, thus referring to all of them as ``characteristic class of \(\pi\)''. Indeed, no confusion has to be feared as the identification we are subtending will be always clear from the context.
\end{remark}

The expression of \(\gamma(\pi)\) can be written explicitly by means of the Chern classes of the components of \(\pi\). Indeed, the choice of a basis for \(\lattice\) induces \(\realTorus{k} \iso \U(1)^k\) as Lie groups, and \(\U(1)^k\) naturally embeds in \(\U(k)\) as maximal torus with the natural inclusion \(i : \U(1)^k \hookrightarrow \U(k)\).

As \(B\) functorially maps groups to topological spaces, \(i\) induces a continuous map on classifying spaces \(Bi : B(\U(1))^k \to B\U(k)\), and a ring homomorphism in cohomology
\begin{equation*}
    Bi^* : H^*(B\U(k);\Z) \to H^*(B(\U(1)^k);\Z).
\end{equation*}
As \(B\) commutes with products \cite[129]{May1999},  by the Künneth Formula we have
\begin{equation*}
    H^*(B(\U(1)^k);\Z) \iso H^*((B\U(1))^k;\Z) \iso H^*(B\U(1);\Z)^{\otimes k} \iso \Z[x_1, \dots, x_k].
\end{equation*}
On the other hand, \(B\U(k)\) is the infinite dimensional complex Grassmannian \(\operatorname{Gr}(k, \C^\infty)\) of rank \(k\). It is a standard result that \(H^*(B\U(k);\Z) \iso \Z[c_1, \dots, c_k]\) (see e.g. \cite{MilnorStasheff1974}), where \(c_j\) is the \(j\)\th Chern class of the tautological bundle of \(\operatorname{Gr}(k, \C^\infty)\); in particular, \(\deg c_j = 2j\). As an equivalent formulation of the splitting principle for complex vector bundles, we have that \(Bi^*(c_j) = \sigma_j\), where \(\sigma_j\) denotes the invariant polynomial of degree \(j\) in the variables \(x_1, \dots, x_k\). Therefore,
\begin{equation}
    \label{eq:Chern1}
    c_1(\pi) = (\Gamma^* \circ Bi^*)(c_1) = \Gamma^*(\sigma_1) = \Gamma^*\left(\textstyle\sum_{j=1}^{k} x_j\right).
\end{equation}
This Chern class can be related to the Chern classes of each component: indeed, since \(B\realTorus{k} \iso (B\U(1))^k\), the classifying map \(\Gamma : M \to B\realTorus{k}\) splits into the product of \(k\) maps \(\Gamma_j : M \to B\U(1)\). More precisely, if \(\Delta : M \to \prod_{j=1}^k M\) denotes the diagonal morphism, we have
\begin{equation*}
    \Gamma = \left(\prod_{j=1}^k \Gamma_j\right) \circ \Delta = (\Gamma_1, \dots , \Gamma_k),
\end{equation*}
hence, by duality, \(\Gamma^* = \Delta^* \circ \bigoplus_{j=1}^k \Gamma_j^* = \sum_{j=1}^k \Gamma_j^*\), so \cref{eq:Chern1} can be read as
\begin{equation*}
    c_1(\pi) = \Gamma^*\left(\textstyle\sum_{j=1}^{k} x_j\right) = \textstyle\sum_{j=1}^k \Gamma_j^*(x_j) = \textstyle\sum_{j=1}^k c_1(\pi_j),
\end{equation*}
where \(\pi_j\) is the \(j\)\th component of \(\pi\) determined by the splitting we chose for \(T^k\). This observation makes also explicit the expression of \(\gamma(\pi)\): indeed, if \(\iota_j : \Z[x_j] \hookrightarrow \Z[x_1, \dots, x_k]\) denotes the \(j\)\th canonical coprojection of the polynomial ring, and \(\langle\cdot, \cdot\rangle\) the pairing between \(\lattice\) and \(\lattice^\vee\), we have \(c_1(\pi_j) = \Gamma_j^*(x_j) = (\Gamma^* \circ \iota_j)(x_j) = \left\langle x_j,\gamma(\pi)\right\rangle\); therefore
\begin{equation}
    \label{eq:generalClassExpression}
    \gamma(\pi) = \sum_{j=1}^k c_1(\pi_j) \otimes \lambda_j \in H^2(M;\Z) \otimes \lattice,
\end{equation}
and, in particular, \(c_1(\pi) = \left\langle\textstyle\sum_{j=1}^k x_j, \gamma(\pi) \right\rangle\).

\subsubsection{Naturality}
\label{naturality}
This construction of \(\gamma(\pi)\) in terms of the classifying map makes proving the naturality with respect to principal bundles maps straightforward. Indeed, let \(\pi : P \to M\) and \(\pi' : P' \to M\) be two principal torus bundles with \(T\) and \(T'\) as fibers, and let \(\varphi : P' \to P\) be a map of principal bundles. Consider the classifying maps \(\Gamma : M \to B\realTorus{}\) and \(\Gamma' : M \to T'\). Then, the map induced on the fiber \(\varphi_0 : T' \to \realTorus{}\) fits in the following diagram which commutes in any face
\begin{equation*}
    \begin{tikzcd}[sep = small]
    &|[label={[label distance=-2mm]-45:\lrcorner}]| P \arrow[dd,near start, "\pi"] \arrow[rr] && E\realTorus{} \arrow[dd] \\
    |[label={[label distance=-2mm]-45:\lrcorner}]| P' \arrow[dd, "\pi'"] \arrow[ru, "\varphi"] \arrow[rr,crossing over] && ET' \arrow[ru]    &\\
    & M \arrow[rr, near start,"\Gamma"] && B\realTorus{} \\
    M \arrow[rr, "\Gamma'"] \arrow[ru,equals] && BT' \arrow[ru, "B\varphi_0"'] \arrow[from=uu,crossing over] &                         
    \end{tikzcd}
\end{equation*}
Following the same ideas in the construction of \(\gamma(\pi)\), the bottom square induces a commutative triangle in cohomology
\begin{equation*}
    \begin{tikzcd}
        H^1(T;\Z) \ar[r, "\gamma(\pi)"] \ar[d,"(B\varphi_0)^*"'] & H^2(M;\Z) \\
        H^1(T';\Z) \ar[ur,"\gamma(\pi')"']
    \end{tikzcd},
\end{equation*}
which proves the naturality of \(\gamma\).

Moreover, if \(\lattice,\lattice'\) are the period lattices of \(T\) and \(T'\) respectively, then \((B\varphi_0)^*\) can be identified with a homomorphism \(\phi : \lattice^\vee \to (\lattice')^\vee\) which pulls back the coordinates of \(\lattice\) to \(\lattice'\). This is precisely the adjoint map of the group homomorphism \(\tilde \varphi_0 : \lattice' \to \lattice\) inducing \(\varphi_0\) as homomorphism of tori.

\subsection{\texorpdfstring{\(\C^*\)}{C*}-bundles over toric manifolds}
\label{bundles_over_toric_manifolds}
What we discussed above holds for any principal torus bundle with simply connected base space, in particular this applies to regular LVMB manifolds. Here we show that the algebro-combinatoric datum \((V,\calT)\) also encodes the expression of the characteristic class \(\gamma(\pi)\).%

First, we briefly recall a circle of well-known facts in toric geometry regarding principal \(\C^*\)-bundles and their associated line bundles, for more details see \cite[\S VII.2]{Audin2004} or \cite[Ch.6]{CoxLittleSchenck2011}. Basically, we need to observe that, in the setting of \cref{section:relationWithToricGeometry}, when \(\Sigma\) is smooth, any \(g \in K_\Sigma^\vee\) induces a regular fan \(\Sigma_g\) in \(\R^{d+1}\) whose corresponding toric variety \(X_{\Sigma_g}\) is the total space of a complex line bundle \(\pi_g : X_{\Sigma_g} \to X_\Sigma\); since any line bundle arises this way, this gives an identification between \(K_\Sigma^\vee\) and \(\operatorname{Pic}(X_\Sigma)\). Moreover, by Demazure Vanishing Theorem (see e.g. \cite[Thm.9.2.3]{CoxLittleSchenck2011}), \(H^i(X_\Sigma, \mathcal{O}_{X_\Sigma}) = 0\) for \(i > 0\); hence the first Chern class map is an isomorphism between \(K_\Sigma^\vee\) and \(H^2(X_\Sigma;\Z)\). With respect to this identification, we may write \(c_1(\pi_g) = g\). 

A similar construction exists for principal \(\C^*\)-bundles. Indeed, if \(\{h_1, \dots, h_t\}\) is a basis of \(K_\Sigma\) and \(\{g_1, \dots, g_t\}\) is its dual, any \(g_i\) can be associated to a principal \(\C^*\)-bundle \(\varpi_i : X_{\Sigma_i} \to X_\Sigma\) such that \(\varpi_i \times_{\id} \C \iso \pi_{g_i}\). As \(\C^*\) and \(S^1\) are homotopically equivalent, their classifying spaces are the same, thus we can repeat all the previous considerations to construct a characteristic class for a principal \(\C^*\)-bundle. In particular, we have \(\gamma(\varpi_i) = g_i \otimes h_i\).

This construction is essentially the same as the one of \(\tilde F_j\) in \cref{section:relationWithToricGeometry}, hence \(\varpi_i = \tilde \pi_{i+k-1}\) for \(i = 1, \dots, t\). Therefore, this observation gives an explicit formula for the characteristic class of the torus bundle underlying a regular LVMB manifold.
\begin{lemma}
    \label{lemma:charClassGhosts}
    Let \((V_\Sigma, \calT_\Sigma)\) be a regular triangulated configuration with underlying smooth polyhedral fan \(\Sigma\). Let \(K_\Sigma\) and \(K\) be the lattices of integral linear relations among the vectors in \(V_\Sigma\) as in \cref{section:relationWithToricGeometry}. Consider \(\{h_1,\dots,h_{2m-k+1}\}\) a basis of \(K_\Sigma\) included in a good Gale dual configuration \((f_0, \dots, f_{2m})\) via the inclusion \(\iota : K_\Sigma \hookrightarrow K\). Let \(N\) be the LVMB manifold constructed from these data and \(\pi : N \to X_\Sigma\) the principal \(\realTorus{2m}\)-bundle coming with it. Then,
    \begin{itemize}
        \item if \(V_\Sigma\) contains \(k \geq 1\) ghost vectors,
        \begin{equation*}
            \gamma(\pi) = \sum_{j=1}^{2m-k+1} h_j^* \otimes \iota(h_j) \in K_\Sigma^\vee \otimes K.
        \end{equation*}
        \item If \((V_\Sigma, \calT_\Sigma)\) does not contain ghost vectors,
        \begin{equation*}
            \gamma(\pi) = \sum_{j=2}^{2m+1} h_j^* \otimes h_j \in K_\Sigma^\vee \otimes K_\Sigma;
        \end{equation*} 
        in this case \(\iota : K_\Sigma \hookrightarrow K = K_\Sigma\) is the identity.
    \end{itemize}
\end{lemma}
\begin{proof}
    The proof consists of an application of \cref{eq:generalClassExpression} in the case of a regular LVMB manifold. First we prove the \(k \geq 1\) case. Since we are considering a good Gale dual configuration, it is of the form
    \begin{equation*}
        (f_0,\dots,f_{2m}) = (f_0,f_1,\dots,f_{k-1}, \iota(h_1),\dots, \iota(h_{2m-k+1}))
    \end{equation*}
    for some \(\{h_1,\dots,h_{2m-k+1}\}\) basis of \(K_\Sigma\). Moreover, as we showed in the proof of \cref{goodGaleDualExistence}, \(K\) splits as \(\Z f_0 \oplus \hat K\), and \(\pi : N \to X_\Sigma\) is a principal \(\realTorus{2m}\)-bundle with fiber \(\R^{2m}/\hat K\). As we discussed in \cref{section:CircleBundles}, the choice of any Gale dual configuration induces a decomposition of \(\pi\) into the product of principal \(S^1\)-components \(\pi_j : F_j \to X_\Sigma\); moreover from \cref{cor:bundle_triviality} we deduce that \(\pi_j\) is trivial for \(j=1\dots,k-1\), being the reduction of the trivial \(\C^*\)-bundle \(\tilde\pi_j\). For all the others, we have \(\gamma(\pi_j) = \gamma(\tilde\pi_j) = c_1(\tilde\pi_j \times_{\text{id}} \C) \otimes \iota(h_{j-k+1})\). Since \(\tilde\pi_{j}\) is by construction the \(\C^*\)-principal bundle obtained from \(h^*_{j-k+1} \in K_\Sigma^\vee\), we immediately obtain
    \begin{equation*}
        \gamma(\pi) = \sum_{j=k}^{2m} c_1(\tilde\pi_j \times_{\text{id}} \C) \otimes \iota(h_{j-k+1}) = \sum_{j=1}^{2m-k+1} h_j^* \otimes \iota(h_j).
    \end{equation*}

    To prove the second part of the statement, we assume that \((V_\Sigma,\calT_\Sigma)\) does not contain ghost vectors; in this case we can prepend two of them \(v_1 = v_2 = 0\) to it in order to get a new configuration \((V_\Sigma',\calT_\Sigma')\). As both triangulated vector configurations are regular, we can complete a good Gale dual configuration for \((V_\Sigma,\calT_\Sigma)\) to a good Gale dual configuration for \((V_\Sigma',\calT_\Sigma')\). Namely, we get a configuration of the form
    \begin{equation*}
        (f_0,f_1,\iota(h_1),\dots, h_{2m+1}),
    \end{equation*}
    with \(f_1 = [0,1, \horzbar \hat f_1 \horzbar]\) where \(\hat f_1 \in \Span_{\Z}(h_1,\dots,h_{2m+1})\). Since the new lattice of relations \(K' = \Span_{\Z}(f_0,f_1) \oplus \iota(K)\), we have a natural retraction \(\rho : K' \to K\) of \(\iota\) with adjoint map \(\rho^\vee : K^\vee \to {K'}^\vee\). By \cref{ghostAddition}, we can construct \(N'\) such that \(\varpi : N' \to N\) is a principal \(\cpxTorus{1}\)-bundle, and we denote its bundle projection as \(\pi' : N' \to X_\Sigma\). In this situation, for any \(x \in X_\Sigma\), the map induced on the fibers \(\varpi_x : (\pi')^\inv(x) \to \pi^\inv(x)\) induces a homomorphism in cohomology \(\varpi_x^* : H^1((\pi')^\inv(x);\Z) \to H^1(\pi^\inv(x);\Z)\).
    
    Therefore, to conclude, it is sufficient to observe that the choice of the Gale dual configuration determines both the splitting of \(N\) as product of principal \(S^1\)-bundles and a basis for the homology of the fibers. Thus, because of the observations in \cref{naturality}, \(\varpi^*_x = \rho^\vee\) under the identifications \(K \iso H_1(\pi^\inv(x);\Z)\) and \(K' \iso H_1((\pi')^\inv(x);\Z)\). Hence, by naturality of \(\gamma(\pi)\), we have \(\gamma(\pi) = \rho \circ \gamma(\pi')\). This can be written explicitly with respect to the fixed bases, namely \(\rho = \sum_{j=2}^{2m+1} \iota(h_j)^* \otimes h_j\), thus
    \begin{equation*}
        \gamma(\pi)  = \left(\sum_{j=2}^{2m+1} \iota(h_j)^* \otimes h_j \right) \circ \left(\sum_{j=1}^{2m+1} h_j^* \otimes \iota(h_j) \right) = \sum_{j=2}^{2m+1} h_j^* \otimes h_j. \qedhere
    \end{equation*}
\end{proof}

\begin{remark}
    \label{change_of_lattices}
    By the identifications discussed in \cref{charClassIdentified}, \(\gamma(\pi)\) has to be intrinsic to the triangulated vector configuration datum \((V_\Sigma,\calT_\Sigma)\). Indeed, under the second isomorphism in \cref{canonical_identifications}, we can see that \(\gamma(\pi)\) corresponds to \(\iota : K_\Sigma \hookrightarrow K\); hence, any choice of Gale dual configuration which is an adapted basis for \(\iota\), provides the same characteristic class. Nonetheless, the construction of LVMB manifolds can be performed starting from any Gale dual configuration \((f_0,\dots,f_{2m})\) which does not necessarily lie in \(K = \Rel(V) \cap \Z^n\). However, as long as \((V_\Sigma,\calT_\Sigma)\) is regular, we can find another full rank lattice \(\Xi\) in \(\R^n\) such that \(\Rel(V_\Sigma) \cap \Xi\) is generated by \(\{f_0,\dots,f_{2m}\}\). In this case, the diagram \labelcref{equation:AudinExactSequence} can be rewritten as
    \begin{equation*}
        \label{equation:skewExactSequence}
        \begin{tikzcd}
            0 \ar[r] & K_\Sigma \ar[r] \ar[d,"\tilde\iota|_{K_\Sigma}"] & \Xi \cap E \ar[r,"\tilde p'"] \ar[d,"\iota"] & \Z^d \ar[r] \ar[d,equals]& 0\\
        0 \ar[r] & K \ar[r] & \Xi \ar[r,"\tilde p"]& \Z^d \ar[r] & 0
    \end{tikzcd}
\end{equation*}
where \(E \subseteq \R^n\) is an \((n-k)\)-dimensional linear subspace; explicitly, if \(\{\xi_1,\dots,\xi_n\}\) is a basis of \(\Xi\), \(E\) will be the subset spanned by a subset \(\{\xi_{i_1},\dots, \xi_{i_{n-k}}\}\) containing the generators of \(\Xi\) which are not mapped by \(\tilde p\) to a ghost vector in \(V_\Sigma\). Actually, this is the same construction seen through a change of basis, hence \(\hat K\) is still canonically identified with \(\pi_1(\realTorus{2m})\), and \(K_\Sigma^\vee\) with \(H^2(X_\Sigma;\Z)\).
\end{remark}

By naturality, the characteristic class \(\gamma(\pi)\) does not depend on the choice of a basis. This yields observing that the formula for \(\gamma(\pi)\) given in \cref{lemma:charClassGhosts} can be written with respect to any choice of Gale dual configuration.

\begin{theorem}
    \label{them:father}
    Let \((V_\Sigma, \calT_\Sigma)\) be a regular triangulated configuration with underlying fan \(\Sigma\). Let \((f_0, f_1, \dots, f_{2m})\) be any Gale dual configuration for this datum and \(K\) be the lattice spanned by these vectors. Then, with \(N\) denoting the corresponding LVMB manifold, the characteristic class of the principal torus bundle \(\pi : N \to X_\Sigma\) is
    \begin{equation*}
        \gamma(\pi) = \sum_{j=1}^{2m} \iota^\vee(f_j^*) \otimes f_j \in K_\Sigma^\vee \otimes K,
    \end{equation*}
    where \(K_\Sigma\) and \(K\) are the lattices described above and \(\iota^\vee : K^\vee \to K_\Sigma^\vee\) is the adjoint map of the natural inclusion.
\end{theorem}
\begin{proof}
    The proof consists in showing how the expression of \(\gamma(\pi)\) transforms under the change of basis between two different choices for the Gale dual configuration.
    
    As before, we first assume that \(k \geq 1\). Since the configuration is regular, there exists a good Gale dual configuration of the form
    \begin{equation*}
        (f_0,f_1,\dots,f_{k-1},\iota(h_1),\dots,\iota(h_{2m-k+1})) \eqqcolon (f_0, f_1',\dots,f_{2m}').
    \end{equation*}
    Let \(\Phi \in \operatorname{Aut}(\Rel(V))\) be a change of basis in \(\Rel(V)\) such that \(\Phi(f_0) = f_0\) and \(\Phi(f_j) = f_j'\) for any \(j\); denote as \(\Phi_{i\ell}\) the corresponding matrix with respect to the given bases. If \((f_0^*, {f_1'}^*,\dots,{f_{2m}'}^*)\) is its dual basis, it can be checked that \(h_j^* = \iota^\vee(f_{j+k-1}^*)\) for all \(j=1,\dots, 2m-k+1\), thus the characteristic class can be written as
    \begin{equation*}
        \gamma(\pi) = \sum_{j=1}^{2m-k+1} h_j^* \otimes \iota(h_{j}) = \sum_{j=k}^{2m} \iota^\vee({f_j'}^*) \otimes f_{j}' = \sum_{j=1}^{2m} \iota^\vee({f_j'}^*) \otimes f_{j}',
    \end{equation*}
    since the first terms in the sum vanish. Thus, as \(f_j' = \sum_\ell \Phi_{j\ell} f_\ell\) and \(f_j^* = \sum_\ell {f_\ell'}^* \Phi_{\ell j}\), we have
    \begin{equation*}
        \gamma(\pi) = \sum_{j=1}^{2m} \sum_{\ell = 1}^{2m} \iota^\vee({f_j'}^*) \otimes \Phi_{j\ell}  f_{\ell} = \sum_{\ell = 1}^{2m} \sum_{j=1}^{2m} \iota^\vee( \Phi_{j\ell} {f_j'}^*)  \otimes f_\ell = \sum_{\ell=1}^{2m} \iota^\vee(f_\ell^*) \otimes f_\ell.
    \end{equation*}

    If \(k = 0\), \(\iota\) is the identity, hence in this case any Gale dual configuration is of the form
    \begin{equation*}
        (f_0,\dots,f_{2m}) = (h_1,\dots,h_{2m+1}) 
    \end{equation*}
    with \(h_1 = [1,\dots,1]\). Therefore, there exists a unimodular automorphism \(\Phi\) such that the configuration \((h_1, h_2' ,\dots,h_{2m+1}')\) with \(h_j' = \Phi(h_j)\) is a good Gale dual configuration. With respect to these bases, \(\Phi\) is represented by the matrix
    \begin{equation*}
        [\Phi_{ij}] =
        \begin{bNiceMatrix}
            1 & 0 &\Cdots & 0\\
            \Phi_{2,1} & 1 &&\\
            \Vdots && \Ddots\\
            \Phi_{2m+1,1} && & 1
        \end{bNiceMatrix},
        \text{ with inverse }
        [\Phi^{\inv}_{ij}] =
        \begin{bNiceMatrix}
            1 & 0 &\Cdots & 0\\
            -\Phi_{2,1} & 1 &&\\
            \Vdots && \Ddots\\
            -\Phi_{2m+1,1} && & 1
        \end{bNiceMatrix}
    \end{equation*}
    Hence, we have \(h_j' = \sum_{\ell=1}^{2m+1} \Phi_{j\ell} h_\ell'\) and \({h_j'}^* = \sum_{i=1}^{2m+1} h_{i}^* \Phi_{ij}^\inv\), thus the change of basis can be written as
    \begin{equation*}
        \gamma(\pi) = \sum_{j=2}^{2m+1} {h_j'}^* \otimes h_j' = \sum_{j=2}^{2m+1}  \sum_{i,\ell=1}^{2m+1}  h_i^* \Phi_{ij}^\inv \otimes \Phi_{j\ell} h_\ell = \sum_{\ell=2}^{2m+1} h_\ell^* \otimes h_\ell = \sum_{\ell =1}^{2m} f_\ell^* \otimes f_\ell
    \end{equation*}
    since for  \(j\geq 2\) the terms \(\Phi_{1j}^\inv\) vanish.
\end{proof}

\begin{remark}
    \label{anticipation}
    The invariance of \(\gamma(\pi)\) under change of Gale dual configuration shows this datum is intrinsic to the triangulated vector configuration \((V_\Sigma,\calT_\Sigma)\). However, the principal torus bundle coming with an LVMB manifold is holomorphic, and its complex structure is determined by the combined choices of Gale dual configuration and period matrix. In \cref{Borel_spectral_sequence} we will see how this is no longer true when we consider the restriction of (the complexification of) \(\gamma(\pi)\) to cohomology classes of type \((1,0)\). In fact, this detects how changing Gale dual configuration affects the complex structure of \(N\).
\end{remark}
    
\subsubsection{Identifying classes in the fan}
\label{ClassesInTheFan}
On a smooth toric variety, torus invariant divisors generate the entire class group. As these are Weil divisors, they are identified with torus invariant hypersurfaces and so, by the orbit--cone correspondence, with rays in the fan. Let \(\Sigma\) be a smooth complete simplicial fan in \(\R^d\) containing \(s\) rays, let \(v_1, \dots, v_s \in \Z^d\) be their primitive generators. Then, the corresponding Weil divisors  \(D_1, \dots, D_s\)  fulfil the linear equation
\begin{equation}
    \label{fanRelations}
    \sum_{i = 1}^s \varepsilon_j(v_i) D_i = 0 \quad \text{for \(j=1,\dots, d\)},
\end{equation}
where \(\varepsilon_j\) is the dual canonical basis of \((\R^d)^\vee\). Their Poincaré dual classes \(w_j = \PD(D_j)\) generate \(H^2(X_\Sigma;\Z)\) \cite[\S 12.3]{CoxLittleSchenck2011}, and by Stokes Theorem, they fulfil the same relations.  Indeed, these classes provide an explicit presentation of the entire cohomology ring, namely:
\begin{equation*}
    H^*(X_\Sigma;\Z) \iso \frac{\Z[w_1, \dots, w_s]}{\mathcal{I}+\mathcal{J}};
\end{equation*}
here \(\mathcal{I}\) is the ideal generated by the linear relations in \cref{fanRelations}, while \(\mathcal{J}\) is the Stanley--Reisner ideal \cite[\S 12.4]{CoxLittleSchenck2011} which consists of relations in degree at least \( 4\).

In general, a Gale dual configuration \((f_0,\dots,f_{2m})\) does not lie in the lattice \(\Rel(V)\cap\Z^n\). However, as we discussed in \cref{change_of_lattices}, as long as the triangulated vector configuration is regular, we can find another lattice, there denoted by \(\Xi\) for which the given Gale dual configuration is a basis. Of course this introduces a change of basis both in \(H^2(X_\Sigma;\Z)\) and in \(H^1(\realTorus{2m};\Z)\). Therefore, it is not restrictive to assume \((f_0,\dots,f_{2m})\) to be a basis of the full rank lattice \(K\) in \(\Rel(V)\).

To represent the cohomology classes \(\iota^\vee(f_j^*) \in K_\Sigma^\vee\) with respect to the cohomology basis given by torus invariant divisors, we observe that for a smooth toric variety, the short exact sequence \(0 \to K_\Sigma \xrightarrow{\kappa} \Z^s \xrightarrow{p} \Z^d \to 0\) is split; this fact was discussed in the proof of \cref{goodGaleDualExistence}. Hence, the monomorphism \(\kappa\) admits a retraction \(r : \Z^s \to K_\Sigma\). Considering the duals, we have the exact sequence
\begin{equation*}
    \begin{tikzcd}
        0 &  K_\Sigma^\vee \ar[l] \ar[r, bend left, "r^\vee",dashed] & (\Z^s)^\vee \ar[l, "\kappa^\vee"] & (\Z^d)^\vee \ar[l, "p^\vee"]& 0 \ar[l]
    \end{tikzcd}
\end{equation*}
with \(K_\Sigma^\vee \iso H^2(X_\Sigma;\Z)\); this isomorphism is induced by identifying divisors on \(X_\Sigma\) with integral continuous maps \(\R^d \to \R\) which are linear when restricted on every maximal cone of \(\Sigma\). Such maps are easily identified with linear functionals on \(\Z^s\) \cite[Rmk.VII.2.7]{Audin2004} modulo the subspace \(p^\vee((\Z^d)^\vee)\). In particular, this tells us that the cohomology classes \(w_j\) coincide with the canonical basis of \((\Z^s)^\vee\) with respect to this identification. Hence, we can use a section of \(\kappa^\vee\) to compute the \emph{torus invariant components} of \(\iota^\vee(f_j^*)\) modulo \(p^\vee\).

\begin{note}
    Here, the use of the term ``component'' is improper. Indeed, as  the classes \(w_j\) are never linearly independent if \(X_\Sigma\) is compact, such ``components'' cannot be unique. Despite this, we can still employ these constants in computations, keeping in mind that the same cohomology class can be presented in different ways. 
\end{note}

This presentation of the cohomology classes is independent of the choice of the retraction \(r\). Indeed, for any other choice \(r'\) we have that \((r-r')\circ \kappa = 0\), so there exists \(g\) such that \(r-r' = g \circ p\). Dualizing these equations, we get \(\Ima(r^\vee-{r'}^\vee) \subseteq \Ima p^\vee = \mathcal{I}\), so any \(h \in K_\Sigma^\vee\) has a unique torus invariant representation \(r^\vee(h)\) in the quotient
\begin{equation*}
    \frac{\Z[w_1,\dots, w_s]}{\mathcal{I}+\mathcal{J}}.
\end{equation*}

This fact shows that the geometry of a regular LVMB manifold and its toric base-space are deeply intertwined. In particular, each torus bundle \(F_j\) in which \(\pi : N \to X\) splits, determines a divisor on the base and its torus-invariant components can be read in a left-inverse of the matrix
\begin{equation*}
    \begin{bmatrix}
        \vertbar &  & \vertbar\\
        \iota^\vee(f_1^*) &\cdots &\iota^\vee(f_{2m}^*)\\
        \vertbar & &\vertbar
    \end{bmatrix}.
\end{equation*}

\subsection{Chern classes of regular LVMB manifolds}
\label{Chern_classes_LVMB}
For a compact Kähler manifold, the vanishing of its first Chern class is equivalent to existence of a Kähler Ricci-flat metric on it \cite{Yau1978}. Kähler manifolds admitting such metrics are called Calabi--Yau manifolds. Geometrically, they are characterized by the existence of an étale covering with trivial canonical bundle and, in particular, their holonomy group reduces to \(\SU(n)\).

In the non-Kähler setting, things become more involved as the vanishing of the first Chern class no longer implies the holomorphic triviality of the canonical bundle: this is  what happens for Calabi--Eckmann manifolds, for example. Indeed, whenever the \(\del\delbar\)-lemma fails, like in many non-Kähler manifolds, the role of the de Rham cohomology is played by other two cohomologies: the Bott--Chern cohomology, and the Aeppli cohomology. If \((\mathcal{A}_M^{p,q},\del,\delbar)\) is the bidifferential bicomplex of \(\C\)-valued smooth forms on \(M\), they are respectively defined in bidegree \((p,q)\) as
\begin{equation*}
    H^{p,q}_\text{BC}(M) = \frac{\ker \del \cap \ker \delbar \cap \mathcal{A}_M^{p,q}}{\Ima(\del\delbar) \cap \mathcal{A}_M^{p,q}}, \qquad H^{p,q}_{\text{Ae}}(M) = \frac{\ker \del\delbar \cap \mathcal{A}_M^{p,q}}{(\Ima \del + \Ima\delbar) \cap \mathcal{A}_M^{p,q}}.
\end{equation*}

On the other hand, a non-Kähler manifold also admits multiple Hermitian connections; among these, a distinguished role is played by the \emph{Chern connection} and the \emph{first canonical connection} \cite{LiuYang2017}. %
Combining these notions, it is possible to generalize the topological first Chern class of a holomorphic line bundle \(L \to M\) so that some of the good properties of the Kähler case are retained. Namely, for any Hermitian metric \(h\) on \(L\), the curvature form \(\Theta^h = -i \del\delbar \log h\) of the Chern connection is \(\del\) and \(\delbar\)-closed, hence it defines three different first Chern classes: the \emph{first Bott--Chern class} \(c_1^\text{BC}(L) \in H^{1,1}_\text{BC}(M)\), the \emph{first Aeppli--Chern class} \(c_1^\text{Ae}(L) \in H^{1,1}_\text{Ae}(M)\) and the (topological) first Chern class \(c_1(L) \in H^2_\text{dR}(M)\). As these are represented by the same real \((1,1)\)-form, they are mapped onto each other by the natural morphisms
\begin{equation*}
    \begin{tikzcd}[row sep = tiny]
        H^{1,1}_\text{BC}(M) \ar[r] & H^2_\text{dR}(M;\C) \ar[r] & H^{1,1}_\text{Ae}(M)\\
        c_1^\text{BC}(L) \ar[r,mapsto] & c_1(L) \ar[r,mapsto] & c_1^\text{Ae}(L) 
    \end{tikzcd}
\end{equation*}
If \(\mathcal{K}^\inv_M\) denotes the anticanonical bundle of \(M\), we define \(c_1^\diamond(M)\) as \(c_1^\diamond(\mathcal{K}^\inv_M)\), where \(\diamond\) is a placeholder used to denote either “BC'' or “Ae'' superscripts.

The first Bott--Chern class appears in \cite{Tosatti2015} where the author proves that its vanishing is equivalent to the existence of a Chern Ricci-flat metric; this gives a sort of weak non-Kähler version of the Calabi--Yau theorem. On the other hand, \(c_1^\text{Ae}\) is represented by the first Ricci form of the first canonical connection \cite{LiuYang2017}, hence Levi--Civita Ricci-flat metrics exist only when this class vanishes. However, no analogue of the Calabi--Yau Theorem exists for these metrics.

This motivates our interest in computing these classes for LVMB manifolds. In the following, we use the characteristic class \(\gamma(\pi)\) to obtain their explicit expressions.

\subsubsection{The tangent bundle of regular LVMB manifolds}
To achieve this goal, we first need to investigate the structure of the holomorphic tangent bundle \(T^{1,0}N\). As \(N\) is the total space of \(\pi : N \to X_\Sigma\), the distribution \(T^{1,0}\mathcal{F}\) generated by the holomorphic vector fields tangent to the fibers defines a holomorphic sub-bundle of \(T^{1,0}N \to N\). In particular, it fits into the short exact sequence of holomorphic vector bundles over \(N\):
\begin{equation}
    \label{equation:tangentBundleExactSequence}
    \begin{tikzcd}
        0 \ar[r] & T^{1,0}\mathcal{F}\ar[r] & T^{1,0}N \ar[r, "d\pi"]& \pi^* T^{1,0}X_\Sigma  \ar[r] & 0.
    \end{tikzcd}    
\end{equation}
We can observe that \(T^{1,0}\mathcal{F}\) is trivialized as a complex vector bundle by the fundamental vector fields of \(\pi : N \to X_\Sigma\). Since the principal action of the torus defining the bundle is holomorphic, such vector fields are holomorphic as well. Hence, \(T^{1,0}\mathcal{F} \iso \underline{\C}^m\) also in the holomorphic category. This has two main consequences, the first of which is discussed in the following:

\begin{proposition}
    \label{Chern_Class_vanish}
    Let \(N\) be a regular LVMB obtained from a triangulated vector configuration \((V,\calT)\) in \(\R^d\); then the topological \(j\)\th Chern class \(c_j(T^{1,0}N) = 0\) for \(j > d\).
\end{proposition}
\begin{proof}
    The exact sequence \labelcref{equation:tangentBundleExactSequence} splits in the smooth category, hence \(T^{1,0} N \iso \pi^*T^{1,0}X \oplus \underline{\C}^m\) as complex vector bundles. By the total Chern class formula for Whitney sums we have \(c(N) = \pi^*c(X_\Sigma)\). Since \(\dim_{\C}X_\Sigma = \rk_{\C} T^{1,0}X_\Sigma = d\), we conclude \(c_j(N) = 0\) for \(j > d\).
\end{proof}

Notice that the splitting \(T^{1,0} N \iso \pi^*T^{1,0}X \oplus \underline{\C}^m\) in general fails in the holomorphic category. However, as \(T^{1,0}\mathcal{F}\) is holomorphically trivial, the exact sequence \labelcref{equation:tangentBundleExactSequence} induces the holomorphic isomorphism \(\mathcal{K}_N \iso \pi^* \mathcal{K}_{X_\Sigma}\) between the corresponding canonical bundles. Thus, if \(\pi^*_\diamond : H^{1,1}_\diamond(X_\Sigma) \to H^{1,1}_\diamond(N)\) denotes the map induced in cohomology, we have \(c_1^\diamond(N) = \pi^*_\diamond c_1^\diamond(X_\Sigma)\).

Since \(X_\Sigma\) is smooth toric, it is a Moishezon manifold \cite[Prop.13.1.2]{CoxLittleSchenck2011}, a fortiori it fulfils the \(\del\delbar\)-lemma. Therefore, we have canonical isomorphism among the complex cohomologies \(H^{p,q}_{\delbar}(X_\Sigma) \iso H^{p,q}_{\text{BC}}(X_\Sigma) \iso  H^{p,q}_{\text{Ae}}(X_\Sigma) \), and \(H^{j}_\text{dR}(X_\Sigma) \iso \bigoplus_{p+q = j} H^{p,q}_{\delbar}(X_\Sigma)\). Under this identification, all the first Chern classes \(c_1^\diamond(X_\Sigma)\) we defined coincide; thus, there is no ambiguity in writing \(c_1(X_\Sigma)\) to denote the first Chern class in any of the above cohomologies.

Toric geometry can be exploited once again to compute combinatorially the Chern classes of a toric manifold. If \(w_1, \dots, w_s\) are the Poincaré-dual classes associated to the invariant divisors of a smooth toric manifold \(X_\Sigma\), \(c(X_\Sigma) = c(T^{1,0}X_\Sigma) = \prod_{j=1}^s (1 + w_j)\) \cite[prop. 13.1.2]{CoxLittleSchenck2011}, therefore,
\begin{equation*}
    c_1^\diamond(N) = \pi^*_\diamond c_1(X) = \pi^*_\diamond\sum_{j=1}^s w_j.
\end{equation*}

This means that the problem of computing \(c_1^\diamond(N)\) shifts to describing \(\pi^*_\diamond\). As this is the pullback of a bundle map, \(\pi^*_\text{dR}\) and \(\pi^*_{\delbar}\) arise from  Leray--Serre and Borel spectral sequences respectively: the first converges to de Rham cohomology, while the second to Dolbeault. Even though we have no spectral sequence directly computing Bott--Chern nor Aeppli cohomologies, we can infer the behaviour of \(\pi^*_\text{BC}\) and \(\pi^*_\text{Ae}\) on classes of type \((1,1)\) from \(\pi^*_{\delbar}\).

Before addressing this, we need the following well-known fact regarding the Chern class of smooth toric varieties.
\begin{lemma}
    \label{ToricCY}
    Let \(X_\Sigma\) be a smooth compact toric variety associated to a smooth complete simplicial fan \(\Sigma\). Then, \(c_1(X_\Sigma) = 0\) if and only if \(X_\Sigma = \{\cdot\}\).
\end{lemma}
\begin{proof}
    Let \(d\) be the dimension of the ambient space of \(\Sigma\); when considering a rational fan \(\Sigma\), we are implicitly assuming that \(\R^d\) is endowed with a full rank lattice \(L\) with respect to which \(\Sigma\) is rational.

    Let \(\{v_i\}_i\) be the primitive generators of the rays in \(\Sigma^{(1)}\) and \(\{D_i\}_i\) the corresponding torus-invariant divisors. Because of the expression of \(c_1(X_\Sigma)\), this vanishes if and only if \(\sum_i D_i\) is linearly equivalent to a principal divisor. As we discussed in \cref{ClassesInTheFan}, such divisors correspond to linear functionals in \((\Z^s)^\vee\) via the map there defined as \(p^\vee\). Hence, any principal divisor can be written as \(\sum_i \alpha(v_i) D_i\) with \(\alpha \in (\Z^s)^\vee\). As \(\{D_i\}_i\) is a basis for the class group, \(\sum_i D_i = \sum_i \alpha(v_i) D_i\) holds if and only if \(\alpha(v_i) = 1\) for any \(i\). This forces \(\{v_i\}_i\) to lie in the same half-space. On the other hand, since \(X_\Sigma\) is compact if and only if \(\Sigma\) is complete, \(\{v_i\}_i\) must span \(\R^d\) with positive linear combinations. These two conditions force \(\Sigma\) to be the trivial fan: correspondingly \(X_\Sigma = \{\cdot\}\).
\end{proof}

\subsubsection{The Leray--Serre spectral sequence of regular LVMB manifolds}
\label{LeraySerre}
The Leray--Serre spectral sequence for a principal \(G\)-bundle \(G \hookrightarrow P \to X\) converges to the singular cohomology of the total space at its third page. Whenever the monodromy action of \(\pi_1(X)\) on \(G\) is trivial, its second page is given by
\begin{equation*}
    E_{2}^{p,q} = H^p (X; {H}^q(G;\Z)) \iso H^p (X;\Z) \otimes  H^q(G;\Z).
\end{equation*}
This is always the case for regular LVMB manifolds as in this case the base is a compact toric manifold which is always simply connected.

In general, the differentials \(d_2^{p,q} : E_2^{p,q} \to E_2^{p+2,q-1}\) are not explicit, but in case of torus bundles they are: indeed the characteristic class \(\gamma(\pi) \in H^2(X;\Z) \otimes H^1(\realTorus{};\Z)^\vee\) coincides with \(d_2^{0,1}\) in the Leray--Serre spectral sequence \cite{Hofer1993}. Moreover, since \(H^*(\realTorus{};\Z)\) is generated in degree 1 and the differentials are compatible with wedge products, any other differential can be rewritten by means of \(d_2^{0,1}\).

Therefore, the Leray--Serre spectral sequence for a regular LVMB manifold \(\realTorus{2m}\hookrightarrow N \xrightarrow{\pi} X_\Sigma\) has the following form
\begin{center}
    \includestandalone[width=\textwidth]{pictures/Leray1}
\end{center}
with \(d_2^{0,1} = \gamma(\pi)\) and, for \(p\) even,
\begin{equation*}
    d_2^{p,q} \left(f^*_{j_1} \wedge \dots \wedge f^*_{j_q} \otimes w \right) = \sum_{\ell=1}^q (-1)^\ell \bigwedge_{i \neq \ell} f^*_{j_i} \otimes d_2^{0,1}(f^*_{j_\ell}) \wedge w \in H^{q-1}(\realTorus{2m};\Z) \otimes H^{p+2}(X_\Sigma;\Z).
\end{equation*}
In particular, \(H^1(N;\Z) \iso \ker \gamma(\pi)\). Its rank depends on the number \(k\) of ghosts in the triangulated vector configuration. Indeed, from \cref{them:father}, we see that \(\gamma(\pi)\) is injective if \(k=0\), while if \(k\geq 1\), \(\gamma(\pi)\) is surjective. Since \(\rk H^1(\realTorus{2m};\Z) = 2m\), while \(\rk H^2(X_\Sigma;\Z) = 2m+1-k\), we immediately obtain
\begin{equation}
    \label{H1LVMB}
        \rk H^1(N;\Z) =
    \begin{cases}
        0 & \text{if \(k =0\)}\\
        k-1 & \text{if \(k \geq 1\)}.
    \end{cases}
\end{equation}

On the other hand, \(H^2(N;\Z) \iso \coker d_2^{0,1} \oplus \ker d^{0,2}_2\), therefore \(\pi^* : H^2(X;\Z) \to H^2(N;\Z)\) factors as \(\coker d_2^{0,1}\) and an injection \(\coker d_2^{0,1} \hookrightarrow H^2(N;\Z)\). For any \(h \in H^2(X;\Z)\) let us denote as \([h]_{\pi^*}\) its projection onto \(\coker d_2^{0,1}\), then we obtain the following expression for the first Chern class of a regular LVMB manifold.

\begin{proposition}
    \label{DeRhamFirstChernClass}
    Let \(N\) be a regular LVMB manifold associated to a triangulated vector configuration \((V_\Sigma,\calT_\Sigma)\) with Gale transform \((f_0, \dots, f_{2m})\); denote as \(\pi : N \to X\) the associated torus bundle. Let also \(n\) be the number of vectors in \(V\), \(k\) of which are ghost vectors. Then
    \begin{equation*}
        c_1^\text{dR}(N) =
        \begin{cases}
            n[f_0^*]_{\pi^*} & \text{if \(k=0\)}\\
            0 & \text{if \( k \geq 1\)}
        \end{cases}
    \end{equation*}    
\end{proposition}
\begin{proof}
    We already observed that for \(k\geq 1\), the transgression \(d_2^{0,1} = \gamma(\pi)\) is surjective. Thus, in this case \(c_1(N) = 0\) as \(\coker d_2^{0,1} = 0\).

    On the other hand, the case \(k = 0\) occurs if and only if \(\Sigma\) is odd and balanced by itself; in particular, \(K_\Sigma = K\) and \(\{f_0, \dots, f_{2m}\}\) is a basis for \(K_\Sigma\) with \(f_0 = [1,\dots, 1]\). Notice that, for \(j \geq 1\), each \(f_j\) is determined only as \(\Z f_0\)-coset. Indeed, adding integer multiples of \(f_0\) to the other vectors neither affects \(N\) nor its sub-bundles \(F_j\). This also means that, when considering the dual basis, any choice of a representative for \(f_j + \Z f_0\), induces the same linear functional \(f_j^* \in K^\vee\).
    
    Therefore, since \(\Ima d_2^{0,1} = \Span_\Z(f_1^*,\dots,f_{2m}^*) \subset H^2(X;\Z)\),
    \begin{equation}
        \label{cokerRepresentatives}
        \coker d_2^{0,1} = \frac{H^2(X;\Z)}{\Span_\Z(f_1^*, \dots, f_{2m}^*)} \iso \frac{\Span_\Z(w_1,\dots,w_n)}{\mathcal{I} + \Span_\Z(r^\vee(f_1^*), \dots, r^\vee(f_{2m}^*))}
    \end{equation}
    where \(r\) is any retraction for the inclusion \(\kappa : K \hookrightarrow \Z^n\). As discussed in \cref{ClassesInTheFan}, any two choices differ for an element in \(\mathcal{I}\).

    To compute \(c_1(N)\) it is sufficient to project \(c_1(X) = \sum_{j=1}^n w_j\) onto \(\coker d_2^{0,1}\). This can be done observing that \(\kappa^\vee(w_j)(f_0) = 1\) for any \(j\), hence \([\kappa^\vee (w_j)]_{\pi^*} = [f_0^*]_{\pi^*}\). By linearity,
    \begin{equation*}
        c_1(N) = [\kappa^\vee(c_1(X))]_{\pi^*} = n[f_0^*]_{\pi^*}. \qedhere
    \end{equation*}
\end{proof}
The isomorphism in \cref{cokerRepresentatives} can be exploited to write \(c_1(N)\) with respect to a basis of torus-invariant divisors of \(X\).

From this description of the transgression map \(d_2^{0,1}\) we also deduce the following result.
\begin{corollary}
    \label{trivial_pullback}
    Let \(N\) be an LVMB manifold associated to a triangulated vector configuration containing at least one ghost vector, then \(\pi^*_\text{dR}(w) = 0\) for any \(w \in H^2(X_\Sigma;\Z)\). In particular, the pullback to \(N\) of the Poincaré-dual class of any torus-invariant divisor is trivial in de Rham cohomology.
\end{corollary}
\subsubsection{Example}
\label{BalancedLVM}
Consider the vector configuration in \(\R^3\) given by \(V = (e_1, -e_1, e_2, -e_2, e_3, -e_3)\), together with the triangulation \(\calT\) given by the octants, namely the one whose maximal cones are
\begin{equation*}
    \big\{ (135), (136), (145), (146), (235), (236), (245), (246)\big\};
\end{equation*}
in this case, \(n = 6\), \(d = 3\) and \(m = 1\). In particular, \((V,\calT)\) is odd and balanced, so we can construct a corresponding LVMB manifold without adding ghost vectors. To this end, we consider the following Gale dual configuration
\begin{equation*}
    M = 
    \begin{bmatrix}
        \vertbar &\vertbar& \vertbar\\
        f_0 & f_1 & f_2\\
        \vertbar &\vertbar& \vertbar
    \end{bmatrix} =
    \begin{bmatrix}
        1 & 1 & 0\\
        1 & 1 & 0\\
        1 & 0 & 1\\
        1 & 0 & 1\\
        1 & 0 & 0\\
        1 & 0 & 0
    \end{bmatrix},
\end{equation*}
whose corresponding point configuration in \(\mathbb{A}^1_\C\) is \(\Lambda = (1,1,i,i,0,0)\). Notice that this configuration is regular, so it determines a smooth fan \(\Sigma\) and a holomorphic torus bundle \(\pi : N_3 \to X\) where \(X = \CP^1 \times \CP^1 \times \CP^1\).

To compute its first Chern class, it is sufficient to find a left-inverse for \(M\) whose rows are the coordinates of \(f_j^*\) with respect to the basis of torus invariant divisors of \(\Sigma\). A different choice for such a left-inverse corresponds to a different choice of the retraction \(r : \Z^6 \to K\) in the construction, but as we proved in \cref{ClassesInTheFan}, any expression obtained this way represents the same cohomology class. Therefore,
\begin{equation*}
    \begin{bmatrix}
        \vertbar &\vertbar& \vertbar\\
        f_0^* & f_1^* & f_2^*\\
        \vertbar &\vertbar& \vertbar
    \end{bmatrix} =
    \begin{bmatrix}
        0 & 1 & 0\\
        0 & 0 & 0\\
        0 & 0 & 1\\
        0 & 0 & 0\\
        1 & -1 & -1\\
        0 & 0 & 0
    \end{bmatrix} =
    \begin{bmatrix}
        \vertbar & \vertbar & \vertbar\\
        w_5 & w_1 - w_5 & w_3 - w_5\\
        \vertbar & \vertbar & \vertbar
    \end{bmatrix}.
\end{equation*}
Thus, \(c_1(N) = 6[w_5]_{\pi^*}\). Notice that \([w_5]_{\pi^*}\) is a generator of
\begin{equation*}
    \frac{\Span_\Z(w_1,\dots, w_6)}{\Span_\Z(w_1-w_2, w_3-w_4,w_5-w_6) + \Span_\Z(w_1-w_5, w_3-w_5)} \iso \Z.
\end{equation*}

We can recognize \(N_3\) as a holomorphic fiber bundle over \(\CP^1\) with the Calabi--Eckmann manifold \(M_{1,1;i}\) as fiber. Indeed, the projection
\begin{align*}
    \psi : N_3 &\longrightarrow \CP^1\\
    [z_1: \dots:z_6] &\longmapsto (z_5:z_6)
\end{align*}
is well-defined as \(N_3\) is a quotient of \(U(\calT) = (\C^2\setminus\{0\})^3\), and it is locally trivialized by the maps
\begin{align*}
    f_0 : U_0 \times M_{1,1;i} &\longrightarrow \psi^\inv(U_0)\\
    ((u_0:u_1),[z_1,z_2,z_3,z_4]_{M_{1,1;i}}) &\longmapsto \left(z_1:z_2:z_3:z_4:1:\frac{u_1}{u_0}\right)\\
    f_1 : U_1 \times M_{1,1;i} &\longrightarrow \psi^\inv(U_1)\\
    ((u_0:u_1),[z_1,z_2,z_3,z_4]_{M_{1,1;i}}) &\longmapsto \left(z_1:z_2:z_3:z_4:\frac{u_0}{u_1}:1\right).
\end{align*}
On \(U_0 \cap U_1\), the composition \(f_1^\inv \circ f_0\) gives the clutching function
\begin{align*}
    (U_0 \cap U_1) \times M_{1,1;i} &\longrightarrow (U_0 \cap U_1) \times M_{1,1;i}\\
    ((u_0:u_1),[z_1,z_2,z_3,z_4]_{M_{1,1;i}}) &\longmapsto \left((u_0:u_1),\left[\frac{u_0}{u_1} z_1,\frac{u_0}{u_1}z_2,\frac{u_0}{u_1}z_3,\frac{u_0}{u_1}z_4\right]_{M_{1,1;i}}\right)
\end{align*}
which is clearly holomorphic.

Such bundle is not holomorphically trivial. Indeed, since \(b_2(M_{1,1;i}) = 0\), \(c_1(\CP^1 \times M_{1,1;i}) = \pi^* c_1(\CP^1) = \pi^*(2w)\) where \(w\) is the hyperplane class of \(\CP^1\). On the other hand, \(c_1(N_3) = 6 [w_5]_{\pi^*}\). Thus, \(N_3\) cannot be a product in the holomorphic category as the subgroup \(c_1(N_3)\Z \subset H^2(N_3;\Z)\) has index 6, while, by the  Künneth Formula, \(\pi^*(2w)\Z \subset H^2(\CP^1 \times M_{1,1;i};\Z)\) has index 2.

\subsubsection{The Borel spectral sequence of a regular LVMB manifold}
\label{Borel_spectral_sequence}
As announced before, computations of Dolbeault cohomology can be performed via the Borel Spectral Sequence, the main reference for this tool is \cite[Appendix 2]{Hirzebruch1978}. For the holomorphic LVMB bundle \(\cpxTorus{m} \hookrightarrow N \xrightarrow{\pi} X_\Sigma \), its second page consists of 4-graded objects \(\fourDegreed{E}{p,q}{u,v}{2}\) where \((u,v)\) are respectively the base-index and the fiber-index, while \((p,q)\) is the type bidegree. All terms with \(p+q \neq u+v\) vanish, while the others are given by
\begin{equation*}
    \fourDegreed{E}{p,q}{u,v}{2} = \sum_{i \geq 0} H_{\delbar}^{i, u-i}(X) \otimes H_{\delbar}^{p-i, v-p+1}(\cpxTorus{m}).
\end{equation*}
In \cite[Thm.6.3]{Hofer1993}, the author shows that, whenever the base space has the Hodge decomposition, the differential \(d_2^{0,1} : \fourDegreed{E}{*,*}{0,1}{2} \to \fourDegreed{E}{*,*}{2,0}{2}\) is the restriction of the complexification of \(\gamma(\pi)\) to the subspace of classes of type \((1,0)\). All the other maps can be obtained from \(d_2^{0,1}\) as extensions over the tensor product. This is always the case for toric manifolds, being Moishezon \cite[Thm.6.1.18]{CoxLittleSchenck2011}, they fulfil the \(\del\delbar\)-lemma, a fortiori the Hodge decomposition holds.

Since \(H^{*,*}_{\delbar}(X)\) is generated in bidegree \((1,1)\), in our case Höfer's decomposition reduces to:
\begin{equation*}
    \gamma(\pi)(a) = \fourDegreed{d}{1,0}{0,1}{2}(a^{1,0}) + \overline{\fourDegreed{d}{1,0}{0,1}{2}(\overline{a^{1,0}})} = \fourDegreed{d}{1,0}{0,1}{2}(a^{1,0}) + \overline{\fourDegreed{d}{1,0}{0,1}{2}(a^{0,1})}
\end{equation*}
where \(a \in H^1_\text{dR}(\cpxTorus{m}) \otimes \C\) and \(a^{1,0}, a^{0,1}\) are the projections on \(H^{1,0}_\delbar(\cpxTorus{m}), H^{0,1}_\delbar(\cpxTorus{m})\) respectively. 

In the construction of LVMB manifolds, we observed that the complex structure of the fibers (and so of the manifold \(N\)), depends on the choice of the period matrix \(\Pi\); thus, we need to take this into account when restricting \(\gamma(\pi)\) to \(H^{1,0}_\delbar(\cpxTorus{m})\) and \(H^{0,1}_\delbar(\cpxTorus{m})\). We denote these restrictions as \(\gamma(\pi)^{1,0}\) and \(\gamma(\pi)^{0,1} = \overline{\gamma(\pi)^{1,0}(\overline{\,\cdot\,})}\) respectively. Under the same notations of \cref{section:construction}, for any Gale dual configuration \(\{f_0, \dots, f_{2m}\}\), the lattice \(\hat K = \Span(f_1,\dots, f_{2m})\) identifies canonically with \(H^1(\cpxTorus{2m};\Z)^\vee\). Therefore, for the standard period matrix \(\Pi_\text{std}\), we have \(H^{1,0}_\delbar(\cpxTorus{2m})^\vee\) and \(H^{0,1}_\delbar(\cpxTorus{2m})^\vee\) respectively spanned by
\begin{equation*}
    \left\{\eta_j = \frac{1}{2}(f_j - i f_{j+m})\right\}_{j=1, \dots, m}, \quad \text{and} \quad \left\{\bar\eta_j = \frac{1}{2}(f_j + i f_{j+m})\right\}_{j=1, \dots, m}.
\end{equation*}
Similar expressions can be obtained for a different choice of \(\Pi\), however, as we are not making assumptions on the Gale dual configuration, because of \cref{ThreeOutOfFour}, we can assume \(\Pi = \Pi_\text{std}\). Under the identifications above, we are able to provide a formula for the \((1,0)\) part of the characteristic class \(\gamma(\pi)\). Of course, as it is real, a conjugation will provide the \((0,1)\) component. Namely, denoting as \(\tilde f_j^* \coloneqq \iota^\vee(f_j^*)\) the projection onto \(K_\Sigma^\vee\) of the dual basis of the Gale dual configuration, we have
\begin{equation*}
    \label{DolbeaultCharacteristicClass}
    \gamma(\pi)^{1,0} = \sum_{j=1}^{m} \left(c_1(\pi_j) + i c_1(\pi_{j+m})\right) \otimes \eta_j = \sum_{j=1}^{m} (\tilde f_{j}^* + i \tilde f_{j+m}^*) \otimes \eta_j.
\end{equation*}
\begin{note}
    As we anticipated in \cref{anticipation}, \(\gamma(\pi)^{1,0}\) changes accordingly with the combined choices of Gale duality and period matrix. Indeed, for any period matrix \(\Pi\), one has \(\gamma(\pi)^{1,0} = \sum_{j=1}^m \varsigma_j \otimes \eta_j\), with
    \begin{equation*}
        \begin{bmatrix}
            \vertbar & & \vertbar\\
            \eta_1 &\cdots& \eta_m\\
            \vertbar & & \vertbar
        \end{bmatrix}
        = \frac{1}{2}
        \begin{bmatrix}
            \vertbar & & \vertbar\\
            f_1 &\cdots& f_{2m}\\
            \vertbar & & \vertbar
        \end{bmatrix}
        \trans{\overline{\Pi}}
    ,\qquad
    [\varsigma_1, \cdots, \varsigma_m] = [c_1(\pi_1), \cdots,  c_1(\pi_{2m})] \trans{\Pi}.
    \end{equation*}
\end{note}

Since the base-space is toric, \(E_2^{u,v} = 0\) for \(u\) odd, thus \(E_3 = E_\infty\), and so the isomorphisms
\begin{equation*}
    H^{p,q}_\delbar(N) \iso \sum_{p+q = u+v} \fourDegreed{E}{p,q}{u,v}{3}
\end{equation*}
extend to an isomorphism of graded algebras. These observations lead to the following result, which is the LVMB counterpart of \cite[Thm.5.4]{PanovUstinovsky2012}.
\begin{proposition}
    \label{prop:DolbeaultModel}
    Let \(\pi : N \to X\) be the holomorphic \(\cpxTorus{m}\)-bundle coming with an LVMB manifold \(N\). Then \(H^{*,*}_\delbar(N)\) is quasi-isomorphic to the bigraded differential algebra
    \begin{equation*}
        (\Lambda[\xi_1, \dots, \xi_{m}, \bar\xi_1, \dots, \bar\xi_m] \otimes H^{*,*}_\delbar(X), d)
    \end{equation*}
    where \(\xi_j\) are dual to \(\eta_j\), \(d\xi_j = \gamma(\pi)^{1,0}(\xi_j)\), \(d\bar\xi_j = 0\) and \(d|_{H^{*,*}_\delbar(X)} = 0\).
\end{proposition}

To prove this fact, it is sufficient to repeat the argument of \cite[Thm.5.4]{PanovUstinovsky2012} which, actually, holds in general for holomorphic torus bundles over a formal base-space. This is our case as compact toric manifolds always fulfil the \(\del\delbar\)-lemma and, a fortiori, are formal \cite{DeligneGriffithsMorganSullivan1975}.

In this model the classes \(\xi_j\) represent the holomorphic coordinates on the fiber \(\cpxTorus{2m}\); of course such coordinates are dual to generators \(\eta_j\) which are determined by the Gale dual configuration \((f_0, \dots, f_{2m})\).

\begin{remark}
    Since \(\gamma(\pi)^{0,1}(\bar\xi_j) = \overline{\gamma(\pi)^{1,0}(\xi_j)}\), this model is real as it preserves conjugation, thus, it can be promoted to a bidifferential model for \((\mathcal{A}_N^{*,*}, \del,\delbar)\) just by introducing a new differential \(\bar d\) defined as \(\bar d(a)  = \overline{d(\bar a)}\).
\end{remark}

Such result is significant since the Dolbeault cohomology ring of smooth toric variety \(X_\Sigma\) is generated in bidegree \((1,1)\) by the Poincaré-dual classes of its invariant divisors \cite[\S 12.4]{CoxLittleSchenck2011}. In our notation, we can take as generators the classes \(\tilde f_j^* = \iota^\vee(f_j^*)\), which in general are not linearly independent. Therefore in \cref{prop:DolbeaultModel},
\begin{equation*}
    H^{1,1}_\delbar(X) = \Span_\C (\tilde f_1^*, \dots, \tilde f_{2m}^*) \iso \frac{\Span_\C(w_1, \dots, w_\ell)}{\mathcal{I}}.
\end{equation*}
These observations lead to the failure of the \(\del\delbar\)-lemma in all non-trivial cases:
\begin{proposition}
    \label{deldelbar_lemma}
    Let \(N\) be a regular LVMB manifold associated to a triangulated configuration \((V_\Sigma,\calT_\Sigma)\) in \(\R^d\) containing \(n\) vectors, \(k\) of which are ghosts.
    Then, \(N\) fulfils the \(\del\delbar\)-lemma if and only if it is a complex torus.
\end{proposition}
\begin{proof}
    The ``if'' part is trivial as complex tori are Kähler; a fortiori, they fulfil the \(\del\delbar\)-lemma.

    Conversely, suppose \(N\) fulfils the \(\del\delbar\)-lemma. To investigate the double complex of smooth forms, consider the classes \(\eta_j\) described above, together with their duals \(\xi_j \in H^{1,0}(\cpxTorus{m})\). By \cref{prop:DolbeaultModel}, these define Dolbeault classes if and only if they are \(\delbar\)-closed, hence \(h_{\delbar}^{1,0} = \dim_{\C}\ker \gamma(\pi)^{1,0}\). On the other hand, in the model for the Dolbeault cohomology we have \(\delbar\xi_j = 0\) for any \(j\), thus \(h_\delbar^{0,1} = m\).

    If \(k = 0\), we have \(K_\Sigma = K\), thus all \(\tilde f_j^*\) are linearly independent over \(\R\), a fortiori, the set \(\{\tilde f_j^* + i\tilde f_{j+m}^*\}_{j=1,\dots,m}\) is linearly independent in \(H^{1,1}(X_\Sigma)\). Thus, \(\dim \ker\gamma(\pi)^{1,0} = 0\).

    For \(k > 0\) we observe that \(\gamma(\pi)\) is surjective, therefore, among the \(f_j^*\), there are \(2m-k+1\) of them such that \(\tilde f^*_{i_1}, \dots, \tilde f^*_{i_{2m-k+1}}\) are linearly independent in \(H^{1,1}(X_\Sigma)\). Up to reordering them, we may assume they are \(\tilde f^*_k, \dots, \tilde f^*_{2m}\). Hence, the image of \(\gamma(\pi)^{1,0}\) contains  a complex subspace of dimension at least \(\min(m,2m-k+1)\), thus \(\dim\ker\gamma(\pi)^{1,0} \leq \max(0,k-m-1)\).
    
    Therefore, for any value of \(k\), we have \(h_{\delbar}^{1,0} \leq \max(0,k-m-1)\), so the \(\del\delbar\)-lemma holds only if  \(k = 2m+1\): i.e. only if all the vectors in \((V_\Sigma,\calT_\Sigma)\) are ghosts. As \(2m+1 = n - d\), the \(\del\delbar\)-lemma forces \(d = n-k\). But the latter coincides by construction with the number of rays in \(\Sigma\); being \((V_\Sigma,\calT_\Sigma)\) balanced, \(\Sigma\) is also complete, and so the \(\del\delbar\) lemma forces \(d = 0\). By \cite{Bosio2001}, this condition assures \(N = \cpxTorus{m}\).
\end{proof}

\subsubsection{First Bott--Chern and Aeppli--Chern classes}
\Cref{prop:DolbeaultModel} provides a bidifferential model which can be used to check the behaviour of the classes \(h_i^*\) in Bott--Chern and Aeppli cohomologies. In particular, we can use it to compute \(c_1^\text{BC}(N)\) and \(c_1^\text{Ae}(N)\).

\begin{proposition}
    \label{First_Bott-Chern_Class}
    For a regular LVMB manifold \(\pi : N \to X_\Sigma\), the pull-back map induced in Bott--Chern cohomology \(\pi^*_\text{BC} : H^{1,1}(X_\Sigma) \to H^{1,1}_\text{BC}(N)\) is injective. In particular, \(c_1^\text{BC}(N) = 0\) if and only if \(N\) is a complex torus. 
\end{proposition}
\begin{proof}
    As we already observed in the proof of \cref{deldelbar_lemma}, it is not restrictive to assume that for a Gale dual configuration \((f_0, \dots, f_{2m})\) of some regular triangulated configuration \((V_\Sigma,\calT_\Sigma)\), \(\tilde f_{k}^*,\dots, \tilde f_{2m}^*\) are a basis of \(H^{1,1}(X_\Sigma)\). Because of \cref{prop:DolbeaultModel}, we can use them as generators of a model for \((\mathcal{A}^{*,*}_N,\del,\delbar)\); hence, since \(d \tilde f_j^* = 0\) for any \(j\), these also represent Bott--Chern classes in \(H^{1,1}_\text{BC}(N)\). Therefore, the statement is equivalent to showing that they are linearly independent in this cohomology. In \cref{DolbeaultCharacteristicClass} we observed that
    \begin{equation*}
        \tilde f_{j}^* + i \tilde f_{j+m}^* = \delbar \xi_j \quad \text{for \(j = 1, \dots, 2m\)},
    \end{equation*}
    hence,
    \begin{equation*}
        \tilde f_j^* =
        \begin{cases}
            d\Re\xi_j & j=1,\dots,m\\
            d\Im\xi_{j-m} & j= m+1 \dots, 2m.
        \end{cases}
    \end{equation*}
    Therefore, a linear combination of \(\tilde f_k^*, \dots, \tilde f_{2m}^*\) is trivial in Bott--Chern cohomology if and only if there exists \(g \in \mathcal{A}^{0,0}_N\) such that \(\sum_{j=k}^{2m} a_j \tilde f_j^* = \del\delbar g\).
    
    If \(k=0\) this condition reads as
    \begin{equation}
        \label{BCIndependenceNoGhosts}
        a_0 \tilde f_0^* + d\left(\sum_{j=1}^m a_{j} \Re \xi_j + b_{j} \Im \xi_j \right) = d\delbar g.
    \end{equation}
    Since \(\tilde f_0^*\) is not \(d\)-exact, this forces \(a_0 = 0\); for the others, notice that if \(k=0\) then \(H^1_\text{dR}(N) = 0\) by \labelcref{H1LVMB}, hence the last equation holds if and only if
    \begin{equation*}
        \sum_{j=1}^m a_{j} \Re \xi_j + a_{j+m} \Im \xi_j = \delbar g + \del g'
    \end{equation*}
    for some complex valued smooth function \(g'\). This equation splits over \(\mathcal{A}^{1,0}_N \oplus \mathcal{A}^{0,1}_N\) as 
    \begin{equation*}
        \begin{cases}
            \sum_{j=1}^m (a_{j} - i a_{j+m})\xi_j = 2\del g'\\
            \sum_{j=1}^m (a_{j} + i a_{j+m})\bar\xi_j = 2\delbar g
        \end{cases}
    \end{equation*}
    which forces \(a_j = a_{j+m} = 0\) for any \(j = 1,\dots, m\) since \(\xi_j\) and \(\bar\xi_j\) are linearly independent generators of \(H^{1,0}_\del(N)\) and \(H^{0,1}_\delbar(N)\) respectively.

    The same argument applies to the case \(k \geq 1\) where more care of the indices is needed. This time \cref{BCIndependenceNoGhosts} reads as
    \begin{equation*}
        d\left(\sum_{j=k}^{m} a_j \Re \xi_{j} + \sum_{j=m+1}^{2m} a_j \Im \xi_{j-m}\right) = d\delbar g.
    \end{equation*}
    In \cref{LeraySerre} we observed that \(H^1_\text{dR}(N) \iso \ker \gamma(\pi)\), thus the last equation holds if and only if 
    \begin{equation*}
        \sum_{j=k}^m a_{j} \Re \xi_j + \sum_{j = m+1}^{2m} a_{j} \Im \xi_{j-m} = \delbar g + \del g' + \vartheta
    \end{equation*}
    for some \(g \in \mathcal{A}^{0,0}_N\) and \(\vartheta = \sum_{j=1}^{m} b_j \Re \xi_j + \sum_{j=m+1}^{2m} b_j \Im \xi_{m+j} \in \ker \gamma(\pi)\). Projecting the above equation to \(\mathcal{A}^{1,0}_N\), we get
    \begin{equation*}
        - \sum_{j=1}^{k-1} \left(b_j +i (a_{j+m} - b_{j+m})\right) \xi_{j} + \sum_{j=k}^{m} ((a_j - b_j) -i (a_{j+m} - b_{j+m})) \xi_{j} = 2\del g'.
    \end{equation*}
    Since by \cref{prop:DolbeaultModel}, the vectors \(\xi_j\) are linearly independent generators of \(H^{1,0}_\del(N)\), this forces \(b_j = 0\) for \(j \leq k-1\) and \(a_j = b_j\) for \(j \geq k\). Therefore, as \(\vartheta \in \ker\gamma(\pi)\), we have
    \begin{equation*}
        0 = \gamma(\pi)(\vartheta) = \gamma(\pi)\left(\sum_{j=k}^m b_j \Re \xi_j + \sum_{j=m+1}^{2m} b_j \Im \xi_{j-m}\right) = \sum_{j=k}^{2m} b_j \tilde f_j^*
    \end{equation*}
    whence \(b_j = 0\) for any \(j\), as we are assuming \(\{\tilde f_j^*\}_{j=k,\dots,2m}\) to be a basis of \(H^{1,1}(X_\Sigma)\). Therefore, \(a_j = 0\) for \(j \geq k\), which proves the linear independence of \(\{\tilde f_j^*\}_{j=k,\dots, 2m}\) as Bott--Chern classes.

    The last part of the statement follows from observing that \(c_1^\text{BC}(N) = \pi^*_\text{BC}(c_1(X))\). Since \(\pi^*_\text{BC}\) is injective, \(c_1^\text{BC}(N) = 0\) if and only if \(c_1(X) = 0\); as \(X\) is smooth and compact, this must be a point because of \cref{ToricCY}. Thus, any LVMB manifold \(N\) with \(c_1^\text{BC}(N) = 0\) is a complex torus.
\end{proof}

Computing the first Aeppli--Chern class is somehow simpler. Indeed, since \(c_1(N) = 0\) whenever its configuration contains ghost vectors, \(c_1^\text{Ae}(N)\) vanishes as well. On the other hand, we proved that for configurations without ghost vectors, \(c_1(N) = n[f_0^*]_{\pi^*} \in H^2_\text{dR}(N)\). From \cref{prop:DolbeaultModel} we see that \(f_0^*\) is never a linear combination of \(\delbar \xi_j\) and \(\del \bar\xi_j\), hence it does not vanish in Aeppli cohomology. This can be resumed in the following
\begin{proposition}
    \label{First_Aeppli-Chern_class}
    Let \(N\) be a regular LVMB manifold associated to a triangulated vector configuration \((V_\Sigma,\calT_\Sigma)\) with \((f_0,\dots, f_{2m})\) as Gale dual configuration. Denote by \(\pi : N \to X\) the torus bundle it comes with, and let \(n\) be the number of vectors in \(V\), \(k\) of which are ghost vectors. Then
    \begin{equation*}
        c_1^\text{Ae}(N) =
        \begin{cases}
            n[f_0^*]_{\pi^*_\text{Ae}} & \text{if } k = 0\\
            0 & \text{if } k \geq 1.
        \end{cases}
    \end{equation*}
\end{proposition}

\begin{remark}
    \label{remark:nonrational}
    The construction of LVMB manifolds does not require the LVMB datum \((\Lambda,\mathcal{E})\) to be rational at all \cite{LopezVerjovsky1997,Meersseman2000,Bosio2001}. In particular, the approach to LVMB we have adopted extends to any complete simplicial fan \cite{BattagliaZaffran2015}.
    Moreover, any LVMB datum \((\Lambda,\mathcal{E})\) can always be deformed into a rational one. Furthermore, the canonical foliation of an LVMB manifold provides natural substitutes of forms and cohomology of the base space when this is no longer a manifold \cite{LoebNicolau1996,Meersseman2000,BattagliaZaffran2015,Ishida2017,IshidaKrutowskiPanov2022}: these are basic forms and basic cohomology respectively.
    Thus, it is very natural to ask whether and to what extent our results, including those about “special'' Hermitian metrics, hold true for the whole class of LVMB manifolds. In the rational non-smooth case, we believe that our techniques can be suitably adapted and that our results, including those concerning the metrics, carry over essentially unchanged. They might also hold true in the general case, but their proofs would likely require different techniques.
\end{remark}

\section{Special metrics on regular LVMB manifolds}
\label{Special_metrics_LVMB}

In this section we will investigate the existence of \emph{balanced} and \emph{SKT} metrics on regular LVMB manifolds.

We recall that balanced metrics, sometimes also called semi-Kähler or co-Kähler metrics, are Hermitian metrics whose fundamental form is co-closed. They were introduced by Michelsohn in \cite{Michelsohn1982}, where she also characterizes their existence on compact complex manifolds. Like the LCK metrics we mentioned in the introduction, balanced ones have good geometrical properties as well. Indeed, they are stable under holomorphic modifications \cite{AlessandriniBassanelli1995}: on complex surfaces, this fact translates to the well-known stability of  Kähler metrics under both blow-ups and blow-downs.

On the other hand, SKT metrics, often referred to as \emph{pluriclosed} metrics, are the ones whose fundamental form satisfies \(\del\delbar\omega = 0\). The acronym “SKT'', standing for \emph{Strong Kähler with Torsion}, recalls that the corresponding Bismut connection (sometimes called \emph{Kähler with Torsion}, or \emph{Strominger} connection) has closed torsion \cite{Bismut1989}. In that paper, Bismut shows how the SKT condition naturally appears in an index theorem for non-Kähler complex manifolds. SKT metrics appear ubiquitously on complex surfaces: this is a direct consequence of the celebrated \emph{“Théorème de l'excentricité nulle''} of Gauduchon \cite{Gauduchon1977}.

\subsection{Submersion metrics on holomorphic principal torus bundles}
\label{section:submersion_metric}
In general, the space of Hermitian metrics on a complex manifold is huge: this makes it difficult to explore without explicit construction techniques which are capable of detecting special metrics. Nevertheless, on an LVMB manifold, we can exploit its holomorphic torus bundle structure to ``assemble'' a Hermitian metric on the total space by lifting a metric defined on the base, and then completing it to a positive definite tensor.

Metrics constructed this way are called \emph{submersion metrics} and, since they detect the bundle structure, in case of a holomorphic principal torus bundle, those that are special can be explicitly characterized \cite{GrantcharovPoon2008}. As we will discuss in this section, this is particularly remarkable for LVMB manifolds. Indeed, their bundle structure is determined by the characteristic class \(\gamma(\pi)\), and we showed in \cref{them:father} that this is encoded in the algebro-combinatoric datum, which the LVMB manifold is associated with.

\begin{definition}
    Let \(\pi : M \to X\) be a holomorphic fiber bundle. We say that a Hermitian metric \(g_M\) on \(M\) is a \emph{submersion metric} if \(\pi : (M,g_M) \to (X,g_X)\) is a Riemannian submersion for some Hermitian metric \(g_X\) on \(X\). That is, if
    \begin{equation*}
        d\pi|_{(\ker d\pi)^{\perp_{g_M}}} : \left((\ker d\pi)^{\perp_{g_M}},g_M\right) \to (TX,g_X)
    \end{equation*}
    is an isomorphism of Hermitian vector bundles.
\end{definition}

A holomorphic principal torus bundle \(\pi : P \to X\) naturally admits submersion metrics. With the same ideas employed in \cref{section:relationWithToricGeometry}, such bundle can be decomposed into the product of principal \(S^1\)-bundles, namely \(\pi \iso \Delta^*\prod_{j = 1}^{2m} \pi_j\). As we are interested in preserving the complex structure of the fibers, we consider only decompositions such that the fibers of \(\pi_{2j-1}\) and \(\pi_{2j}\) are coupled together by the complex structure \(J_P\) of \(P\). From another perspective, we are decomposing \(P\) as the product of \(m\) holomorphic principal \(\cpxTorus{m}\)-bundles. Thus, any principal connection on \(\pi\) is given by an \(m\)-tuple of 1-forms \((\tilde\theta_1, \dots, \tilde\theta_{m})\) in \(\Omega^1(P;\C)\). By taking \(\theta_{2j-1} = \Re\tilde\theta_j\) and \(\theta_{2j} = \Im\tilde\theta_j\), we get a list of \(2m\) real-valued 1-forms \((\theta_1,\dots, \theta_{2m})\), each of which defines a connection on the principal \(S^1\)-bundle \(\pi_j\); moreover, this way they fulfil \(\theta_{2j} = J_P \theta_{2j-1}\).
Therefore, for any Hermitian metric \(g_X\) on \(X\), the symmetric tensor
\begin{equation*}
    g_P = \pi^* g_X + \sum_{j = 1}^{2m} \theta_{j} \otimes \theta_{j}
\end{equation*}
defines a Hermitian metric on \(P\) with corresponding Kähler form given by
\begin{equation}
    \label{submersion_metric}
    \omega_P = \pi^* \omega_X + \sum_{\ell = 1}^{2m} J\theta_j \otimes \theta_j = \pi^*\omega_X + \sum_{j = 1}^m \theta_{2j-1} \wedge \theta_{2j}.
\end{equation}
\begin{remark}
    In \cref{section:splitting_of_pi} we discussed how the choice of a Gale dual configuration provides a decomposition \(\pi \iso \Delta^*\left(\prod_{j=1}^{2m} \pi_j\right)\) of the holomorphic principal torus bundle \(\pi : N \to X_\Sigma\) that an LVMB manifold \(N\) comes with. Thus, for these bundles, the coupling between \(\pi_{2j-1}\) and \(\pi_{2j}\) via the complex structure comes for free.
\end{remark}

In fact, Riemannian metrics which are invariant under the principal action of a torus bundle are submersion metrics. For sake of completeness, we give a proof of this well-known fact in the Hermitian setting.

\begin{lemma}
    \label{InvariantIsSubmersion}
    Let \(\pi : N \to X\) be a holomorphic principal torus bundle with \(\cpxTorus{m}\) as fiber. Let \(a : \cpxTorus{m} \times N \to N\) be the corresponding holomorphic action with \(a_t : N \to N\) the biholomorphism induced by any \(t \in \cpxTorus{m}\). Then, any Hermitian metric \(g\) on \(N\) which is \(a\)-invariant coincides with a submersion metric.
\end{lemma}
\begin{proof}
    Since \(g\) is \(a\)-invariant, it defines a scalar product on the Lie algebra of the torus \(\operatorname{Lie}(\cpxTorus{m}) \iso \R^{2m}\). Considering an orthonormal basis \(\{e_1,\dots, e_{2m}\}\), the corresponding fundamental vector fields \(\{E_1, \dots, E_{2m}\}\) (see e.g. \cite[\S 27.3]{Tu2017} for the definition) are an orthonormal frame for the vertical distribution of \(\pi : N \to X\). Taking the \(\R^{2m}\)-valued 1-form \(\theta^{\R} = [\theta_1,\dots,\theta_{2m}] \in \Omega^1(N;\R^{2m})\) with components
    \begin{equation*}
        \theta_j \coloneqq E_j \contr g \in \Omega^1(N;\R),
    \end{equation*}
    we obtain a principal connection. Indeed, it is clearly smooth, it fulfils \(\theta_j(E_i) = \delta_{ij}\) and, for any \(V \in \mathfrak{X}(N)\),
    \begin{equation*}
        a_{-t}^*\theta_j(V) = \theta_j(da_{-t}V) = g(da_{-t}V,E_j) = g(V, da_{t}E_j) =  g(V, E_j) = \theta_j(V).
    \end{equation*}

    Moreover, as the bundle is holomorphic, \(J_N\) is compatible with the complex structure of \(\operatorname{Lie}(\cpxTorus{m})\) inherited from the fiber complex torus. As \(J_N g = g\), the fundamental vector fields \(E_j\) can be taken such that \(J_N E_{2j-1} = E_{2j}\) and so, for any \(V \in \mathfrak{X}(V)\), we have
    \begin{equation*}
        J_N \theta_{2j-1}(V) = -  g(E_{2j-1}, J_NV) =  g(J_N E_{2j-1}, V) =  g (E_{2j}, V) = \theta_{2j}(V);
    \end{equation*}
    thus \(J_N \theta_{2j-1} = \theta_{2j}\). This way, the complex-valued 1-form \(\theta = [\dots,\theta_{2j-1} + i \theta_{2j}, \dots] \in \Omega^1(N;\C^m)\) defines a principal connection for the holomorphic torus bundle \(\cpxTorus{m} \hookrightarrow N \xrightarrow{\pi} X\).

    To construct \(g\) as submersion metric, we need a Hermitian one on the base, we define it as
    \begin{equation*}
        g_X (W_1,W_2) =  g(W_1^\sharp, W_2^\sharp) 
    \end{equation*}
    for any \(W_1, W_2 \in \mathfrak{X}(X)\), and \(W_1^\sharp, W_2^\sharp \in \mathfrak{X}(N)\) horizontal lifts with respect to the connection defined by \(\theta\). This is Hermitian as \(\theta\) is compatible with the complex structure.

    The objects we have constructed fulfil
    \begin{equation}
        \label{submersionFancy}
         g = \pi^*  g_X + \sum_{j=1}^{2m} \theta_j \otimes \theta_j.
    \end{equation}
    This can be proved by checking that the equality holds in three cases.
    \begin{itemize}
        \item Let \(W_1, W_2 \in \mathfrak{X}(N)^\text{hor}\) be two vector fields that are horizontal with respect to \(\theta\). Then, as the 1-forms \(\theta_j\) are supported in the vertical component, \(\theta_j(W_i) = 0\) for \(i=1,2\). On the other hand,
        \begin{equation*}
            \pi^*  g_X(W_1,W_2) =  g_X(d\pi W_1, d\pi W_2) =  g((d\pi W_1)^\sharp, (d\pi W_2)^\sharp) =  g(W_1,W_2)
        \end{equation*}
        by uniqueness of the horizontal lift.
        \item Let \(W_1 \in \mathfrak{X}(N)^\text{hor}\) and \(W_2 \in \mathfrak{X}(N)^\text{vert}\). Notice that, by construction,  the horizontal distribution defined by \(\theta\) is \( g\)-orthogonal to the vertical one. Thus, \( g(W_1,W_2) = 0\). On the other hand
        \begin{equation*}
            \pi^*  g_X(W_1, W_2) + \sum_{j=1}^{2m} (\theta_j \otimes \theta_j)(W_1,W_2) =  g_X(d\pi W_1, d\pi W_2) + \sum_{j=1}^{2m} \theta_j(W_1)  \theta_j(W_2) = 0
        \end{equation*}
        as \(d\pi W_2 = 0\), and \(\theta_j(W_1) = 0\) for any \(j\).
        \item Let \(W_1, W_2 \in \mathfrak{X}(N)^\text{vert}\), as \(\{E_j\}\) is a \( g\)-orthonormal frame for the vertical distribution of \(\pi\), there exist smooth functions \(W_{ij} \in \smooth(N)\) such that \(W_i = \sum_{j} W_{ij} E_j\)
        \begin{equation*}
            \sum_{j=1}^{2m} (\theta_j \otimes \theta_j)(W_1,W_2) = \sum_{j=1}^{2m}  g(E_j,W_1) \,  g(E_j, W_2) = \sum_{j=1}^{2m} W_{1j} W_{2j} = \sum_{j=1}^{2m}  g(W_1,W_2).
        \end{equation*}
        As before, \(\pi^*  g(W_1,W_2) = 0\) since \(d\pi W_j = 0\); thus, also in this case, \cref{submersionFancy} holds.
    \end{itemize}
    By bilinearity of \(g\), \cref{submersionFancy} holds for any pair of vector fields.
\end{proof}

In fact, the SKT and balanced conditions are preserved under averaging. Hence, the preceding \namecref{InvariantIsSubmersion} yields:

\begin{proposition}
    \label{AveragedSpecialMetrics}
    Let \(\pi : N \to X\) be a holomorphic principal torus bundle with \(\cpxTorus{m}\) as fiber; let \(g\) be a Hermitian metric on \(N\). Then,
    \begin{itemize}
        \item if \(g\) is SKT, there exists a submersion metric \(\tilde g\) which is SKT as well,
        \item if \(g\) is balanced, there exists a submersion metric \(\tilde g\) which is balanced as well.
    \end{itemize}
\end{proposition}
\begin{proof}
    We first address the proof of the SKT statement. So, let \(g\) be an SKT metric on \(N\) with fundamental form \(\omega\); this fulfils \(\del\delbar\omega=0\) since \(g\) is SKT. As before, for any \(t \in \cpxTorus{m}\), we can consider the holomorphic map \(a_t : N \to N\) induced by the principal action of the bundle. If \(d\mu_t\) denotes a (bi)invariant Haar measure of \(\cpxTorus{m}\), we can consider the symmetric tensor \(\tilde g\) defined on any \(W_1,W_2 \in \mathfrak{X}(N)\) as    
    \begin{equation*}
        \tilde g(W_1,W_2) = \int_{\cpxTorus{m}} a_t^* g (W_1,W_2) \, d\mu_t.
    \end{equation*} 
    This tensor is clearly symmetric and positive definite, moreover it preserves the complex structure \(J_N\) of \(N\) as \(a_t\) is holomorphic for any \(t\). Thus, \(\tilde g\) is an \(a\)-invariant Hermitian metric on \(N\). We can check this metric to be SKT. Indeed, if \(\tilde\omega\) denotes the fundamental form of \(\tilde g\), then
    \begin{equation*}
        \del\delbar \tilde\omega = \del\delbar \int_{\cpxTorus{m}} a_t^* \omega  \, d\mu_t = \int_{\cpxTorus{m}} \del\delbar  (a_t^* \omega) \, d\mu_t = \int_{\cpxTorus{m}}  a_t^* (\del\delbar \omega ) \, d\mu_t =  0;
    \end{equation*}
    whence \(\tilde g\) is SKT as well.
    
    The second part of the statement is an adaptation of \cite[Thm 2.2]{FinoGrantcharov2004} to our torus bundle case. Let \(\nu = \dim_{\C}N\) and assume there exists a balanced metric \(g\) on \(N\) with fundamental form \(\omega\): thus \(d\omega^{\nu - 1} = 0\). Under the same notations of the first part, we consider the \((2\nu-2)\)-form \(\eta\) defined as
    \begin{equation*}
        \eta(W_1,\dots, W_{2\nu-2}) = \int_{\cpxTorus{m}} a_t^*\omega^{\nu-1} (W_1,\dots, W_{2\nu-2}) d\mu_t.
    \end{equation*}
    Such \(\eta\) is \(a\)-invariant and strictly positive; thus, as discussed in \cite[279]{Michelsohn1982}, \(\eta = \tilde \omega^{\nu-1}\) for some strictly positive (1,1)-form \(\tilde\omega\). Moreover, \(d\tilde\omega^{\nu-1} = d\eta = 0\) and so \(\tilde g(\cdot,\cdot) = \tilde\omega(\cdot, J_N(\cdot))\) defines an \(a\)-invariant balanced metric.

    To conclude the proof, we observe that, since in both cases the Hermitian metrics \(\tilde g\) we constructed are \(a\)-invariant, they must be submersion metrics by \cref{InvariantIsSubmersion}.
\end{proof}

\subsection{Balanced metrics on regular LVMB manifolds}
Before addressing how to construct balanced metrics of submersion type on LVMB manifolds, we establish an obstruction to their existence.

In general, the existence of such metrics on a compact complex manifold \(M\) is obstructed in cohomology \cite{Michelsohn1982}. Indeed, if \(\PD(S) \in H^2_\text{dR}(M)\) denotes the Poincaré-dual class of a complex hypersurface \(S \subseteq M\), since \(\omega^{n-1}\) defines a closed transverse \((n-1,n-1)\)-form, we have
\begin{equation*}
    \int_M \omega^{n-1} \wedge \PD(S) = \int_S \omega^{n-1} = (n-1)! \int_S \star_{\omega|_S} 1 > 0.
\end{equation*}
Thus, \(\PD(S) \neq 0\) for any \(S\) or, equivalently, any such \(S\) represents a non-trivial cycle in homology.

In \cref{section:obstruction} we will show how this obstruction reflects on the algebro-combinatoric datum an LVMB manifold is associated to. For this, we first briefly recall some well-known notions about the integration along the fibers (see e.g. \cite[\S 7.12]{GreubHalperinVanstone1972}).

\subsubsection{Integration along the fibers}
Let \(\pi : P \to X\) be a smooth fiber bundle over a manifold \(X\). Suppose its fiber \(F\) is \(m\)-dimensional, compact and orientable. Let \(\omega \in \Omega^k(P)\) be a differential \(k\)-form. Consider a local trivialization for the bundle \(\{U_i\}_i\), then \(\omega|_{\pi^\inv(U_i)} \eqqcolon \omega_i\) decomposes as \(\omega_i = \eta_i \wedge \omega_i^\text{vert} + \hat\omega_i\) where \(\omega_i^\text{vert} \in \Omega^m(U_i)\) is supported in \(\ker d\pi\), \(\eta_i\) a suitable \((k-m)\)-form and \(\hat\omega_i\) the remainder term with vertical rank strictly lower than \(m\). Since the fiber is compact and by construction \(\int_F \hat\omega_i = 0\), we can integrate over the vertical component to produce a new form
\begin{equation*}
    \pi_* \omega_i = \left(\int_F \eta_i \wedge \omega_i^\text{vert} + \hat\omega_i\right) = \left(\int_F \omega_i^\text{vert}\right) \eta_i + \int_F \hat\omega_i = \left(\int_F \omega_i^\text{vert}\right) \eta_i \in \Omega^{k-m}(U_i).
\end{equation*}
Gluing these forms with a partition of the unit, we get a well-defined \((k-m)\)-form \(\pi_* \omega \in \Omega^{k-m}(X)\). Moreover, it can be easily checked that \(\pi_*\) commutes with the exterior differential, hence it descends to a map in cohomology \(\pi_* : H^k_\text{dR}(P) \to H^{k-m}_\text{dR}(X)\). In particular, a well-known fact, often called \emph{projection formula}, holds:
\begin{proposition}[{\cite[Prop. IX]{GreubHalperinVanstone1972}}]
    \label{projectionFormula}
    Let \(\pi : P \to X\) be as above. If \(X\) is compact and orientable, for any \(\eta \in \Omega^k(P)\) and \(\tau \in \Omega^{\dim P - k}(X)\), we have
    \begin{equation*}
        \int_P \eta \wedge \pi^*\tau = \int_X \pi_* \eta \wedge \tau.
    \end{equation*}
\end{proposition}
Another well-known property of the integration along the fibers is its compatibility with the pullback map:

\begin{proposition}[{\cite[\S 7.12 Corollary II]{GreubHalperinVanstone1972}}]
    \label{intAlongFibersPB}
    Let \(\pi : P \to X\) be a principal bundle as above, and let \(f : Y \to X\) be a smooth map with \(d = \dim Y\). Consider the pulled-back bundle \(f^*\pi : f^*P \to Y\) with the natural map \(\tilde f : f^*P \to P\). Then, for any \(\omega \in \Omega^k(P)\),
    \begin{equation*}
        (f^*\pi)_*(\tilde f^* \omega) = f^*(\pi_* \omega).
    \end{equation*}
\end{proposition}

\subsubsection{An obstruction to balanced metrics}
\label{section:obstruction}
We can exploit the integration along the fibers of the LVMB bundle to obtain an obstruction to existence of balanced metrics on regular LVMB manifolds. First, we need the following 

\begin{lemma}
    \label{DivisorObstruction}
    Let \(\pi : N \to X\) be the principal torus bundle associated to a regular LVMB manifold. Let \(d = \dim_{\C} X\) and \(m\) be the complex dimension of the torus fiber. If \(X\) admits a torus-invariant divisor \(D\) such that \(\pi^* \PD(D) = 0\), then the total space \(\iota^* N\) of the bundle pull-back along the inclusion \(\iota : D \to X\) is a complex hypersurface of \(N\) whose corresponding singular homology class \([\iota^* N]\) is trivial.
\end{lemma}
\begin{proof}
    Since the bundle pullback preserves injections, \(\tilde\iota : \iota^* N \to N\) is holomorphic and injective. Since \(D\) is effective and irreducible, its support is a hypersurface of \(X\), thus also the total space of the pull-back \(\iota^* N\) is a hypersurface of \(N\). By construction \(\dim_{\C} N = d+m\), thus \(\dim_\C \iota^* N = d+m-1\).
    
    By Poincaré duality,  its homology class vanishes if and ony if its Poincaré-dual class \(\PD(\iota^*N) \in H^2_\text{dR}(N)\) does so. To compute such class, consider any \(\eta \in \Omega^{2d+2m-2}(N)\). Then by \cref{intAlongFibersPB,projectionFormula} we have
    \begin{equation*}
        \int_{\iota^* N} \iota^*\eta = \int_D (\iota^*\pi)_* (\tilde\iota^*\eta) = \int_D \iota^* \pi_* \eta = \int_X \pi_* \eta \wedge \PD(D) = \int_N \eta \wedge \pi^* \PD(D);
    \end{equation*}
    Therefore, \(\PD(\iota^*N) = \pi^* \PD(D)\) in de Rham cohomology. Since the latter vanishes by hypothesis, by Poincaré duality we deduce that \([\iota^* N] = 0\) in singular homology. 
\end{proof}

The existence of an irreducible effective divisor of the base such that \(\pi^*\PD(D) = 0\) is not uncommon among LVMB manifolds. Indeed, as we observed in \cref{trivial_pullback}, this happens in any regular LVMB manifold containing a ghost vector in its triangulated configuration. This fact provides the following obstruction to the existence of balanced metrics.
\begin{theorem}
    \label{balancedness_obstruction}
    Let \((V,\calT)\) be a regular vector configuration containing at least one ghost vector. Then, any corresponding LVMB manifold admits a balanced metric if and only if it is a complex torus.
\end{theorem}
\begin{proof}
    Let \(N\) be any LVMB manifold associated to \((V_\Sigma,\calT_\Sigma)\) and \(\pi : N \to X_\Sigma\) the principal torus bundle it comes with. As we discussed in \cref{trivial_pullback}, \(\pi^*\PD(D) = 0\) if \((V_\Sigma,\calT_\Sigma)\) contains a ghost vector, thus balanced metrics are obstructed unless \(X_\Sigma\) has no divisors. This condition is very restrictive for a toric variety. Indeed, this forces \(b_2(X_\Sigma) = 0\). Nevertheless, the Betti numbers of a smooth \(d\)-dimensional toric variety \(X_\Sigma\) are linked to the combinatorics of its fan with the celebrated Stanley formula (see e.g. \cite[Thm.12.3.12]{CoxLittleSchenck2011})
    \begin{equation*}
        b_{2j}(X_\Sigma) = \sum_{i=j}^{d} (-1)^{j-i} \binom{i}{j} |\Sigma^{(d-i)}|,
    \end{equation*}
    where \(|\Sigma^{(j)}|\) denotes the number of \(j\)-dimensional cones in the fan. Since \(X_\Sigma\) is compact, we can use Poincaré duality to write \(b_2(X_\Sigma) = b_{2d-2}(X_\Sigma) = |\Sigma^{(1)}| - d\). Hence, \(b_2(X_\Sigma)\) vanishes if and only if \(\Sigma\) has \(d\) rays. This never happens in a complete fan of positive dimension. Indeed, positive linear combinations of \(d\) vectors in \(\R^d\) span at most a strongly convex polyhedral cone, which is strictly contained in \(\R^d\) if \(d > 0\).
\end{proof}

\subsubsection{An equation for balanced metrics}
\label{section:submersionMetricsBalanced}
In \cite{GrantcharovPoon2008}, the authors characterize the submersion metrics on a holomorphic principal torus bundle that are balanced. As this is precisely the structure underlying a regular LVMB manifold, we can apply their result to produce such metrics on these manifolds. Tailored to our specific case, their result reads as
\begin{proposition}[{\cite[Lemma 2]{GrantcharovPoon2008}}]
    \label{balancedGGP}
    Let \(\pi : N \to X\) be the holomorphic torus bundle associated to an LVMB manifold \(N\); let \(J_N\) be its complex structure. Consider a principal connection \(\theta \in \Omega^1(N;\R^{2m})\) with components \(\theta_j\)'s such that \(J_N \theta_{2j-1} = \theta_{2j}\) and \(\Theta_j \in \Omega^{2}(X)\) the corresponding curvature forms. Then, the submersion metric \(g_N = \pi^* g_X + \sum_j \theta_j \otimes \theta_j\) constructed as in \cref{section:submersion_metric} is balanced if and only if
    \begin{equation}
        \label{eq:balancedGGP}
        \pi^* (d^*\omega_X) + \sum_{j=1}^{2m} \pi^*(\Lambda_{\omega_X}\Theta_j) \theta_j = 0.
    \end{equation}
\end{proposition}
If \(d^* \omega_X = 0\), i.e. if \(g_X\) is balanced, \cref{eq:balancedGGP} reduces to \(\Lambda_{\omega_X}\Theta_j = 0\) for each \(j=1,\dots,2m\). This is the condition we are going to use in the next section to produce new examples of balanced LVMB manifolds. These show that the hypothesis on the number of ghost vectors in \cref{balancedness_obstruction} cannot be relaxed.

\subsubsection*{A balanced example}
Consider the LVMB manifold described in the example \labelcref{BalancedLVM}. As shown there, the characteristic class of the torus bundle \(\pi : N_3 \to (\CP^1)^3\) is
\begin{equation*}
    \gamma(\pi) = h_2 \otimes (w_1-w_5) + h_3 \otimes (w_3-w_5).
\end{equation*}

In general, the torus-invariant divisors of \(\CP^n\) are all linearly equivalent: indeed, as Cartier divisors, they all define the line bundle \(\mathcal{O}(-1)_{\CP^n}\). Therefore, if \(\omega_\text{FS}\) denotes the Fubini--Study metric on \(\CP^n\), the corresponding Poincaré-dual class contains \(-\omega_\text{FS}\). In our setting, this reduces to observing that, if \(p_j : (\CP^1)^3 \to (\CP^1)_j\) is the projection on the \(j\)-th factor, \(w_{2j-1} = -p_j^*[\omega_\text{FS}]\).

To construct a submersion metric on the LVMB bundle \(N_3 \to (\CP^1)^3\), consider \(\omega_{(\CP^1)^3} = \omega_1 + \omega_2 + \omega_3\) as Kähler metric on the base. On the fibers, instead, we consider any complex valued connection \(\tilde\theta \in \Omega^1(N;\C)\). This induces two real valued connection 1-forms \(\theta_1, \theta_2\) on the bundle components \(\pi_1,\pi_2\). Because of \cref{lemma:charClassGhosts} the corresponding curvature forms fulfil \(\Theta_1 \in w_3-w_5, \Theta_2 \in w_1-w_5\). However, even if these classes are respectively represented by the \(\omega_{(\CP^1)^3}\)-harmonic forms \(-\omega_2 + \omega_5\) and \(-\omega_1+\omega_5\), we cannot assure that these are attained as curvature forms by the complex valued connection \(\tilde\theta\). In fact, tweaking separately the connection forms \(\theta_j\) in order to make their curvature forms harmonic, may break the condition \(\theta_2 = J_N \theta_1\).

However, in this particular example, the two components of \(\pi\) are symmetrical, thus one can choose \(\tilde\theta\) such that \(\Theta_1 = -\omega_1+\omega_3\) and \(\Theta_2 = -\omega_2 + \omega_3\).

Such connection can be constructed as follows. If \((\alpha,\beta,\gamma) \in \{0,1\}^3\), we can consider the canonical open covering \(U_{\alpha\beta\gamma} = U_\alpha \times U_\beta \times U_\gamma \subset (\CP^1)^3\). This supports a trivialization of \(\pi\) given by the biholomorphic maps
\begin{equation*}
    \begin{tikzcd}
            \pi^\inv U_{\alpha\beta\gamma} \ar[r] & U_{\alpha\beta\gamma} \times \cpxTorus{1}\\[-2em]
            {[z_1 : \dots : z_6]}_N \ar[r] & \left(\big((z_1:z_2), (z_3:z_4),(z_5:z_6), \log\frac{z_{\alpha+1} z_{\beta+3}^i}{z_{\gamma+5}^{i+1}}\big)\right)
    \end{tikzcd}
\end{equation*}
with holomorphic clutching functions \(g_{\alpha\beta\gamma, \alpha'\beta'\gamma'} : U_{\alpha\beta\gamma} \cap U_{\alpha'\beta'\gamma'} \to \C\). Notice that, as \(U_{000}\) is dense, these are completely determined by \(g_{\alpha\beta\gamma,000}\). If \(u_j\) are complex coordinates for this chart, their expressions are:
\allowdisplaybreaks
\begin{alignat*}{3}
    g_{001,000} &= \frac{1-i}{2\pi}\log u_3  &\qquad g_{101,000} &= \frac{i}{2\pi}\log u_1 + \frac{1-i}{2\pi} \log u_3\\
    g_{010,000} &= -\frac{1}{2\pi} \log u_2 & g_{110,000} &= \frac{i}{2\pi}\log u_1 - \frac{1}{2\pi} \log u_2\\
    g_{011,000} &= -\frac{1}{2\pi} \log u_2 + \frac{1-i}{2\pi} \log u_3 & g_{111,000} &= \frac{i}{2\pi} \log u_1 - \frac{1}{2\pi} \log u_2 + \frac{1-i}{2\pi} \log u_3\\
    g_{100,000} &= \frac{i}{2\pi} \log u_1
\end{alignat*}
As these are \v{C}ech 1-cocycles, the 1-forms
\begin{equation*}
    \tilde\theta_{\alpha\beta\gamma} = d g_{\alpha\beta\gamma,000}
\end{equation*}
fulfil the gauge transformation law for connections; hence they define a complex-valued connection for \(\pi\). Taking real and imaginary parts, we get the local expressions for the connection 1-forms \(\theta_1, \theta_2\). Explicitly,

\begin{alignat*}{3}
    (\theta_1)_{001} &= -\left(\frac{i - 1}{4\pi}\right)  \frac{du_3}{u_3} + \left( \frac{i + 1}{4\pi} \right)  \frac{d\bar{u}_{3}}{\bar u_3}\\
    (\theta_1)_{010} &= -\frac{1}{4\pi}\left(\frac{{du}_{2}}{u_2} + \frac{d\bar{u}_{2}}{\bar{u}_{2}}\right)\\
    (\theta_1)_{011} &= -\frac{1}{4 \, \pi}\left(\frac{{du}_{2}}{u_{2}} + \frac{d\bar{u}_{2}}{\bar{u}_{2}}\right) - \left( \frac{i - 1}{4 \, \pi} \right)  \frac{{du}_{3}}{u_{3}} + \left( \frac{i + 1}{4 \, \pi} \right)  \frac{d\bar{u}_{3}}{\bar{u}_{3}}\\
    (\theta_1)_{100} &= \frac{i}{4 \, \pi}\left(\frac{{du}_{1}}{u_{1}} - \frac{d\bar{u}_{1}}{\bar{u}_{1}}\right)\\
    (\theta_1)_{101} &= \frac{i}{4 \, \pi}\left(\frac{{du}_{1}}{u_{1}} - \frac{d\bar{u}_{1}}{\bar{u}_{1}}\right) - \left( \frac{i - 1}{4 \, \pi} \right)  \frac{{du}_{3}}{u_{3}} + \left( \frac{i + 1}{4 \, \pi} \right)  \frac{d\bar{u}_{3}}{\bar{u}_{3}}\\
    (\theta_1)_{110} &= \frac{i}{4 \, \pi}\left(\frac{{du}_{1}}{u_{1}} - \frac{d\bar{u}_{1}}{\bar{u}_{1}}\right) - \frac{1}{4 \, \pi}\left(\frac{{du}_{2}}{u_{2}} + \frac{d\bar{u}_{2}}{\bar{u}_{2}}\right)\\
    (\theta_1)_{111} &= \frac{i}{4 \, \pi}\left(\frac{{du}_{1}}{u_{1}} - \frac{d\bar{u}_{1}}{\bar{u}_{1}}\right) - \frac{1}{4 \, \pi}\left(\frac{{du}_{2}}{u_{2}} + \frac{d\bar{u}_{2}}{\bar{u}_{2}}\right) - \left( \frac{i - 1}{4 \, \pi} \right)  \frac{{du}_{3}}{u_{3}} + \left( \frac{i + 1}{4 \, \pi} \right)  \frac{d\bar{u}_{3}}{\bar{u}_{3}}.
\end{alignat*}
Since \((\theta_2)_{\alpha\beta\gamma} = J(\theta_1)_{\alpha\beta\gamma}\), using \(J du_j = -i du_j\) and \(J d\bar u_j = i d\bar u_j\), these also determine the local expressions of \(\theta_2\).

Finally, to compute the curvature forms, we use the following partition of the unit: let
\begin{equation*}
    r_j^0 = \frac{1}{1 + |u_j|^2} \quad \text{and} \quad r_j^1 = 1-r_j^0 = \frac{|u_j|^2}{1 + |u_j|^2} \quad \text{for \(j=1,2,3\)},
\end{equation*}
then \(\rho_{\alpha\beta\gamma} = r_1^\alpha r_2^\beta r_3^\gamma\) is a partition of the unit supported by the open covering \(U_{\alpha\beta\gamma}\). Hence, the expressions of the curvature forms on the dense chart \(U_{000}\) are given by
\begin{align*}
    \Theta_1 = \sum_{\alpha\beta\gamma} d\rho_{\alpha\beta\gamma} \wedge (\theta_1)_{\alpha\beta\gamma} = -\frac{i\, du_1 \wedge d\bar u_1}{2\pi (1 + |u_1|^2)^2} + \frac{i\, du_3 \wedge d\bar u_3}{2\pi (1 + |u_3|^2)^2} = -\omega_1 + \omega_3\\ 
    \Theta_2 = \sum_{\alpha\beta\gamma} d\rho_{\alpha\beta\gamma} \wedge (\theta_2)_{\alpha\beta\gamma} =  -\frac{i\, du_2 \wedge d\bar u_2}{2\pi (1 + |u_2|^2)^2} + \frac{i\, du_3 \wedge d\bar u_3}{2\pi (1 + |u_3|^2)^2} = -\omega_2 + \omega_3
\end{align*}
We can easily check that these forms are primitive with respect to the Kähler metric \(\omega_{(\CP^1)^3}\). Indeed, if \(\star\) and \(\Lambda\) respectively denote its corresponding Hodge-star and dual-Lefschetz operators, then
\begin{align*}
    \star\Lambda \Theta_1 &= \star\left\langle \Theta_1, \omega_{(\CP^1)^3}\right\rangle = (\star \omega_{(\CP^1)^3} \wedge \Theta_1) = \frac{1}{3}(\omega_2 \wedge \omega_3 + \omega_1 \wedge \omega_3 + \omega_1 \wedge \omega_2) \wedge (\omega_3-\omega_1) = 0, \\
    \star\Lambda \Theta_2 &= \star\left\langle \Theta_2, \omega_{(\CP^1)^3}\right\rangle = (\star \omega_{(\CP^1)^3} \wedge \Theta_2) = \frac{1}{3}(\omega_2 \wedge \omega_3 + \omega_1 \wedge \omega_3 + \omega_1 \wedge \omega_2) \wedge (\omega_3-\omega_2) = 0,
\end{align*}
as \(\omega_j^2 = 0\). Thus, by \cref{balancedGGP}, the submersion metric \(\pi^*\omega_{(\CP^1)^3} + \theta_1 \otimes \theta_1 + \theta_2 \otimes \theta_2\) is a balanced metric on \(N_3\).

\subsubsection*{Comparison with known examples}

To the best of the author's knowledge, the complex manifold \(N_3\) described above is an example of complex manifold admitting a Hermitian balanced metric not appearing in literature. Indeed, all the examples known to the author belong at least to one of the following classes:%
\begin{description}
    \item [Fujiki class \(\mathcal{C}\) manifolds] These are complex manifolds obtained as bimeromorphic modifications of Kähler manifolds. As shown in \cite{AlessandriniBassanelli1995}, these all admit a balanced metric. Moreover, they all fulfil the \(\del\delbar\)-lemma by \cite[Cor.28]{Stelzig2021}. This result can be applied in this case because complex submanifolds of a Kähler one always fulfil the \(\del\delbar\)-lemma.
    \item[Twistor spaces] Twistor spaces are a natural construction associating a complex manifold to a \(4n\)-dimensional quaternionic Kähler manifold \cite{Salamon1982}; for \(n = 1\) this reduces to the original construction based on anti-self-dual conformal structure on a Riemannian 4-manifold \cite{AtiyahHitchinSinger1978}. As these spaces arise as \(\CP^1\)-bundles, they all have odd complex dimension.
    
    If \(M\) is hypercomplex, the fibration associated to the twistor space is trivialized in the smooth category \cite{HitchinKarlhedeLindstromRocek1987}. In this case, the projective fiber can be replaced with a generic complex manifold \cite{LinZheng2017}. This construction produces generalized twistor spaces in every complex dimension.
    \item[Quotients of Lie groups] Many examples of compact Hermitian manifolds arise as quotients of Lie groups by co-compact lattices. Those admitting an invariant metric are precisely the holomorphically parallelisable manifolds \cite{Wang1954}, and such manifolds always support a balanced metric \cite{AbbenaGrassi1986}. More generally, within a special class of quotients of Lie groups by co-compact lattices, those admitting a balanced metric have been characterized: see \cite{FinoParadiso2025} and references therein. Other examples of balanced manifolds related to Lie groups are generalized flag manifolds \cite{FinoGrantcharovVezzoni2019}, some non-compact simple real Lie groups of inner type, and their compact quotients by co-compact lattices \cite{GiustiPodesta2023}.
    \item[Fu--Li--Yau conifold transitions] In \cite{FuLiYau2012}, the authors showed that conifold transitions of Kähler Calabi--Yau threefolds are always balanced. In a nutshell, this consists of contracting a \((-1,-1)\) curve and then smoothing the resulting singular variety. With this technique, the connected sums \(\#_k S^3 \times S^3\) for \(k \geq 2\) can be given a complex structure supporting a smooth balanced metric.
    \item[Toric suspensions] In \cite{FinoGrantcharovVerbitsky2025} the authors provide a technique to produce new balanced manifolds from known ones. Namely, given a complex manifold \(M\) and a set of commuting automorphisms \(f_1,\dots, f_{2m} \in \operatorname{Aut}(M)\), they define a natural \(\Z^{2k}\)-action on \(M \times \C^{2k}\). The toric suspension is the quotient \((M\times \C^{2k})/\Z^{2k}\); clearly, this fits in the fibration \(\Z^{2k} \to M \times \C^{2k} \to (M \times \C^{2k})/\Z^{2k}\).
\end{description}
\begin{proposition}
    The LVMB manifold \(N_3 \to (\CP^1)^3\) described above supports a Hermitian balanced metric. Nevertheless, it does not belong to any of the classes previously described.
\end{proposition}
\begin{proof}
    We already proved that the natural submersion metric on \(N_3\) is balanced. It remains to check the second part of the statement.

    As \(N_3\) is not a torus, by \cref{deldelbar_lemma} it does not fulfil the \(\del\delbar\)-lemma. Hence, it cannot be a Fujiki class \(\mathcal{C}\) manifold. Neither is it a conifold transition of some Calabi--Yau threefold as \(\dim_{\C} N_3 = 4\).

    For the same reason, \(N_3\) is not a twistor space of any quaternionic Kähler manifold: indeed such manifolds always have odd complex dimension. Neither is \(N_3\) a generalized twistor space: if it were so, then \(N_3\) would be diffeomorphic to a product of compact complex surfaces \(S \times H\) with \(H\) admitting a hypercomplex structure. Such \(H\) must be either a complex torus, a Hopf surface or a K3 surface \cite[Thm. 1]{Boyer1988}. But, as we showed in \cref{BalancedLVM}, \(N_3\) fits also in the fibration \(S^3 \times S^3 \hookrightarrow N_3 \to \CP^1\), hence its Puppe sequence yields \(\pi_1(N_3) = 0\) and \(\pi_2(N_3) \iso \Z\). This is never the case of the products \(S \times H\). Indeed, this is not simply connected unless \(H\) is a K3 surface; in this case \(b_2(S \times H) = b_2(S) + b_2(H) = b_2(S) + 22 \), while \(b_2(N_3) = 1\) by Hurewicz Theorem.

    Similarly, we can show that \(N_3\) is not a toric suspension. Indeed, the Puppe sequence of any toric suspension splits as
    \begin{equation*}
        0 \to \pi_1(M \times \C^{2k}) \to \pi_1\left(\frac{M \times \C^{2k}}{\Z^{2k}}\right) \to \pi_0(\Z^{2k}) \to 0,
    \end{equation*}
    thus non-trivial toric suspensions are never simply connected.

    It remains to exclude all examples related to Lie groups. Of course, \(N_3\) has not the homotopy type of a quotient \(\Gamma\backslash G\) of a Lie group by a co-compact lattice. Indeed, as \(\pi_2(G) = 0\) for any Lie group, also \(\pi_2(\Gamma\backslash G) = 0\). Similarly, we can see that \(N_3\) does not have the homotopy type of a generalized flag manifold. Indeed, Schubert stratification shows that these are homeomorphic to the CW-complex obtained attaching their Schubert cells (see \cite{Brion2004} for a self-contained presentation of Schubert calculus and references therein). As these are all isomorphic to complex affine spaces, their CW-complex structure is generated by cells of even real dimension; in particular, they have positive Euler characteristic. Contrarily, \(\chi(N_3)\) vanishes since, by \cref{Chern_Class_vanish}, \(c_4(N_3) = 0\) and the Gauss--Bonnet--Chern Theorem forces \(\chi(N_3) = \int_N c_4(N_3) = 0\).
\end{proof}
\subsubsection*{Higher dimensional examples}

Similar considerations can be performed to construct higher dimensional examples. Indeed, we can consider \((\CP^1)^{2d +1}\) as base toric variety, if \(\Sigma\) is its associated fan in \(\R^{2d+1}\), the triangulated configuration \((V_\Sigma, \calT_\Sigma)\) is odd and balanced with \(m = d\). If \(\{\varepsilon_j\}_j\) is the canonical basis of \(\R^{2d+1}\), the vector configuration is \(V_\Sigma = (\varepsilon_j, -\varepsilon_j : j=1,\dots, 2d+1)\). An ordered basis for \(K_\Sigma\) is given by
\begin{equation*}
    \left(\sum_{j=1}^{4d+2} \epsilon_j, \epsilon_1 + \epsilon_2, \dots, \epsilon_{4d-1} + \epsilon_{4d}\right),
\end{equation*}
with \(\{\epsilon_j\}_j\) denoting the canonical basis of \(\R^{4d+2}\). For such basis, the Gale dual configuration in \(\mathbb{A}_\C^m\) is given by the points
\begin{equation*}
    \Lambda_{2j-1} = \Lambda_{2j} = 
    \begin{cases}
        e_j & \text{for \(j\) odd} \\
        i e_j & \text{for \(j\) even}
    \end{cases}
    \quad \text{for \(j = 1,\dots, 2d+1\)}.
\end{equation*}
Thus, the corresponding action over \(\proj(U(\calT)) = \proj(\C^2 \setminus\{0\}^{4d+2})\) has the following expression:
\begin{equation*}
    (t_1, \dots, t_{d}).(z_1: \dots: z_{4d+2}) = (\dots: e^{2\pi i t_j} z_{2j-1}: e^{2\pi i t_j} z_{2j}: e^{-2\pi t_{j}} z_{2j+1}: e^{-2\pi t_{j}} z_{2j+2}: \dots ),
\end{equation*}
with \(j = 1, \dots, 2d-1\). Taking its quotient, we get a \((3d+1)\)-dimensional regular LVMB manifold \(N_{2d+1}\) fitting in the bundle  \(\cpxTorus{d} \hookrightarrow N_{2d+1} \xrightarrow{\pi} (\CP^1)^{2d+1}\). Its characteristic class as principal torus bundle is then
\begin{equation}
    \label{niceFamilyClass}
    \gamma(\pi) = \sum_{j=1}^{2m} (w_{2j-1} - w_{4d+1}) \otimes h_{j+1}.
\end{equation}

Let \(\phi_{j} : (\CP^1)^{2d+1} \to (\CP^1)^3\) for \(j = 1,\dots, m\) be the projection onto the first, the \((2j-1)\)\th and the \((2j)\)\th component, and let \(\pi' : N_3 \to (\CP^1)^3\) be the regular LVMB manifold described above. Then, we can see that
\begin{equation*}
    \pi \iso \tilde\Delta^*\left(\prod_{j=1}^m \phi_j^*\pi'\right),
\end{equation*}
where \(\tilde\Delta : (\CP^1)^{2d+1} \to (\CP^1)^{3d}\) is the unique map closing the following diagram:
\begin{equation*}
    \begin{tikzcd}
        (\CP^1)^{2d} \ar[dd, "\operatorname{pr}_j \times \operatorname{pr}_{j+1}"] & (\CP^1)^{2d+1} \ar[l,"\prod_{j=2}^{2d+1}\operatorname{pr}_j"'] \ar[d, dashed, "\tilde\Delta"] \ar[r, "\operatorname{pr}_1"] & \CP^1 \ar[dd,equals]\\
        & \prod_{j=1}^d (\CP^1)^3 \ar[d,"\operatorname{pr}_j"]\\
        \CP^1 \times \CP^1 & (\CP^1)^3 \ar[l, "\operatorname{pr}_2 \times \operatorname{pr}_3"'] \ar[r,"\operatorname{pr}_1"] & \CP^1
    \end{tikzcd}
\end{equation*}
Hence, we see that \(\pi\) admits connection 1-forms \(\theta_j\) with curvature forms \( \Theta_j = -\omega_j + \omega_{2d+1} \in w_{2j-1} - w_{4d+1}\) for \(j=1,\dots, 2d\), where \(\omega_j = \operatorname{pr}_j^*(\omega_\text{FS})\) is the pull-back of the Fubini-Study metric of \(\CP^1\) on the \(j\)\th factor. As each \(\Theta_j\) is harmonic with respect to \(\sum_{j=1}^{2d+1} \omega_j\), we have that the (1,1)-form
\begin{equation*}
    \omega_N = \pi^*\left(\sum_{j=1}^{2d+1} \omega_j\right) + \sum_{j=1}^{d} \theta_{2j-1} \wedge \theta_{2j}
\end{equation*}
is co-closed and defines a Hermitian metric on \(N_{2d+1}\).

Notice that this generalized construction enjoys a secondary bundle structure as well. Namely, similarly to the case of \(N_3\), the projection onto the last two coordinates of \(\proj(U(\calT_\Sigma))\) induces the holomorphic principal bundle:
\begin{equation*}
    (M_{1,1;i})^d \hookrightarrow N_{2d +1} \longrightarrow \CP^1
\end{equation*}
where \(M_{1,1;i}\) is the standard Calabi--Eckmann threefold of period \(i\). Such bundle is principal with respect to the natural action of \(M_{1,1;i} \iso \SU(2) \times \SU(2)\) onto itself. This is also holomorphic as the complex structure of \(M_{1,1;i}\) is left-invariant.

Similarly, we can perform the same construction starting from \((\CP^1)^{2d}\). If \(\Sigma\) is the associated polyhedral fan in \(\R^{2d}\), the corresponding triangulated configuration \((V_\Sigma, \calT_\Sigma)\) is balanced but never odd. Hence, we need to prepend \(0 \in \R^{2d}\) as ghost vector to turn it into an odd and balanced one. Denoting it as \((V,\calT)\), we have \(V = (0, \varepsilon_j, -\varepsilon_j : j=1,\dots, 2d)\), where \(\{\varepsilon_j\}\) is the canonical basis of \(\R^{2d}\). Notice that \(4d+1-2d = 2d+1\), hence also in this case the fiber torus will have complex dimension \(m = d\). As ordered basis for \(\Rel(V)\) we take
\begin{equation*}
    \left(\sum_{j=1}^{4d} \epsilon_j, \epsilon_1+ \epsilon_2, \dots, \epsilon_{4d-3} + \epsilon_{4d-2}\right);
\end{equation*}
 where \(\{\epsilon_j\}_j\) is the canonical basis of \(\R^{4d}\). The corresponding Gale dual configuration in \(\mathbb{A}^d_{\C}\) is then given by the points
\begin{equation*}
    \Lambda_{2j-1} = \Lambda_{2j} =
    \begin{cases}
        e_j &\text{for \(j\) odd}\\
        i e_j &\text{for \(j\) even}
    \end{cases}
    \quad \text{for \(j=1,\dots, 2d\)}.
\end{equation*}
Thus, the corresponding action over \(\proj(U(\calT)) \iso (\C^2 \setminus\{0\})^{4d}\) has the following expression:
\begin{equation*}
    (t_1, \dots, t_{d}).(z_1, \dots, z_{4d}) = (\dots, e^{2\pi i t_j} z_{2j-1}, e^{2\pi i t_j} z_{2j}, e^{-2\pi t_{j}} z_{2j+1}, e^{-2\pi t_{j}} z_{2j+2}, \dots ),
\end{equation*}
where \(j\) is taken odd. We can recognize in each block of four adjacent components the action defining the Calabi--Eckmann threefold \(M_{1,1;i}\). As these blocks can be split, at the end we can see that the corresponding LVMB manifold \(N_{2d} \iso (M_{1,1;i})^d\). For \(d=1\) we obtain the standard construction of the Calabi--Eckmann threefold. In particular, \(N_{2d}\) always admits an SKT metric.

These observations can be summarized in the following:

\begin{proposition}
    \label{niceBalancedFamily}
    Let \(\ell \geq 2\) be an integer, and let \(N_\ell \to (\CP^1)^\ell\) be the LVMB manifold constructed as above. Then \(\{N_\ell\}_\ell\) is a family of compact complex simply connected manifolds such that \(\dim_{\C} N_\ell = \ell + \left\lfloor\frac{l}{2} \right\rfloor\). Moreover, \(N_\ell\) has a balanced metric when \(\ell\) is odd, while it admits an SKT one when \(\ell\) is even. Finally, for any positive integer \(j\), \(N_{2j} \hookrightarrow N_{2j+1} \to \CP^1\) is a holomorphic principal bundle. 
\end{proposition}

\subsection{SKT metrics on regular LVMB manifolds}
The properties of submersion metrics on holomorphic principal torus bundles can be also exploited to construct SKT metrics on regular LVMB manifolds. As we did in \cref{section:submersionMetricsBalanced}, we start by adapting a theorem in \cite{GrantcharovPoon2008} in the context of LVMB manifolds.

\begin{proposition}[{\cite[Prop. 14]{GrantcharovPoon2008}}]
    \label{GGP}
    Let \(\pi : N \to X\) be the holomorphic torus bundle associated to an LVMB manifold \(N\); let \(J_N\) be its complex structure. Consider a principal connection \(\theta \in \Omega^1(N;\R^{2m})\) with components \(\theta_j\)'s such that \(J_N \theta_{2j-1} = \theta_{2j}\) and \(\Theta_j \in \Omega^{2}(X)\) the corresponding curvature forms. Then, the submersion metric \(g_N = \pi^* g_X + \sum_j \theta_j \otimes \theta_j\) constructed as in \cref{section:submersion_metric} is an SKT metric if and only if
    \begin{equation*}
        dd^c\omega_X = \sum_{j=1}^m(\Theta_{2j-1}^2 + \Theta_{2j}^2).
    \end{equation*}
\end{proposition}

In order to prove a stronger result about the existence of SKT metrics, we need to restrict ourselves to LVM manifolds \cite{Meersseman2000} because in this case the toric base space \(X_\Sigma\) is projective \cite{MeerssemanVerjovsky2004}. From the viewpoint of the initial datum \((V_\Sigma,\calT_\Sigma)\), \(N\) is an LVM manifold if and only if the fan \(\Sigma\) is \emph{polytopal}, namely, when it is the normal fan of a simple complex polytope \cite{Cupit-FoutouZaffran2007}. This is the only case in which \(X_\Sigma\) is Kähler: indeed, since toric manifolds are Moishezon \cite[Thm. 6.1.18]{CoxLittleSchenck2011}, they admit Kähler metrics if and only if they are projective \cite{7Papers1967}.

\begin{lemma}
    \label{SKTcohomology}
    Let \(N\) be a regular LVMB manifold with corresponding holomorphic principal torus bundle \(\pi : N \to X_\Sigma\); let \(\pi_j\) be the principal \(S^1\)-bundles components. If \(N\) admits an SKT metric then \(\sum_j c_1(\pi_j)^2 = 0\). Furthermore, if \(N\) is an LVM manifold, i.e. if \(X_\Sigma\) is projective, the converse holds as well. 
\end{lemma}
\begin{proof}    
    Suppose \(N\) admits an SKT metric \(g\) with fundamental form \(\omega\). Because of \cref{AveragedSpecialMetrics}, it is not restrictive to assume \(g\) to be a submersion metric. Thus, its fundamental form fulfils
    \begin{equation*}
        0 = dd^c \omega_N = \pi^* \left(dd^c \omega_X - \sum_{j=1}^{2m} \Theta_j^2\right).
    \end{equation*}
    Since \(\pi^* : \mathcal{A}^{2,2}(X) \to \mathcal{A}^{2,2}(N)\) is injective and \([\Theta_j] = c_1(\pi_j)\) in de Rham cohomology, this forces \(\sum_{j=1}^{2m} c_1(\pi_j)^2 = 0\).

    For the second part, we moreover assume that \(N\) is an LVM manifold such that \(\sum_{j=1}^{2m} c_1(\pi_j)^2 = 0\). As in this case \(X_\Sigma\) is Kähler, we can apply \cite[Lemma 4]{CobanTurcoPoddar2021} to construct a submersion metric on \(N\) which is SKT.
\end{proof}
\begin{remark}
    The SKT metric constructed in \cite[Lemma 4]{CobanTurcoPoddar2021} might not cover a Kähler metric on the base. The reason is that in that proof, a Kähler metric on the base has to be tweaked so that the equation in \cref{GGP} is fulfilled. Therefore, also in \cref{SKTcohomology}, the SKT metric whose existence is ensured might not be the horizontal lift of any Kähler metric defined on the base.
\end{remark}

Combining \cref{SKTcohomology} with the expression of the characteristic class of the LVMB bundle \(\pi : N \to X\) in \cref{bundles_over_toric_manifolds}, we get an obstruction for the existence of SKT metrics on \(N\) in terms of the characteristic class \(\gamma(\pi)\). As before, when restricting ourselves to LVM manifolds, we obtain a complete characterization.
\begin{theorem}
    \label{LVMB_SKT}
    Let \(\pi : N \to X_\Sigma\) be the principal torus bundle associated to a regular LVMB manifold \(N\), and let \(\gamma(\pi) = \sum_{j=1}^{2m} \iota^\vee(f_j^*) \otimes f_j\) be its characteristic class. Then, \(N\) admits an SKT metric only if
    \begin{equation*}
        \sum_{j=1}^{2m} \iota^\vee(f_j^*) \wedge \iota^\vee(f_j^*) = 0 \in H^4(X;\Z).
    \end{equation*}
    Furthermore, if \(N\) is LVM, or equivalently \(X_\Sigma\) is projective, this condition is also sufficient.
\end{theorem}

In the class of manifolds described in \cref{niceBalancedFamily}, balanced and SKT metrics never coexist. Indeed, we already know that for \(\ell\) even, \(N_\ell\) does not admit any balanced metric as it is constructed from a triangulated vector configuration containing a ghost vector, while we can construct an SKT metric on it, as it is biholomorphic to the product of Calabi--Eckmann threefolds. On the other hand, for \(\ell\) odd, we proved that \(N_\ell\) admits a balanced metric; moreover, we can use \cref{LVMB_SKT} to show that the existence of SKT metrics is obstructed. Indeed, given the expression of its characteristic class in \cref{niceFamilyClass}, we can compute
\begin{equation*}
    \sum_{j=1}^{2m} (w_{2j} - w_{4d+1})^2 = -2 w_{4d+1} \wedge \sum_{j=1}^{2m} w_{2j-1}.
\end{equation*}
As the right-hand side does not vanish in \(H^4((\CP^1)^{2d+1};\Z)\), the existence of SKT metrics is obstructed on \(N_\ell\) with \(\ell\) odd. This observation can be summarized in the form of theorem.
\begin{theorem}
    \label{SKT_either_balanced_family}
    For any \(\ell \geq 2\), the LVMB manifold \(N_\ell \to (\CP^1)^\ell\) constructed as in \cref{niceBalancedFamily} admits a balanced metric if and only if \(\ell\) is odd, while it admits an SKT metric if and only if \(\ell\) is even. In this case it is a product of Calabi--Eckmann threefolds.
\end{theorem}
This behaviour is consistent with the conjecture by A.~Fino and L.~Vezzoni about the coexistence of balanced and SKT metrics on complex manifolds \cite{FinoVezzoni2015}. In fact, it might be worthwile to reformulate \cref{balancedGGP} as a cohomological characterization, as was done for SKT metrics, in order to prove that their conjecture holds true within the class of regular LVM manifolds.

\begin{example}[Calabi--Eckmann manifolds]
    One of the main examples of LVM manifolds are Calabi--Eckmann manifolds \cite{LopezVerjovsky1997}. Thus, we can apply \cref{LVMB_SKT} to show that the Calabi--Eckmann manifolds \(M_{a,b;\tau}\) that support an SKT metric are precisely those with \(0 \leq a,b \leq 1\): i.e. elliptic curves, Hopf surfaces and Calabi--Eckmann threefolds. This fact was first proved in \cite[Thm. 5.16, Example 5.17]{Cavalcanti2020}, see also \cite[Thm. 2.3.1]{BarbaroThesis}.
    
    To compute the characteristic class of the LVMB bundle \(\pi : M_{a,b;\tau} \to \CP^a \times \CP^b\) we consider the Gale dual configuration provided in \cref{example:C-E_manifolds}. Inverting that matrix on the left we obtain:
    \begin{equation*}
        \begin{bmatrix}
            \horzbar f_0^* \horzbar\\
            \horzbar f_1^*\horzbar\\
            \horzbar f_2^* \horzbar
        \end{bmatrix}
        =
        \begin{bNiceMatrix}
            1 & 0& & \Cdots & 0\\
            -1 & 1 & 0 &\Cdots & 0\\
            -1 & 0 &\Cdots & 0 & 1 
        \end{bNiceMatrix}.
    \end{equation*}
    Thus, as a holomorphic torus bundle, the characteristic class of \(M_{a,b;\tau}\) is \(\gamma(\pi) = w_1 \otimes f_1 + w_2 \otimes f_2\), where \(w_1,w_2\) are the generators in degree 2 of
    \begin{equation*}
        H^*(\CP^a \times \CP^b;\Z) \iso \frac{\Z[w_1,w_2]}{(w_1^{a+1}, w_2^{b+1})}.
    \end{equation*}
    Hence, we observe that the quantity
    \begin{equation*}
        \iota^\vee(f_1^*) \wedge \iota^\vee(f_1^*) + \iota^\vee(f_2^*) \wedge \iota^\vee(f_2^*) = w_1^2 + w_2^2 \in H^4(\CP^a \times \CP^b;\Z)
    \end{equation*}
    vanishes if and only if \(0 \leq a,  b \leq 1\). The assertion we claimed thus follows from \cref{LVMB_SKT}.
\end{example}

\printbibliography
\noindent\begin{minipage}{\linewidth}
    \begin{center}
        \small
        \rule{4cm}{.5pt}
        \bigskip
        
        Università degli Studi di Firenze, Dipartimento di Matematica e Informatica ``Ulisse Dini'',\\ V.le Morgagni 67/a, 50134 Firenze, Italia
        
        email: \href{mailto:federico.thiella@unifi.it}{federico.thiella@unifi.it}
    \end{center}
\end{minipage}

\end{document}

%% file: preamble.tex
\usepackage[utf8]{inputenc}
\usepackage[T1]{fontenc}
\usepackage[english]{babel}
\usepackage{amsmath}
\usepackage{amsthm}
\usepackage{amsfonts}
\usepackage{amssymb}
\usepackage{mathtools}
\usepackage{graphicx}
\usepackage{xcolor}
\usepackage{hyperref}
\usepackage[textsize = scriptsize, color=green]{todonotes}
\usepackage{csquotes}
\usepackage{accents}
\usepackage{nicematrix}
\usepackage{tikz}
\usetikzlibrary{cd,arrows.meta,math,positioning}
\usepackage{enumerate}
\usepackage{leftindex}
\usepackage[babel]{microtype}
\newcommand{\fourDegreed}[4]{\leftindex^{#2}_{\phantom{#4}}{#1}^{#3}_{#4}}

\newcommand{\R}{\mathbb{R}}
\newcommand{\C}{\mathbb{C}}
\newcommand{\CP}{\mathbb{C}\mathrm{P}}
\newcommand{\Z}{\mathbb{Z}}
\newcommand{\Q}{\mathbb{Q}}
\newcommand{\U}{\mathrm{U}}
\newcommand{\SU}{\mathrm{SU}}
\newcommand{\PD}{\mathrm{PD}}
\DeclareMathOperator{\Span}{Span}
\newcommand{\del}{\partial}
\newcommand{\delbar}{{\bar\partial}}
\newcommand{\iso}{\cong}
\newcommand{\inv}{{-1}}
\newcommand{\contr}{\lrcorner\,}
\newcommand{\smooth}{C^\infty}
\DeclareMathOperator{\Hom}{Hom}

\DeclareMathOperator{\coker}{coker}
\DeclareMathOperator{\rk}{rk}
\DeclareMathOperator{\id}{id}
\DeclareMathOperator{\Rel}{Rel}

\newcommand{\proj}{\mathbb{P}}
\newcommand{\dfn}{\coloneqq}
\newcommand{\nfd}{\eqqcolon}

\let\Re\relax
\DeclareMathOperator{\Re}{\mathrm{Re}}
\let\Im\relax
\DeclareMathOperator{\Im}{\mathrm{Im}}

\newcommand{\calT}{\mathcal{T}}
\DeclareMathOperator{\cone}{cone}
\newcommand{\realTorus}[1]{T^{#1}}
\newcommand{\cpxTorus}[1]{\mathbb{T}^{#1}}

\DeclareMathOperator{\Ima}{Im}

\newcommand{\trans}[1]{\leftindex^{t}{#1}}
\newcommand{\action}{\mathop{\!\curvearrowright}}
\newcommand{\lattice}{L}
\renewcommand{\th}{-th\ }

\newcommand*{\vertbar}{\rule[-1ex]{0.5pt}{2.5ex}}
\newcommand*{\horzbar}{\hspace{.5ex}\rule[.5ex]{2.5ex}{0.5pt}\hspace{.7ex}}

\newcommand*\circled[1]{\tikz[baseline=(char.base)]{
            \node[shape=circle,draw,inner sep=.5pt] (char) {\tiny#1};}}

\newcommand{\overnumber}[2]{\mathrel{\stackrel{\makebox[0pt]{\mbox{\normalfont \circled{#2}}}}{#1}}}

\let\svthefootnote\thefootnote
\newcommand\freefootnote[1]{%
  \let\thefootnote\relax%
  \footnotetext{#1}%
  \let\thefootnote\svthefootnote%
}

\usepackage{lipsum}
\usepackage{faktor}

\newtheorem{theorem}{Theorem}
\newtheorem{proposition}{Proposition}[section]
\newtheorem{lemma}[proposition]{Lemma}
\newtheorem{corollary}{Corollary}

\theoremstyle{definition}
\newtheorem{definition}{Definition}
\newtheorem{remark}[proposition]{Remark}
\theoremstyle{remark}
\newtheorem*{note}{Note}

\newtheoremstyle{named}{}{}{\itshape}{}{\bfseries}{.}{.5em}{#1 \thmnote{#3}}
\theoremstyle{named}
\newtheorem*{namedtheorem}{Theorem}

\newtheoremstyle{named_example}{}{}{}{}{\bfseries}{.}{.5em}{#1Example \thmnumber{#2}\,: \thmnote{#3}}
\theoremstyle{named_example}
\newtheorem{example}[proposition]{}

\usepackage{zref-clever}
\newcommand{\cref}[1]{\zcref{#1}}
\newcommand{\Cref}[1]{\zcref[S]{#1}}

\newcommand{\labelcref}[1]{\zcref[noname]{#1}}
\newcommand{\namecref}[1]{\zcref[noref]{#1}}

\zcsetup{nameinlink=false,abbrev}

\AddToHook{env/lemma/begin}{%
  \zcsetup{countertype={proposition=lemma}}}
\AddToHook{env/remark/begin}{%
  \zcsetup{countertype={proposition=remark}}}
\AddToHook{env/named_example/begin}{%
  \zcsetup{countertype={proposition=named_example}}}

\zcRefTypeSetup{named_example}{
Name-sg = Example ,
name-sg = example ,
Name-pl = Examples ,
name-pl = examples ,
}

%% file: pictures/Leray1.tex
\begin{tikzpicture} [every node/.append style={align=left}]
  \tikzmath{
    \dx = 3;
    \x1 = 0;
    \dy = 1.5;
    \y1 =0;
    \x2 = \x1 + \dx/4;
    \x3 = \x2 + \dx/2;
    \x4 = \x3 + \dx/2;
    \x5 = \x4 + \dx;
    \x6 = \x5 + \dx;
    \x7 = \x6 + \dx;
    \x8 = \x7 + \dx;
    \y2 = \y1 + \dy/3;
    \y3 = \y2 + \dy/3;
    \y4 = \y3 + \dy;
    \y5 = \y4 + \dy;
    \y6 = \y5 + \dy;
    \y7 = \y6 + \dy;
    \y8 = \y7 + \dy;
  };
\node[] at (\x3,\y1)  {0};
\node[] at (\x4,\y1)  {1};
\node[] at (\x5,\y1)  {2};
\node[] at (\x6,\y1)  {3};
\node[] at (\x7,\y1)  {4};

\node[] at (\x1,\y3) {0};
\node[] at (\x3,\y3)  {\(\Z\)};
\node[] at (\x4, \y3) (E10) {\(0\)};
\node[] at (\x5, \y3) (E20) {\(H^2(X_\Sigma,\Z)\)};
\node[] at (\x6, \y3) {\(0\)};
\node[] at (\x7, \y3) (E40) {\(H^4(X_\Sigma,\Z)\)};

\node[] at (\x1,\y4)  {1};
\node[] at (\x3,\y4) (E01) {\(H^1(\realTorus{2m},\Z)\)};
\node[] at (\x4, \y4) {0};
\node[] at (\x5, \y4) (E21) {\(H^1(\realTorus{2m},\Z) \otimes H^2(X_\Sigma,\Z)\)};
\node[] at (\x6, \y4) {\(0\)};
\node[] at (\x7, \y4) (E41) {\(H^1(\realTorus{2m},\Z) \otimes H^4(X_\Sigma,\Z)\)};

\node[] at (\x1,\y5) {2};
\node[] at (\x3,\y5) (E02) {\(H^2(\realTorus{2m},\Z)\)};
\node[] at (\x4, \y5) {0};
\node[] at (\x5, \y5) (E22) {\(H^2(\realTorus{2m},\Z) \otimes H^2(X_\Sigma,\Z)\)};
\node[] at (\x6, \y5) {\(0\)};
\node[] at (\x7, \y5) {\(H^2(\realTorus{2m},\Z) \otimes H^4(X_\Sigma,\Z)\)};

\draw[->] (E01)--(E20) node[midway, above right, yshift=-2mm] {\small\(d_2^{0,1}\)};
\draw[->] (E02)--(E21) node[midway, above right, yshift=-2mm] {\small\(d_2^{0,2}\)};
\draw[->] (E21)--(E40) node[midway, above right, yshift=-2mm] {\small\(d_2^{2,1}\)};
\draw[->] (E22)--(E41) node[midway, above right, yshift=-2mm] {\small\(d_2^{2,2}\)};

\draw[thick](\x1-\dx/4,\y2)--(\x7+\dx/2, \y2);
\draw[thick](\x2, \y1-\dy/4)--(\x2, \y5+\dy/4);

\node at (\x7+.7*\dx, \y2) {\(p\)};
\node at (\x2, \y5+\dy/2) {\(q\)};
  
\end{tikzpicture}